\documentclass[%
paper=A4,%
pagesize%
]{scrbook}

\usepackage[english]{babel}

\usepackage{amsmath,amssymb,amsthm,amsfonts}
\usepackage{exscale}

\usepackage{bbm}
\usepackage{lmodern}

\usepackage{hyperref}

\usepackage{slashed}
\usepackage{graphicx}

\usepackage[all]{xy}
\CompileMatrices
\usepackage{tensor}

\usepackage{ifthen}
\usepackage{fourier-orns} % \leafNE symbol

\numberwithin{equation}{section}

\newcommand{\tree}[1]{%
\vcenter{\hbox{\includegraphics{t_#1}}}%
}
\newcommand{\wald}[1]{%
  \vcenter{\hbox{${#1}$}}  %
}
\newcommand{\Graph}[2][1.0]{%
\vcenter{\hbox{\includegraphics[scale=#1]{g_#2}}}%
}

\newcommand{\K}{
\mathbb{K}
}
\newcommand{\Q}{
\mathbb{Q}
}
\newcommand{\N}{
\mathbb{N}
}
\newcommand{\Z}{
\mathbb{Z}
}
\newcommand{\R}{
\mathbb{R}
}
\newcommand{\C}{
\mathbb{C}
}
\newcommand{\tp}{
\otimes
}

\newcommand{\coeff}[2][]{c_{#2}^{#1}}
\newcommand{\rp}{\mu} % Subtraction point for momentum scheme
\newlength{\wurelwidth}
\newcommand{\urel}[2][=]{\mathrel{\mathop{#1}\limits_{\!\scalebox{0.5}{#2}\!}}}
\newcommand{\wurel}[2][=]{\mathrel{\mathop{#1}_{\!\scalebox{0.5}{\makebox[\the\wurelwidth]{#2}}\!}}}
\newcommand{\toymor}{\rho}
\newcommand{\Plin}{P_{\mathrm{lin}}}
\newcommand{\toycocycle}{%
B_+^{\!\raisebox{1mm}{\includegraphics[scale=0.6]{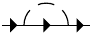}}\!}%
}
\newcommand{\Htoy}{%
H^{}_{\!\includegraphics[scale=0.6]{g_+-.pdf}}%
}
\newcommand{\toyform}{\eta}
\newcommand{\toycc}{L}
\newcommand{\loops}[1]{\abs{#1}}
\newcommand{\internals}[1]{I(#1)}
\newcommand{\externals}[1]{E(#1)}
\newcommand{\nodes}[1]{V(#1)}

\newcommand{\master}[2]{L(#1,#2)}
\newcommand{\conj}[1]{#1^{\ast}}
\newcommand{\reg}{z} % Parameter fuer dimreg und analytische Regularisierung
\newcommand{\toy}[1][\reg]{%
{_{#1}}\phi
}% Feynmanregeln fuer das Toymodel
\newcommand{\toyR}[1][]{%
{^{}_{\reg}}\phi^{}_{R\ifthenelse{\equal{#1}{}}{}{,#1}}
}% Renormierte Feynmanregeln fuer das Toymodel
\newcommand{\dd}[1][]{\mathrm{d}^{#1}} % Integrations-differential
\newcommand{\toyphy}[1][]{{}^{}_0\phi^{}_{#1}}
\newcommand{\toyZ}{Z} % Counterterms fuer das Toymodel
\newcommand{\toylog}{\gamma} %beta-function bzw. anomalous dimension
\newcommand{\treeabc}{\tree{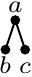}}
\newcommand{\cupdot}{\mathbin{\dot{\cup}}}
\newcommand{\trees}{\mathcal{T}}
\newcommand{\dsecoeff}{\sigma}
\newcommand{\orderedtrees}{\widetilde{\mathcal{T}}}
\newcommand{\orderedforests}{\widetilde{\mathcal{F}}}
\newcommand{\intrules}[1][]{{^{#1}}\varphi}
\newcommand{\forests}{\mathcal{F}}
\newcommand{\convolution}{\star}
\newcommand{\counit}{\varepsilon}
\newcommand{\unit}{u}
\newcommand{\autoconc}{\triangleright} % Produkt auf H_R' induziert von der Konkatenation in Aut(H_R) ueber die universelle Eigenschaft
\newcommand{\mul}{m}
\newcommand{\unimor}[1]{{^{#1}}\!\rho} % Morphismus induziert von der universellen Eigenschaft von H_R
\newcommand{\convalg}[2]{{\Hom \left( #1, #2 \right)}_{\convolution}}
\newcommand{\convend}[1]{{\End (H)}_{\convolution}}
\newcommand{\units}[1]{#1^{\times}} % Einheitengruppe eines Ringes
\newcommand{\cored}{\widetilde{\Delta}} % Koproduct ohne 1 tensor id und id tensor 1
\newcommand{\hide}[1]{}
\newcommand{\leaves}{\text{\leafNE}}
\newcommand{\convgroup}[2]{G^{#1}_{#2}}
\newcommand{\convliealg}[2]{\mathfrak{g}_{#2}^{#1}}
\newcommand{\convkommu}[2]{{\left[ #1, #2 \right]}_{\convolution}} % Kommutator induziert von Konvolution
\newcommand{\chars}[2]{\widetilde{G}_{#2}^{#1}} % Gruppe der Charaktere (Algebrenmorphismen) unter Konvolution
\newcommand{\Rms}{R_{\text{\tiny MS}}} % Minimal Subtraction Scheme
\newcommand{\auto}[1]{{^{#1}}\chi} % Automorphismen von H_R aus Funktionalen mittels universeller Eigenschaft
\newcommand{\cmodcp}{\psi} % Strukturabbildung von Komoduln
\newcommand{\alg}[1][A]{\mathcal{#1}}
\newcommand{\decor}{\mathcal{D}}
\newcommand{\bigo}[1]{\mathcal{O}\left(#1\right)}
\newcommand{\landautheta}[1]{\Theta\left(#1\right)}
\newcommand{\momsch}[1]{R_{#1}}
\newcommand{\restrict}[2]{%
{\left. #1 \right|}_{#2}%
}
\DeclareMathOperator{\Hom}{Hom}

\DeclareMathOperator{\lin}{lin}
\DeclareMathOperator{\Aut}{Aut}
\DeclareMathOperator{\End}{End}
\DeclareMathOperator{\dilog}{dilog}
\DeclareMathOperator{\sdd}{sdd}
\newcommand{\aboveeq}{\preceq}
\newcommand{\beloweq}{\succeq}
\newcommand{\indep}{I}
\newcommand{\Prim}{
\mathrm{Prim}
}
\newcommand{\Grp}{
\mathrm{Grp}
}
\newcommand{\ev}{
\mathrm{ev}
}
\DeclareMathOperator{\im}{im}
\renewcommand{\1}{
\mathbbm{1}
}
\newcommand{\isomorph}{
\cong
}
\newcommand{\defas}{
\mathrel{\mathop:}=
}
\newcommand{\safed}{
=\mathrel{\mathop:}
}
\newcommand{\gdw}{
\ensuremath{\Leftrightarrow}
}

\newcommand{\set}[1]{
\left\{ #1 \right\}
}
\newcommand{\setexp}[2]{
\left\{ #1\!:\ #2 \right\}
}
\newcommand{\abs}[1]{
\left\lvert #1 \right\rvert
}
\newcommand{\todo}{
\text{\huge TODO}
\addcontentsline{toc}{subsection}{TODO}
}
\newcommand{\id}{
\mathrm{id}
}
\newcommand{\qft}{%
quantum field theory%
}
\newcommand{\qfts}{%
quantum field theories%
}
\newtheorem{satz}{Theorem}[section]
\newtheorem{definition}[satz]{Definition}
\newtheorem{lemma}[satz]{Lemma}
\newtheorem{korollar}[satz]{Corollary}
\newtheorem{proposition}[satz]{Proposition}

\title{Hopfalgebraic renormalization of Kreimer's toy model}
\author{Erik Panzer}

\begin{document}

\begin{titlepage}
	\begin{center}

		\vspace*{2cm}
		{\Huge
			Hopf-algebraic renormalization\\[0.8cm] of Kreimer's toy model
		}\\[1.0cm]
		
		Erik Panzer\footnote{\href{mailto:panzer@mathematik.hu-berlin.de}{panzer@mathematik.hu-berlin.de}}\\
		{\today}\\[0.5cm]

		A masters thesis written at the research group\\[0.5cm]
		{\large\textsc{Structure of Local Field Theories}}\\[1cm]
		\includegraphics[scale=3.0]{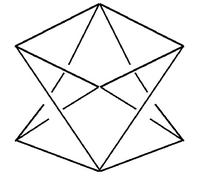}\\[1cm]
		lead by\\[0.5cm]
		{\textit{Prof. Dr. Dirk Kreimer}}\\[0.5cm]
		at the Humboldt-Universit\"{a}t zu Berlin\\
		Rudower Chaussee 25\\
		12489 Berlin\\[0.5cm]

		\includegraphics[width=.75\textwidth]{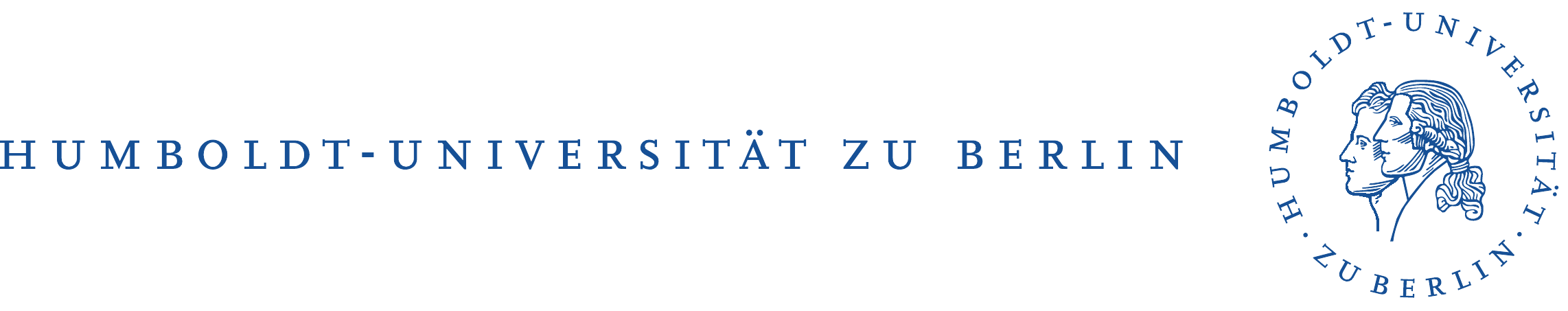}
	\end{center}
\end{titlepage}

\tableofcontents

\chapter{Introduction}

Out of the development of {\qfts} during the last century emerged the incredibly successful standard model of particle physics. The accuracy to which its theoretic predictions match the high precision measurements in particle colliders is amazing and leaves no doubt of the usefulness and importance of \qft.

Therefore it is most astounding that still today, the mathematical framework of this family of theories is far from being well understood. Mathematicians have been working incredibly hard to establish a consistent definition of {\qfts} allowing for the desired physical properties, but have so far only succeeded in this business in space-time dimensions different than four.

%These are clearly reasons for the fascination {\qft} has on mathematicians. In fact, it became a fruitful pool of ideas and concepts leading to new discoveries and methods in pure mathematics.
For this and other reasons {\qft} remains a fascinating subject and continues to pose extremely challenging problems to mathematics.

Recently, the intricate problem of renormalization (a procedure necessary for physical {\qfts}) has been formulated in an illuminating manner by Dirk Kreimer and collaborators -- giving it a precise mathematical definition and prescription. It is the aim of this work to provide a brief introduction into the algebraic structures employed by this mechanism and to learn about its implications and the benefits of its use while studying a particular example -- the \emph{toy model}.

In the following chapter, we develop the necessary algebra to formulate renormalization and perturbative {\qft} using \emph{Hopf algebras}. Key concepts are the \emph{convolution product}, the \emph{algebraic Birkhoff decomposition} and \emph{Hochschild cohomology}.

Chapter \ref{chap:toymodel} is mostly devoted to the investigation of Kreimer's a toy model and traces the path of defining a perturbative {\qft}: Starting with the definition of \emph{Feynman rules} on a combinatoric Hopf algebra, requiring \emph{regularization}, we study \emph{renormalization} using a subtraction scheme in section \ref{sec:renormalization}. The last step is to take the \emph{physical limit} to remedy the regulator introduced earlier.

At this stage we have well-defined renormalized Feynman rules at hand and discuss how to obtain physically meaningful quantities, the \emph{correlation functions}. In this setting we will encounter \emph{Dyson-Schwinger equations} and the \emph{renormalization group}.

Finally, section \ref{sec:toymodel-origin} will exhibit how the just studied toy model is indeed realistic in the sense that it occurs as a subset of certain physical {\qfts}.

\chapter{Hopf algebras}

The fundamental mathematical structure behind perturbative renormalization is the Hopf algebra as discovered in \cite{Kreimer:HopfAlgebraQFT}. In the first section we will merely state the basic definitions and properties of bialgebras and refer to \cite{Manchon,Sweedler} for details and omitted proofs.

Thereafter, we focus on the ingredients of particular relevance to us: the convolution product, the concept of connected filtrations and finally the algebraic Birkhoff decomposition, which effectively describes the recursive process of renormalization.

We will then introduce the Hopf algebra $H_R$ of rooted trees which provides a model for nested and disjoint subdivergences of Feynman graphs (see section \ref{sec:toymodel-origin}). It forms the starting point of the toy model to be discussed in the following chapter.

Finally we define the Hochschild cohomology of bialgebras and apply it to $H_R$ and the Hopf algebra $\K[x]$ of polynomials. We stress the universal property \eqref{eq:H_R-universal} of $H_R$ and obtain a result on how it behaves under coboundaries (proposition \ref{satz:change-coboundary-equals-auto}).

\section{Bialgebras}
We consider algebras as well as co-, bi- and Hopf algebras over a field $\K$, usually thinking of $\Q$ or $\C$ (though for the algebraic properties it suffices that $\K$ enjoys characteristic zero). All vector spaces and tensor products are to be understood over this field, as is the functor $\Hom(\cdot, \cdot)$ (which always denotes just the space of linear maps, no matter if its arguments are algebras or other objects endowed with a more subtle structure).

We further identify any vector space $V$ (canonically) with $V \tp \K$ and correspondingly linear maps $f\in \Hom(V,W)$ with $f \tp \id_{\K}\in\Hom(V \tp \K,W \tp \K)$ without saying so explicitly. For example this happens in \eqref{eq:alg-eins} and \eqref{eq:koalg-koeins}.

The linear span of a subset $M\subseteq V$ of a vector space $V$ will be denoted by $\lin M$.
\begin{definition}\label{def:alg}
	An (associative) \emph{algebra} $(\alg,\mul)$ consists of a vector space $\alg$ and a \emph{product} $\mul \in \Hom(\alg \tp \alg, \alg)$ fulfilling the \emph{associativity}
	\begin{equation}\label{eq:alg-asso}
		\mul \circ (\id \tp \mul) = \mul \circ (\mul \tp \id).
	\end{equation}
	Should there exist a function $\unit\in \Hom(\K, \alg)$ such that
	\begin{equation}\label{eq:alg-eins}
		\mul \circ (\unit \tp \id) = \id = \mul \circ (\id \tp \unit),
	\end{equation}
	we call $\unit$ the \emph{unit map} and $(\alg,\mul,\unit)$ a \emph{unital} algebra. A \emph{morphism} of (unital) algebras $(\alg[A], \mul_{\alg[A]})$ and $(\alg[B], \mul_{\alg[B]})$ is a map $\phi \in \Hom(\alg[A],\alg[B])$ such that $\phi \circ \mul_{\alg[A]} = \mul_{\alg[B]} \circ (\phi \tp \phi)$ and (in the unital case) $\phi \circ \unit_{\alg[A]} = \unit_{\alg[B]}$.
\end{definition}
By $\unit(\lambda) = \lambda \unit(1)$ (for $\lambda \in \K$) the unit map can be identified with the \emph{unit} $\1 \defas \unit(1)$, being the neutral element of the multiplication $\mul \circ \tp\!: \alg \times \alg \rightarrow \alg$ through \eqref{eq:alg-eins}:
\begin{equation*}
	\forall a\in \alg\!:
	\quad a \cdot \1 \defas \mul(a\tp \1) = a = \mul(\1 \tp a) \safed \1 \cdot a.
\end{equation*}
The requirement $\phi\, \circ\, \unit_{\alg[A]} = \unit_{\alg[B]}$ for a morphism of unital algebras is equivalent to \mbox{$\phi(\1_{\alg[A]}) = \1_{\alg[B]}$}. In the following all algebras will be associative and unital unless stated otherwise.
We define the \emph{iterated products} $\mul^n\!:\ \alg^{\tp n+1} \rightarrow \alg$ by
\begin{equation}
	\forall n \in \N_0\!: \quad 
	\mul^{n+1} \defas 
			\mul \circ \left( \mul^n \tp \id \right) 
			\urel{\eqref{eq:alg-asso}} \mul \circ \left( \id \tp \mul^n \right),
	\label{eq:mul-iterated}
\end{equation}
which are independent of the order of multiplications (arbitrary placement of brackets).

The properties \eqref{eq:alg-asso} and \eqref{eq:alg-eins} are equivalent to the commutativity of the diagrams
\begin{equation}\label{cd:algebra}
	\vcenter{\xymatrix@C+0.5cm{
		{A \tp A \tp A} \ar[r]^-{m \tp \id} \ar[d]_{\id \tp m} & {A \tp A} \ar[d]^{m} \\
		{A \tp A} \ar[r]^-{m} & {A}
	}}
	\quad \text{and} \quad
	\vcenter{\xymatrix@C+0.5cm{
		{\K \tp A} \ar[r]^{u \tp \id} \ar[dr]_{\isomorph} & {A \tp A} \ar[d]^{m} & {A \tp \K} \ar[l]_{\id \tp u} \ar[dl]^{\isomorph} \\
		& {A}
	}},
\end{equation}
which readily suggest the definition of the dual object by reversal of arrows:

\begin{definition}\label{def:koalg}
	A (coassociative) \emph{coalgebra} $(C,\Delta)$ consists of a vector space $C$ and a \emph{coproduct} $\Delta \in \Hom(C,C\tp C)$ fulfilling the \emph{coassociativity} property
	\begin{equation}\label{eq:koalg-koasso}
		(\id \tp \Delta) \circ \Delta = (\Delta \tp \id) \circ \Delta.
	\end{equation}
	Should there exist a functional $\counit \in \Hom(C,\K) = C'$ such that
	\begin{equation}\label{eq:koalg-koeins}
		(\counit \tp \id) \circ \Delta = \id = (\id \tp \counit) \circ \Delta,
	\end{equation}
	we call $\counit$ the \emph{counit} and $(C,\Delta,\counit)$ a \emph{counital} coalgebra. A \emph{morphism} of (counital) coalgebras $(C, \Delta_C)$ and $(D, \Delta_D)$ is a map $\phi \in \Hom(C,D)$ such that $\Delta_D \circ \phi = (\phi \tp \phi) \circ \Delta_C$ and (in the counital case) also $\counit_D \circ \phi = \counit_C$ hold.
\end{definition}
As announced, \eqref{eq:koalg-koasso} and \eqref{eq:koalg-koeins} are nothing but the commutativity of
\begin{equation}\label{cd:coalgebra}
	\vcenter{\xymatrix@C+0.5cm{
		C \ar[r]^{\Delta} \ar[d]_-{\Delta} & {C \tp C} \ar[d]^{\id \tp \Delta} \\
		{C \tp C} \ar[r]^-{\Delta \tp \id} & {C \tp C \tp C}
		}}
	\quad \text{and} \quad
	\vcenter{\xymatrix@C+0.5cm{
		{\K \tp C} & {C \tp C} \ar[l]_{\counit \tp \id} \ar[r]^{\id \tp \counit} & {C \tp \K} \\
		& C \ar[u]_{\Delta} \ar[ul]^{\isomorph} \ar[ur]_{\isomorph}
		}},
\end{equation}
dual to \eqref{cd:algebra}. As in the case of algebras, the counit is unique if existent by
\begin{align*}
	\counit 
	&= \counit \circ \id
	 = \counit \circ (\id \tp \counit') \circ \Delta
	 = (\counit \tp \counit') \circ \Delta
	 = \counit' \circ (\counit \tp \id) \circ \Delta
	 = \counit' \circ \id
	 = \counit'
\end{align*}
for any two counits $\counit$ and $\counit'$ as a consequence of \eqref{eq:koalg-koeins}. We remark:
\begin{enumerate}
	\item By \eqref{eq:koalg-koeins}, the coproduct of counital coalgebras is injective\footnote{Note the duality to the surjectivity of the product for unital algebras!}.
	\item An element $g \in C\setminus\set{0}$ of a coalgebra is called \emph{grouplike} iff $\Delta (g) = g \tp g$. The set $\Grp(C)$ of grouplike elements is linearly independent and spans a subcoalgebra\footnote{%
A \emph{subcoalgebra} is a subspace $V \subseteq C$ such that $\Delta (V) \subseteq V \tp V$.%
} (see \cite{Sweedler}).
	\item Any grouplike $g \in \Grp(C)$ fulfils $\counit(g) = 1$ by $g = \counit(g) g$ and $g \neq 0$, using \eqref{eq:koalg-koeins}.
	\item Through \eqref{eq:koalg-koasso}, the \emph{iterated coproducts} $\Delta^n\!:\ C \rightarrow C^{\tp (n+1)}$ defined recursively by
		\begin{equation}
			\Delta^0 \defas \id
			\quad \text{and} \quad
			\Delta^{n+1} \defas \left( \Delta \tp {\id}^{\tp n} \right) \circ \Delta^n
			\quad\text{for any $n\in\N_0$,}
			\label{eq:iterated-coproduct}
		\end{equation}
		do not depend on the order in which the coproducts are applied. Hence we have
		\begin{equation*}
			\forall n\in\N_0\!: \quad \forall 0 \leq k \leq n\!: \quad
			\Delta^{n+1} = \left( {\id}^{\tp k} \tp \Delta \tp {\id}^{\tp (n-k)} \right) \circ \Delta^n.
		\end{equation*}
	\item Often we will denote $\Delta (x)$ by the \emph{Sweedler notation} $\sum_x x_1 \tp x_2$, a shorthand for a representation $\Delta (x) = \sum_i x_1^{(i)} \tp x_2^{(i)}$.
\end{enumerate}
Naturally we can define algebra and coalgebra structures on tensor products in
\begin{definition}
	Let $(\alg[A],\mul_{\alg[A]}, \unit_{\alg[A]})$ and $(\alg[B], \mul_{\alg[B]}, \unit_{\alg[B]})$ be (unital) algebras, then $\alg[A] \tp \alg[B]$ is a (unital) algebra with multiplication and unit defined by
	\begin{equation}\label{eq:alg-tensor}
		\mul_{\alg[A] \tp \alg[B]} \defas (\mul_{\alg[A]} \tp \mul_{\alg[B]}) \circ \tau_{(2,3)} \quad \text{and}\quad \unit_{\alg[A] \tp \alg[B]} \defas \unit_{\alg[A]} \tp \unit_{\alg[B]}.
	\end{equation}
	Analogously, for (counital) coalgebras $(C, \Delta_C, \counit_C)$ and $(D, \Delta_D, \counit_D)$ the product $C \tp D$ becomes a (counital) coalgebra via
	\begin{equation}\label{eq:koalg-tensor}
		\Delta_{C \tp D} \defas \tau_{(2,3)} \circ (\Delta_C \tp \Delta_D) \quad \text{and}\quad \counit_{C \tp D} \defas \counit_C \tp \counit_D.
	\end{equation}
	Here we introduced for any permutation $\sigma \in S_n$ and vector space $V$ the induced map
	\begin{equation}\label{eq:permu-auto}
		\tau_{\sigma} \in \Aut\left( V^{\tp n} \right), \quad v_1 \tp \ldots \tp v_n \mapsto v_{\sigma_1} \tp \ldots \tp v_{\sigma_n}.
%		\tilde{\sigma}: V^{\times n} \rightarrow V^{\times n}, (v_1,\ldots,v_{n}) \mapsto (v_{\sigma_1},\ldots,v_{\sigma_n}).
	\end{equation}
\end{definition}
Beware that $\Delta\!:\ C \rightarrow C \tp C$ and $m\!:\ \alg \tp \alg \rightarrow \alg$ are in general \emph{not} morphisms of coalgebras and algebras! This is only guaranteed in the case of cocommutative $C$ and commutative $\alg$, respectively.

These are the structures occuring in points 1. and 2. of the following
\begin{definition}\label{def:bialg}
	A vector space $H$ which is both an algebra $(H,\mul)$ as well as a coalgebra $(H,\Delta)$ is called \emph{Bialgebra} $(H, \mul, \Delta)$ iff any of the equivalent\footnote{%
See proposition 3.1.1 in \cite{Sweedler}. Note that $\counit(\1) = 1$ does not need to be requested separately, as by $\Delta (\1) = \1 \tp \1$ we have either $\1 \in \Grp(H)$ (resulting in $\counit(\1) = 1$) or otherwise $\1=0$ implying $H = \set{0}$, which we exclude.}
	conditions hold:
	\begin{enumerate}
		\item $\mul$ is a morphism of coalgebras: $\Delta \circ \mul = (\mul \tp \mul) \circ \Delta_{H \tp H} = (\mul \tp \mul) \circ \tau_{(2,3)} \circ (\Delta \tp \Delta)$
		\item $\Delta$ is a morphism of algebras: $\Delta \circ \mul = \mul_{H \tp H} \circ (\Delta \tp \Delta) = (\mul \tp \mul) \circ \tau_{(2,3)} \circ (\Delta \tp \Delta)$
		\item The following diagram commutes:
			\begin{equation}\label{cd:bialg}
				\vcenter{\xymatrix{
					{H \tp H \tp H \tp H} \ar[rr]^-{\tau_{(2,3)}} & & {H \tp H \tp H \tp H} \ar[d]^-{m \tp m} \\
					{H \tp H} \ar[u]^-{\Delta \tp \Delta} \ar[r]_-{m} & {H} \ar[r]_-{\Delta} & {H \tp H}
				}}
			\end{equation}
	\end{enumerate}
	If $H$ is unital and counital we additionaly request for both of
	\begin{enumerate}
		\item $\unit$ is a morphism of coalgebras, that is $\Delta \circ u = u \tp u$ or equivalently $\Delta \1 = \1 \tp \1$.
		\item $\counit$ is a morphism of algebras, that is $\counit \tp \counit = m_{\K} \circ (\counit \tp \counit) = \counit \circ m$.
	\end{enumerate}
	These are equivalent to the commutativity of the diagrams
	\begin{equation*}
		\vcenter{\xymatrix{
					{\K} \ar[d]_{\isomorph} \ar[r]^{u} & {H} \ar[d]^{\Delta} \\
					{\K \tp \K} \ar[r]_{u \tp u} & {H \tp H}
		}}
		\quad \text{and} \quad
		\vcenter{\xymatrix{
					{H \tp H} \ar[r]^-{m} \ar[d]_{\counit \tp \counit} & {H} \ar[d]^{\counit} \\
					{\K \tp \K} \ar[r]_-{\isomorph} & {\K}
		}},
	\end{equation*}
	expressing that $\Delta$ and $\mul$ are to be morphisms of unital algebras and counital coalgebras, respectively.
\end{definition}
We will always assume bialgebras $H \neq \set{0}$ to be unital and counital. Then note $\1 \in \Grp(H)$ and $\counit(\1) = 1$, so $H$ decomposes naturally into
\begin{equation}\label{eq:augmentation}
	H = \K \cdot \1 \oplus \ker \counit = \im \unit \oplus \ker \counit.
\end{equation}
We denote the projection induced by \eqref{eq:augmentation} as $P \defas \id - \unit \circ \counit\!:\ H \twoheadrightarrow \ker \counit$ and call $\ker \counit$ the \emph{augmentation ideal}. It is an ideal of algebras and at the same time a coideal of coalgebras, saying $H \cdot \ker \counit + \ker \counit \cdot H \subseteq \ker \counit$ and $\Delta (\ker \counit) \subseteq \ker \counit \tp H + H \tp \ker \counit$.
\begin{definition}
	On a bialgebra $H$ we define the \emph{reduced coproduct} $\cored$ to be
	\begin{equation}
		\cored \defas \Delta - \1 \tp \id - \id \tp \1\!: \quad H \rightarrow H \tp H
		\label{eq:cored}
	\end{equation}
	and the space $\Prim(H)$ of \emph{primitive elements} by
	\begin{equation}
		\Prim(H) 
		\defas \ker \cored 
		= \setexp{p \in H}{\Delta (p) = \1 \tp p + p \tp \1}.
		\label{eq:prim}
	\end{equation}
\end{definition}
Note that $\Prim(H)$ is a Lie algebra with the Lie bracket induced by the commutator of the associative algebra $H$! Similarly, the product of two grouplike elements is again grouplike such that $\lin \Grp(H)$ is a subbialgebra\footnote{A subbialgebra is a subspace $V\subseteq H$ that is a subcoalgebra and a unital subalgebra.}.

The reduced coproduct is itself coassociative and therefore allows for well defined iterated reduced coproducts
\begin{equation*}
	\cored^0 \defas \id
	\quad \text{and} \quad
	\cored^{n+1} \defas \left( {\id}^{\tp k} \tp \cored \tp {\id}^{\tp (n-k)} \right) \circ \cored^{n}
	\quad \text{for any $n\in\N$,}
\end{equation*}
where the choice of $0 \leq k \leq n$ does not matter. Note that on $\ker\counit$, $\cored^n = P^{\tp n+1} \circ \Delta^n$, so in particular $\cored$ maps $\ker \counit$ into $\ker \counit \tp \ker \counit$ and turns $(\ker \counit, \cored)$ into a coalgebra on its own. In Sweedler's notation we indicate the reduced coproduct by $\cored (x) = \sum_x x' \tp x''$.

\subsection{The convolution product}
\begin{definition}\label{def:convolution}
	Let $(C,\Delta)$ be a coalgebra and $(A,m)$ an algebra, then define the \emph{convolution product} $\convolution$ on $\Hom(C,A)$ by
	\begin{align}
		\convolution &\in \Hom\left( \Hom(C, A) \tp \Hom(C,A), \Hom(C, A) \right) \nonumber\\
		\convolution &\defas \Hom(\Delta, m) \circ \iota ,\ f \tp g \mapsto m \circ (f \tp g) \circ \Delta \label{eq:convolution}
	\end{align}
	using the canonical embedding $\iota\!:\ \Hom(C, A) \tp \Hom(C,A) \hookrightarrow \Hom(C \tp C, A \tp A)$. As usual we also use $\convolution$ to denote the multiplication map
	\begin{equation*}
		\convolution \circ \tp\!:\ \Hom(C, A) \times \Hom(C,A) \rightarrow \Hom(C, A).
	\end{equation*}
\end{definition}

\begin{lemma}
	$\convalg{C}{A} \defas \left( \Hom_{\K}(C, A), \convolution \right)$ is an associative algebra. If $C$ is counital with counit $\counit$ and $A$ unital with unit $\unit$, then
	$\convalg{C}{A}$ is unital with unit $e\defas u \circ \counit$.
\end{lemma}
\begin{proof}
	For arbitrary $f,g,h \in\convalg{C}{A}$ observe
	\begin{align*}
		f \convolution ( g \convolution h ) 
		&= m \circ \left[ f \tp ( m \circ g \tp h \circ \Delta) \right] \circ \Delta 
		 = m \circ (\id \tp m) \circ (f \tp g \tp h) \circ (\id \tp \Delta) \circ \Delta \\
		&= m \circ ( m\tp \id) \circ (f \tp g \tp h) \circ (\Delta \tp \id) \circ \Delta
		 = m \circ \left[ (m\circ f \tp g \circ \Delta) \tp h \right] \circ \Delta \\
		&= ( f \convolution g ) \convolution h,
	\end{align*}
	while the neutrality of $e$ follows by
	\begin{align*}
		e \convolution f 
		&= m \circ \left[ ( u \circ \counit) \tp f \right] \circ \Delta
		 = m \circ (u \tp \id) \circ (\id \tp f) \circ (\counit \tp \id) \circ \Delta \\
		&= m \circ (u \tp \id) \circ (\id \tp f) \circ (1_{\K} \tp \id) 
		 = m \circ (\1 \tp f) = f = \ldots = f \convolution e. \qedhere
	\end{align*}
\end{proof}
Note that inverses in $\convalg{C}{A}$ (denoted by $\phi^{\convolution -1}$) are uniquely determined (if existent). Given a bialgebra $H$ and introducing the group of units
\begin{equation}
		\units{\convend{H}} \defas \setexp{\phi \in \End(H)}{\exists {\psi}\in\End(H):\ \phi \convolution {\psi} = {\psi} \convolution \phi = e = \counit \circ \unit},
\end{equation}
of the algebra $\convend{H} \defas \convalg{H}{H}$, considering the canonical element $\id\in\convend{H}$ leads to

\begin{definition}\label{def:hopf}
	A bialgebra $H$ is called \emph{Hopf algebra} iff 	$\id \in \units{\convend{H}}$. This unique inverse $S\defas \id^{\convolution -1}$ of a Hopf algebra is called \emph{antipode}.
\end{definition}
The antipode of a Hopf algebra enjoys a rich list of properties, a few of which being mentioned here (details of the proofs may be found in \cite{Manchon}):
\begin{enumerate}
	\item $S \circ u = u$ and $\counit \circ S = \counit$, this says $S(\1) = \1$ and implies $S(\ker \counit) \subseteq \ker \counit$.
	\item $S$ is an antimorphism of algebras and an antimorphism of coalgebras, explicitly $S \circ m = m \circ \tau \circ (S \tp S)$ and $\Delta \circ S = \tau \circ (S \tp S) \circ \Delta$ with $\tau \defas \tau_{(1,2)}$ from \eqref{eq:permu-auto}.
	\item If $H$ is commutative or cocommutative, then $S^2 = \id$.
	\item $\Prim(H) \subseteq \ker (S + \id)$, hence $S(p) = -p$ for any $p \in \Prim(H)$.
	\item For any grouplike $g \in \Grp(H)$, note $g\cdot S(g) = S(g) \cdot g = \1 = e(g)$. Hence $S$ multiplicatively inverts the grouplike elements. In particular a bialgebra can admit an antipode only if $\Grp(H) \subseteq \units{H} \defas \setexp{x\in H}{\exists y \in H\!:\ y \cdot x = x \cdot y = \1}$.
\end{enumerate}

\subsection{Filtrations, graduations and connectedness}
Along with their combinatoric nature, the Hopf algebras considered here allow for inductive proofs and constructions in various places. As always, those inductions need two ingredients to work:
\begin{itemize}
	\item A start of the induction; it will be trivial in the case of connected Hopf algebras (see definition \ref{def:connected}).
	\item A guarantee that each element (of the Hopf algebra) is reached after a finite number of induction steps; this is assured by a filtration (or a graduation).
\end{itemize}
\begin{definition}\label{def:filtration}
	A family $(H^{n})_{n\in\N_0}$ of growing subspaces $H^n\subseteq H^{n+1}\ \forall n\in\N_0$ of a Hopf algebra $(H,m,u,\Delta,\counit,S)$ is called a \emph{filtration} iff all of the conditions
	\begin{enumerate}
		\item $H = \sum_{n\in\N_0} H^n$
		\item $\forall n\in\N_0\!:\ \Delta(H^n) \subseteq \sum_{i+j=n} H^i \tp H^j = \sum_{i=0}^n H^i \tp H^{n-i}$
		\item $\forall n,m\in\N_0\!:\ H^n \cdot H^m \defas m\left( H^n \tp H^m \right) \subseteq H^{n+m}$
		\item $\forall n\in \N_0\!:\ S\left( H^n \right) \subseteq H^n$
	\end{enumerate}
	hold. Omitting condition $4$ still yields a \emph{filtration of a bialgebra}, whereas providing only properties $\set{1, 2}$ and $\set{1, 3}$ defines filtrations of coalgebras and algebras, respectively. 
\end{definition}
Considering such a filtration, some remarks are in order:
\begin{enumerate}
	\item $H^0$ is a subalgebra / subcoalgebra / subbialgebra / Hopf subalgebra -- whatever is $H$ (immediate from the definition). 
	\item All grouplike elements are necessarily contained in $H^0$: $\Grp(H)\subseteq H^0$. For a proof suppose $g \in \Grp(H) \cap H^n \setminus H^{n-1}$ for $n\in\N$, write $H^{n} = H^{n-1} \oplus \K \cdot g \oplus V$ for some complement $V$ and consider $g\tp g = \Delta (g) \in H^n \tp H^{n-1} + H^{n-1} \tp H^n$.
	\item In general a Hopf algebra does not necessarily admit a \emph{non-trivial}\footnote{%
The only trivial filtration is given by $H^n = H$ for all $n \in \N_0$.}
				filtration! Though there are Hopf subalgebras like $\K \cdot \1$ or more generally $\lin \Grp(H)$ at hand, these do not necessarily provide the start $H^0$ of a filtration due to proposition \ref{satz:filtration}.
	\item Intuitively, a filtration reduces every element in a finite number of steps to $H^0$ using the coproduct. Hence we will have to start inductions on $H^0$.
	\item In any coalgebra $H$, there is an associative product on the set of its vector subspaces, namely the \emph{wedge product}. For subspaces $V,W\subseteq H$ it is defined as
		\begin{equation*}%\label{eq:wedge}
			V \wedge W \defas \Delta^{-1} \left( V \tp H + H \tp W \right).
		\end{equation*}
		By definition \ref{def:filtration} it follows that given any filtration of a coalgebra $H$, the spaces
		\begin{align*}
			\widetilde{H}^n	
			&\defas {\left( H^0\right)}^{\wedge (n+1)} 
%			 \defas {\wedge}^{(n+1)} H^0 
			 = \underbrace{H^0 \wedge \ldots \wedge H^0}_{\text{$(n+1)$ times $H^0$}}
			 = {\left( {\Delta}^{n} \right) }^{-1} \left( \sum_{i=0}^n H^{\tp i} \tp H^0 \tp H^{\tp (n-i)} \right)
		\end{align*}
		fulfil $H^n \subseteq \widetilde{H}^n$ and define a filtration on their own (see \cite{Manchon}). In particular, $H=\sum_{n\in \N_0} {\widetilde{H}}^n$ is thus the \emph{largest} filtration of $H$ that begins with ${\widetilde{H}}^0 = H^0$.
\end{enumerate}
The last remark generalizes to bi- and also Hopf algebras (details in \cite{Manchon}), resulting in
\begin{proposition}\label{satz:filtration}
	Let $H$ be a co-/bi-/Hopf algebra and $L$ a sub(co/bi/Hopf)algebra, then there exists a filtration of $H$ starting with $H^0=L$ iff
	\begin{equation*}
		H=\sum_{n \in\N_0} L^{\wedge (n+1)}
		\quad\gdw\quad
		\forall x\in H\!:\ \exists n\in\N_0\!:\ {\Delta}^n (x) \in \sum_{i=0}^n H^{\tp i} \tp L \tp H^{\tp (n-i)}.
	\end{equation*}
\end{proposition}
As mentioned already, our inductions are going to exploit a filtration and need to start on $H^0$. This motivates the
\begin{definition}\label{def:connected}
	A bialgebra $H$ is \emph{connected} iff there exists a filtration $H=\sum_{n\in\N_0} H^n$ with $H^0 = \K \cdot \1$. By theorem \ref{satz:filtration} this is equivalent to
	\begin{equation}\label{eq:connected}
		H = \sum_{n\in\N} {\left( \K \cdot \1 \right) }^{\wedge (n+1)} 
		= \K \cdot \1 \oplus \sum_{n\in \N} \ker\left( {\cored}^n \right).
	\end{equation}
	If $H$ is connected, $H^n \defas {(\K \cdot \1)}^{\wedge n+1}$ is called the \emph{coradical filtration}.
\end{definition}
\hide{
\begin{enumerate}
	\item Given a filtration, for any $x \in H$ define $\abs{x} \defas \min \setexp{n\in N_0}{x \in H^n}$ \todo to be its \emph{degree} with respect to this filtration.
	\item Taking the sets $U_n \defas H \setminus H^n$ as basis for the neighbourhoods of $0$, one defines a topology on $H$ induced by the filtration, turning $H$ into a topological vector space. By \ref{def-filtration} all operations of the Hopf algebra are continuous with respect to this topology, and it is Hausdorff.
	\item The \emph{degree} induces an ultrametric $d(x,y) \defas 2^{-\abs{x-y}}$ ($2^{-\infty}\defas 0$) and its topology is the one just described.
\end{enumerate}
}
Now we can prove\footnote{Note how in \eqref{eq:conv-inv-recursive} we obtain an inductive proof as a special case of the Birkhoff decomposition.} the existence of a huge subgroup of the convolution algebra in
\begin{satz}\label{satz:convolution-group}
	Let $H$ be a connected bialgebra and $\alg$ an algebra. Then the subset
	\begin{equation}
		\convgroup{H}{\alg} \defas \setexp{\phi \in \Hom(H,\alg)}{\phi(\1_H) = \1_{\alg}} \subseteq \convalg{H}{\alg}
		\label{eq:convolution-group}
	\end{equation}
	of linear maps $\phi\!: H \rightarrow \alg$ with $\phi(\1_H)=\1_{\alg}$ forms a group under the convolution product.
\end{satz}

\begin{proof}
	$\convgroup{H}{\alg}$ is clearly closed under convolution, as for any grouplike $g$ we have
	\begin{equation*}
		\phi \convolution \psi (g)
		= \phi(g) \cdot \psi(g).
	\end{equation*}
	Hence it only remains to show the existence of an inverse $\phi^{\convolution -1}$ for any given $\phi \in \convgroup{H}{\alg}$. By \eqref{eq:connected} and $(\phi - e)(\1) = 0$, for any fixed $x\in H$ we find some $N_x\in\N$ such that $(\phi -e)^{\tp n}\,\circ\,\Delta^{n-1}(x) = 0$ for all $n \geq N_x$. Hence the formal von Neumann series
	\begin{equation}\label{eq:zsh-conv-inverse}
		\phi^{\convolution -1} 
		= {\left[ e - (e - \phi) \right] }^{\convolution -1} 
		\defas \sum_{n \in \N_0} {(e-\phi)}^{\convolution n}
	\end{equation}
	is locally a finite sum and therefore well defines an element of $\convgroup{H}{\alg}$! So the series \eqref{eq:zsh-conv-inverse} converges pointwise in the discrete topology on $\alg$ (eventually it becomes constant), hence as the coproduct $\Delta (x) = \sum_{i=1}^k x_1^{(i)} \tp x_2^{(i)}$ is a \emph{finite} linear combination we find $N\in \N$ with
	\begin{equation*}
		\forall n \geq N\!: \quad
		\forall y \in \set{x, x_2^{(1)}, \ldots, x_2^{(k)}}\!: \quad
		{(e - \phi)}^{\convolution n} (y) = 0.
	\end{equation*}
	This allows us to work with well defined finite sums and to check
	\begin{align*}
		\left[ \phi \convolution \phi^{\convolution -1} \right] (x)
		&= \sum_{i=1}^k \phi\left(x_1^{(i)}\right) \phi^{\convolution -1} \left(x_2^{(i)}\right)
		= \sum_{i=1}^k \phi\left(x_1^{(i)}\right) \sum_{n=0}^N {(e-\phi)}^{\convolution n} \left(x_2^{(i)}\right) \\
		&= \sum_{n=0}^{N} \left[ \phi \convolution {(e-\phi)}^{\convolution n} \right] (x) 
		= \left\{ \sum_{n=0}^{N} (e-\phi)^{\convolution n} - (e - \phi) \convolution \sum_{n=0}^{N} (e-\phi)^{\convolution n} \right\} (x)\\
		&= \left[ {(e-\phi)}^{\convolution 0} - \right.\underbrace{\left.{(e-\phi)}^{\convolution N+1} \right] (x)}_0
		= e (x),
	\end{align*}
	proving $\phi \convolution \phi^{\convolution -1} = e$ pointwise. Clearly $\phi^{\convolution -1} \convolution \phi = e$ follows analogously.
\end{proof}

\begin{korollar}\label{satz:zsh-bialg-hopf}
	Any connected bialgebra $H$ is a Hopf algebra by $\id\in\convgroup{H}{H}$.
\end{korollar}
After these general statements, we want to investigate how the convolution algebra restricts to multiplicative maps like the Feynman rules we will encounter in the next chapter.
\begin{definition}\label{def:character}
	Given a bialgebra $H$ and an algebra $\alg$ we define the set of \emph{characters}
	\begin{equation}
		\chars{H}{\alg} \defas \setexp{\phi \in \convgroup{H}{\alg}}{\phi \circ \mul_H = \mul_{\alg} \circ (\phi \tp \phi)}
		\label{eq:characters}
	\end{equation}
	to consist of the morphisms $\phi\!:\ H \rightarrow \alg$ of unital algebras.
\end{definition}

\begin{lemma}
	If $H$ is a Hopf algebra and $\alg$ a commutative algebra, then $\chars{H}{\alg}$ is a group under convolution. Explicitly, with the antipode $S$ of $H$ we have
	\begin{equation}
		\forall \phi \in \chars{H}{\alg}\!:\quad \phi^{\convolution -1} = \phi \circ S.
		\label{eq:character-group}
	\end{equation}
	\label{satz:character-group}
\end{lemma}

\begin{proof}
	Let $\phi \in \chars{H}{\alg}$, then observe
	\begin{align*}
		(\phi \circ S) \convolution \phi
		&= \mul_{\alg} \circ \left[ (\phi \circ S) \tp \phi \right] \circ \Delta
		= \mul_{\alg} \circ (\phi \tp \phi) \circ (S \tp \id) \circ \Delta \\
		&= \phi \circ \mul_H (S \tp \id) \circ \Delta 
		= \phi \circ (S \convolution \id)
		= \phi \circ \unit \circ \counit
		= \unit_{\alg} \circ \counit
		= e
	\end{align*}
	and analogously $\phi \convolution (\phi\circ S) = e$ such that indeed $\phi$ is invertible in $\convalg{H}{\alg}$ and it fulfils \eqref{eq:character-group}. Moreover, as $S$ is an antimorphism we find $\phi^{\convolution -1} \in \chars{H}{\alg}$ by
	\begin{align*}
		\phi^{\convolution -1} \circ \mul
		&= \phi \circ S \circ \mul
		= \phi \circ \mul \circ \tau \circ (S \tp S)
		= \mul_{\alg} \circ (\phi \tp \phi) \circ \tau \circ (S \tp S) \\
		&= \mul_{\alg} \circ \tau \circ \left[ (\phi \circ S) \tp (\phi \circ S) \right]
		= \mul_{\alg} \circ \left(\phi^{\convolution -1} \tp \phi^{\convolution -1} \right),
	\end{align*}
	exploiting the commutativity of $\alg$ and $\phi^{\convolution -1}(\1) = \phi \circ S(\1) = \phi (\1) = \1_{\alg}$. Given any two $\phi, \psi \in \chars{H}{\alg}$ we observe
	\begin{align*}
		(\phi \convolution \psi) \circ \mul
		&= \mul_{\alg} \circ (\phi \tp \psi) \circ \Delta \circ \mul
		= \mul_{\alg} \circ (\phi \tp \psi) \circ (\mul \tp \mul) \circ \tau_{(2,3)} \circ ( \Delta \tp \Delta ) \\
		&= \mul_{\alg} \circ \left[ (\phi \circ \mul) \tp (\psi \circ \mul) \right] \circ \tau_{(2,3)} \circ (\Delta \tp \Delta) \\
		&= \mul_{\alg} \circ (\mul_{\alg} \tp \mul_{\alg}) \circ (\phi \tp \phi \tp \psi \tp \psi) \circ \tau_{(2,3)} \circ (\Delta \tp \Delta) \\
		&= \mul_{\alg} \circ (\mul_{\alg} \tp \mul_{\alg}) \circ \tau_{(2,3)} \circ (\phi \tp \psi \tp \phi \tp \psi) \circ (\Delta \tp \Delta) \\
		&= \mul_{\alg} \circ \left\{ \left[ \mul_{\alg} \circ (\phi \tp \psi) \circ \Delta \right] \tp \left[ \mul_{\alg} \circ (\phi \tp \psi) \circ \Delta \right] \right\} \\
		&= \mul_{\alg} \circ \left[ (\phi \convolution \psi) \tp (\phi \convolution \psi) \right],
	\end{align*}
	again making use of $\alg$'s commutativity. Together with $(\phi \convolution \psi)(\1) = \phi(\1) \psi (\1) = \1_{\alg}$ this shows $\phi \convolution \psi \in \chars{H}{\alg}$ and finishes the proof.
\end{proof}

\subsubsection{Graduations}

By \eqref{eq:H_R-graduation}, the Hopf algebra $H_R$ of rooted trees we will introduce in section \ref{sec:H_R} comes along with a \emph{graduation} as described in
\begin{definition}
	A \emph{graduation} of a Hopf algebra $H$ is a decomposition $H = \bigoplus_{n\in\N_0} H_n$ such that the following conditions hold for any $n,m\in\N_0\!:$
	\begin{enumerate}
		\item $\Delta(H_n) \subseteq \bigoplus_{i+j=n} H_i \tp H_j = \bigoplus_{i=0}^n H_i \tp H_{n-i}$
		\item $H_n \cdot H_m \defas m\left( H_n \tp H_m \right) \subseteq H_{n+m}$
		\item $S\left( H_n \right) \subseteq H_n$
	\end{enumerate}
\end{definition}
Apparently, a graduation is a structure more subtle than a filtration! In fact, any graduation $H=\bigoplus_{n\in\N_0} H_n$ induces a filtration by $H^n \defas \bigoplus_{k=0}^n H_k$. Thus the results derived for connected bialgebras in this section will in particular apply to $H_R$.

\subsubsection{The Lie group of convolution}
By defining the Lie algebra (which is in fact an ideal in the convolution algebra)
\begin{equation}
	\convliealg{H}{\alg} \defas \setexp{\phi \in \Hom(H,\alg)}{\phi(\1) = 0}
	\label{eq:def-conv-liealg}
\end{equation}
with the lie bracket $\convkommu{v}{w} = v\convolution w - w\convolution v$, the at first only formal definitions
\begin{align}
	\exp_{\convolution}\!:\ &\ \convliealg{H}{\alg} \rightarrow \convgroup{H}{\alg}, \quad \phi \mapsto \sum_{n\in\N_0} \frac{{\phi}^{\convolution n}}{n!}
		\label{eq:exp-conv}\\
	\log_{\convolution}\!:\ &\ \convgroup{H}{\alg} \rightarrow \convliealg{H}{\alg}, \quad \phi \mapsto \sum_{n \in \N} \frac{(-1)^{n+1}}{n} {(\phi - e)}^{\convolution n}
		\label{eq:log-conv}
\end{align}
become locally (that is pointwise at each $x\in H$) finite sums if $H$ is connected, just as in the proof of theorem \ref{satz:convolution-group}. After realizing this well-definedness, it is an easy exercise\footnote{Simply expand the series \eqref{eq:log-conv}, \eqref{eq:exp-conv} and use the relations among their coefficients known from the real analogues $\exp$ and $\ln$.} to check that they deliver bijections between $\convgroup{H}{\alg}$ and $\convliealg{H}{\alg}$ through $\log_{\convolution} \circ \exp_{\convolution} = \restrict{\id}{\convliealg{H}{\alg}}$ and $\exp_{\convolution} \circ \log_{\convolution} = \restrict{\id}{\convgroup{H}{\alg}}$. Similarly it is straightforward to derive
\begin{align*}
	\forall \phi, \psi \in \convliealg{H}{\alg}\!:\ \phi \convolution \psi = \psi \convolution \phi &\Rightarrow \exp_{\convolution} (\phi + \psi) = (\exp_{\convolution} \phi) \convolution (\exp_{\convolution} \psi)\quad\text{and}\\
	\forall \phi, \psi \in \convgroup{H}{\alg}\!:\ \phi \convolution \psi = \psi \convolution \phi &\Rightarrow \log_{\convolution} (\phi \convolution \psi) = \log_{\convolution} \phi + \log_{\convolution} \psi.
\end{align*}
This construction provides an infinite\footnote{unless $H$ and $\alg$ are finite dimensional} dimensional Lie group together with its Lie algebra! It is easy to check that $\exp_{\convolution}$ indeed is the exponential map, saying that
\begin{equation}
	\frac{\partial}{\partial t} \exp_{\convolution} (t v) = v \convolution \exp_{\convolution} ( t v )
	\label{eq:exp-diff}
\end{equation}
for any $v \in \convliealg{H}{\alg}$ (the differentiation is to be understood pointwise at fixed $x\in H$). Also we find that the Lie bracket on $\convliealg{H}{A}$ is induced by the convolution product through
\begin{equation*}
	\forall v,w \in \convliealg{H}{\alg}\!: \quad 
	\convkommu{v}{w} = 
	\restrict{\frac{\partial^2}{\partial s\, \partial t}}{s=t=0} \left[
			\exp_{\convolution}(tv) \convolution \exp_{\convolution}(sw) \convolution \exp_{\convolution}(-tv) \convolution \exp_{\convolution}(-sw)
	\right].
\end{equation*}
The bijectivity of the exponential map allows for the definition of fractional product
\begin{equation}\label{eq:conv-power}
	\forall g \in \convgroup{H}{\alg}\!:\ \forall \mu \in \K\!:\ g^{\convolution \mu} \defas \exp_{\convolution} \left( \mu \log_{\convolution} g \right)
\end{equation}
in the group, coinciding with the usual iterated convolution product in the case of integer $\mu \in \Z$! In particular any $g\in\convgroup{H}{\alg}$ defines a one-parameter subgroup $\K \ni \mu \mapsto g^{\convolution \mu}$.

Apparently $\convgroup{H}{A}$ is a very interesting structure to study and it turns out that a subgroup of it (given by the characters) is the natural setting of the physicists \emph{renormalization group}. We will fruitfully employ these ideas in section \ref{sec:higher-orders} and recommend \cite{CK:RH2} for further reading.

\section{Algebraic Birkhoff decomposition}
As was discovered by Dirk Kreimer in \cite{CK:RH1}, the recursive procedure of renormalization\footnote{We refer to chapter 5 of \cite{Collins}, in particular section 3. Equations (5.3.6) and (5.3.7) therein essentially are \eqref{eq:birkhoff-dec} below!} may be formulated in algebraic terms as the \emph{Birkhoff decomposition} from
\begin{definition}\label{def:birkhoff}
	Let $H$ be a bialgebra and $\alg = \alg_- \oplus \alg_+$ an algebra, decomposed into the direct sum of two vector spaces $\alg_{\pm}$. Then a \emph{Birkhoff decomposition} of some $\phi \in \convgroup{H}{\alg}$ is a pair $\phi_{\pm} \in \convgroup{H}{\alg}$ such that
	\begin{equation}
		\phi = \phi_-^{\convolution -1} \convolution \phi_+
		\quad\text{and}\quad
		\phi_{\pm} (\ker\counit) \subseteq \alg_{\pm}.
		\label{eq:birkhoff}
	\end{equation}
%	\begin{enumerate}
%		\item $\phi = \phi_-^{\convolution -1} \convolution \phi_+$
%		\item $\phi_- (\1) = \phi_+(\1) = \1_{\alg}$ (remember $\phi(\1) = \1_{\alg}$)
%		\item $\phi_{\pm}(\ker\counit) \subseteq \alg_{\pm}$
%	\end{enumerate}
\end{definition}
For example, as we will see in section \ref{sec:toymodel-dimreg}, \emph{dimensional regularization} yields characters \mbox{$\phi\!: H \rightarrow \alg$} mapping to meromorphic functions\footnote{without essential singularities at $\reg\rightarrow 0$, hence series $\sum_{n \geq N} a_n \reg^n$ for some $N \in \Z$} in a complex variable $\reg$, identified with their Laurent series around $\reg=0$ in $\alg = \K[z^{-1}, z]]$. We want to take the limit $\reg \rightarrow 0$, which in general is impossible due to the presence of singularities.

The \emph{minimal subtraction scheme} is defined by splitting $\alg$ as
\begin{equation}
	\alg_- \defas z^{-1} \K[z^{-1}]
	\quad \text{and} \quad
	\alg_+ \defas \K[[z]],
	\label{eq:ms-splitting}
\end{equation}
hence a Birkhoff decomposition will provide some $\phi_+$ mapping to functions $\alg_+$ holomorphic at $z=0$. The idea of renormalization is to take $\phi_+$ as the definition of the \emph{renormalized} $\phi$, allowing for the \emph{physical limit} $\restrict{\phi_+}{z=0}$.

Our prior study of connectedness and filtrations now pays off in
\begin{satz}\label{satz:birkhoff-dec}
	Let $H$ be a connected bialgebra and $\alg = \alg_- \oplus \alg_+$ a target algebra splitted into subspaces $\alg_{\pm}$. Then every $\phi \in \convgroup{H}{\alg}$ admits a unique Birkhoff decomposition. For $x\in \ker \counit$ it may be computed recursively by
	\begin{equation}
		\phi_-(x) = -R \left[ \bar{\phi}(x) \right]
		\quad \text{and} \quad
		\phi_+(x) = (\id - R) \left[ \bar{\phi}(x)\right],
		\label{eq:birkhoff-dec}
	\end{equation}
	where $R:\ \alg \twoheadrightarrow \alg_-$ denotes the projection induced by the splitting and
	\begin{equation}
		\bar{\phi} \defas \phi + \mul \circ ( \phi_- \tp \phi) \circ \cored, \quad
		\bar{\phi} (x) = \phi(x) + \sum_x \phi_-(x') \phi(x'')
		\label{eq:rbar}
	\end{equation}
	is the \emph{Bogoliubov map} (also called \emph{$\bar{R}$-map}).
\end{satz}

\begin{proof}
	Given some Birkhoff decomposition $\phi_{\pm}$ of $\phi$, \eqref{eq:birkhoff-dec} is an immediate consequence of $\bar{\phi} = \phi + \phi_- \convolution \phi - \phi -\phi_- = \phi_+ - \phi_-$ as $R^2 = R$ and $\phi_{\pm}(\ker \counit) \subseteq \alg_{\pm}$.
	Taking any connected filtration of $H$, starting with $\phi_-(\1)=\1_{\alg}$ we see inductively that $\phi_-$ is uniquely determined on each $H^n$ through \eqref{eq:birkhoff-dec} and $\cored (H^{n+1} )\subseteq \sum_{k=1}^n H^k \tp H^{n+1-k}$.
	
	Having thus proven uniqueness of $\phi_-$ and therefore of $\phi_+ = \phi_- \convolution \phi$ as well, we obtain existence by defining $\phi_-$ recursively on each $H^n$ using \eqref{eq:birkhoff-dec}. This construction ensures $\phi_-(\ker \counit) \subseteq \alg_- = \im R$, but as we must set $\phi_+ \defas \phi_- \convolution \phi$ to obtain a Birkhoff decomposition it remains to check $\phi_+(\ker \counit) \subseteq \alg_+ = \ker R$, which is immediate by
	\begin{equation*}
		\restrict{\phi_+}{\ker\counit}
		\defas \left[\phi_- \convolution \phi \right]_{\ker\counit}
		= \left[\phi_- + \bar{\phi} \right]_{\ker\counit}
		\urel{\eqref{eq:birkhoff-dec}} \left[ (\id - R) \circ \bar{\phi}  \right]_{\ker\counit}. \qedhere
	\end{equation*}
\end{proof}
In our example of minimal subtraction, the projection $\Rms$ just keeps the (finitely many) pole terms $\sum_{n<0} a_n z^n$. For primitive elements $p\in\Prim(H)$, \eqref{eq:birkhoff-dec} simplifies to
\begin{equation*}
	\bar{\phi}(p) = \phi(p)
	,\quad
	\phi_-(p) = -R \left[ \phi(p) \right]
	\quad\text{and}\quad
	\phi_+(p) = (\id-R) \left[ \phi(p) \right].
\end{equation*}
Thus for primitives, the minimal subtraction scheme simply discards all poles from the Laurent series of $\phi(p)$ to obtain $\phi_+(p)$. Suppose $\phi(p)=s^{-\reg} F(\reg)$ for $F(\reg) = \sum_{n=-1}^{\infty} \coeff{n} \reg^n$, then $\Rms$ delivers the \emph{counterterm}\footnote{This is the name for $\phi_-$ common in physics.} $\phi_-(p) = \frac{\coeff{-1}}{\reg}$ and the renormalized value
\begin{equation*}
	\phi_+ (p)
	= \left( \sum_{n=0}^{\infty} \coeff{n} \reg^n \right) s^{-\reg} + \coeff{-1}\sum_{n=1}^{\infty} \frac{{ (-\ln s)}^n}{n!} \reg^{n-1}.
\end{equation*}
In this case, the physical limit $\reg \rightarrow 0$ becomes
\begin{equation}
	\lim_{\reg\rightarrow 0} \phi_+(p)
	= \coeff{0} - \coeff{-1} \ln s.
	\label{toyZ-MS:()}
\end{equation}

\subsubsection{Inverses as Birkhoff decompositions}
Consider a connected bialgebra $H$ and as target algebra $\alg \defas H$ itself, splitted as $H = H \oplus \set{0}$ with $\alg_- \defas H$ and $\alg_+ \defas \set{0}$ (hence $R=\id$). Then for any $\phi\in\convgroup{H}{H}$, its Birkhoff decomposition fulfils $\phi_+ (\1) = \1$ and $\phi_+(\ker \counit) \subseteq \set{0}$. We conclude that $\phi_+ = e = \unit \circ \counit$ by \eqref{eq:augmentation}.

Hence we obtain $\phi = \phi_-^{\convolution-1} \convolution \phi_+ = \phi_-^{\convolution -1}$: The counterterm $\phi_-$ in this scheme is nothing but the convolution inverse of $\phi$! In particular this gives another proof of theorem \ref{satz:convolution-group} delivering the recursive formula
\begin{equation}
	\forall x \in \ker \counit\!: \quad
	\phi^{\convolution-1} (x) = - \phi(x) - \sum_{x} \phi^{\convolution-1} (x') \phi(x'').
	\label{eq:conv-inv-recursive}
\end{equation}
For example, in this setting the antipode $S$ is the counterterm $S=\phi_-$ of $\phi = \id$, thus
\begin{equation}
	\forall x \in \ker \counit\!: \quad
	S(x) = -x - \sum_{x} S(x') x''.
	\label{eq:antipode-recursive}
\end{equation}
We also obtain $\phi^{\convolution-1} (x) = - \phi(x) - \sum_{x} \phi(x')\phi^{\convolution-1} (x'')$ by considering a flipped decomposition $\phi = \phi_+ \convolution {\phi}^{\convolution -1}$.

\subsection{Decomposition of characters}
As we are particularly interested into \emph{characters} $\phi$, we ask whether the Birkhoff decomposition respects this special property in
\begin{proposition}
	Let $H$ be a connected bialgebra, $\phi \in \chars{H}{\alg}$ a morphism of (unital) algebras with commutative $\alg$ and $\alg = \alg_- \oplus \alg_+$ a splitting into \emph{subalgebras}\footnote{Note that $\alg_+$ and $\alg_-$ do not need to be unital!}. Then the Birkhoff decomposition parts $\phi_-$ and $\phi_+$ are algebra morphisms themselves.
\end{proposition}
\begin{proof}
	We prove the multiplicativity of $\phi_-$ inductively: Let \mbox{$\phi_-(xy) = \phi_-(x) \phi_-(y)$} be true for any $x,y \in H^n$ for some $n \in \N_0$, considering some filtration $H = \sum_{n\in\N_0} H^n$ with $H^0 = \K \cdot \1$ providing a trivial start of the induction. Then for any $x,y \in H^{n+1} \cap \ker \counit$,
	\settowidth{\wurelwidth}{\eqref{eq:birkhoff-dec}}
	\begin{align*}
		\phi_- (xy) &\wurel{\eqref{eq:birkhoff-dec}}
				-R\left[ \phi(xy) + \sum_{x\cdot y} \phi_- \left( \{xy\} ' \right) \phi \left( \{xy\}'' \right) \right] \\
		&\wurel{}
				-R \left[ \phi(x) \phi(y) + \sum_{x} \phi_-(x') \phi(x'') \sum_{y} \phi_-(y') \phi(y'') + \phi_-(x) \phi(y) + \phi(x)\phi_-(y) \right. \\
		&\quad\quad \left. + \left\{ \phi(x) + \phi_-(x) \right\} \sum_{y} \phi_-(y') \phi(y'') + \left\{ \phi(y) + \phi_-(y) \right\} \sum_{x} \phi_-(x') \phi(x'') \right] \\
		&\wurel{} -R \left[ \left\{ \phi(x) + \sum_{x} \phi_-(x') \phi(x'') \right\} \cdot \left\{ \phi(y) + \sum_{y} \phi_-(y')\phi(y'') \right\} \right. \\
		&\quad\quad \left. + \phi_-(x) \cdot \left\{ \phi(y) + \sum_{y} \phi_-(y') \phi(y'') \right\} + \left\{ \phi(x) + \sum_{x} \phi_-(x') \phi(x'') \right\} \cdot \phi_-(y) \right] \\
		&\wurel{} R \left[ \left\{ R\bar{\phi}(x) \right\} \bar{\phi}(y) + \bar{\phi}(x) \left\{ R\bar{\phi}(y) \right\} - \bar{\phi}(x) \bar{\phi}(y) \right] 
		\urel{\eqref{eq:Rota-Baxter}} \left[ R\bar{\phi}(x) \right] \cdot \left[ R\bar{\phi}(y) \right] \\
		&\wurel{\eqref{eq:birkhoff-dec}} \left[ -\phi_-(x) \right] \cdot \left[- \phi_-(y) \right]
		= \phi_-(x) \cdot \phi_-(y)
	\end{align*}
	where we decomposed $\cored (xy) = \Delta (xy) - \1 \tp xy - xy\tp \1 = \Delta(x) \cdot \Delta(y) - \1 \tp xy - xy \tp \1 = (\cored x + \1 \tp x + x \tp \1) \cdot (\cored y + \1 \tp y + y \tp \1) - \1 \tp xy - xy \tp \1$ and exploited the so-called \emph{Rota-Baxter equation}
	\begin{equation}
		R \circ \mul_{\alg} + \mul_{\alg} \circ (R \tp R) = R \circ \mul_{\alg} \left[ R \tp \id + \id \tp R  \right].
		\label{eq:Rota-Baxter}
	\end{equation}
	This is equivalent to $R(xy) + R(x) R(y) = R \left[ (Rx)y + x(Ry) \right]$ for all $x,y \in \alg$ and in particular fulfilled for any projection $R$. This comes about as:
	\begin{enumerate}
		\item If $x,y \in \ker R$, then also $xy \in \ker R$ ($\ker R = \alg_+$ is a subalgebra) such that both sides of \eqref{eq:Rota-Baxter} vanish.
		\item If $x,y \in \im R = \alg_-$, so is $xy$, hence by $\restrict{R}{\im R} = \restrict{\id}{\im R}$ both sides of \eqref{eq:Rota-Baxter} give $2 R (xy) = 2xy$.
		\item Let $x\in \ker R$ and $y \in \im R$, then \eqref{eq:Rota-Baxter} reduces to $R(xy) = R[x(Ry)]$ which follows from $Ry = y$. Analogously treat the case when $x\in\im R$ and $y\in\ker R$. \qedhere
	\end{enumerate}
\end{proof}
More generally, we may use \eqref{eq:birkhoff-dec} to define $\phi_{\pm}$ for arbitrary $R \in \End{\alg}$, without restricting to projections $R=R^2$. This \emph{generalized Birkhoff decomposition} clearly fulfils \mbox{$\phi_- (\ker \counit) \subseteq \im R$} and \mbox{$\phi_+ (\ker \counit) \subseteq \ker R$}.

The above proof applies to this case as well, proving $\phi_{\pm}\in\chars{H}{\alg}$ for \mbox{$\phi \in \chars{H}{\alg}$} as long as $R$ fulfils \eqref{eq:Rota-Baxter}. This motivates the investigation of Birkhoff decompositions and renormalization in the context of Rota-Baxter algebras, an active and recent field of research (see \cite{Spitzer} and references therein).

We close this section by considering multiplicative\footnote{We call renormalization schemes $R=R^2$ \emph{multiplicative} iff they are morphisms of algebras.} renormalization schemes in
\begin{proposition}
	Let $H$ be a connected bialgebra and $\alg = \alg_- \oplus \alg_+$ a commutative algebra with renormalization scheme $R=R^2\!: \alg \twoheadrightarrow \alg_-$ that is also a morphism of unital algebras (hence in particular $\1_{\alg}\in\alg_-$). Then for any $\phi \in \convgroup{H}{\alg}$ the Birkhoff decomposition reads
	\begin{equation}
		\phi_- = R \circ \phi^{\convolution -1}
		\quad \text{and} \quad
		\phi_{+} = (R\circ \phi^{\convolution -1}) \convolution \phi.
		\label{eq:birkhoff-multiplicative-scheme}
	\end{equation}
\end{proposition}
\begin{proof}
	First note that trivially $R \circ \phi^{\convolution-1} (\ker \counit) \subseteq \im R = \alg_-$ and $R \circ \phi^{\convolution-1} (\1) = {\1}_{\alg}$ by $R({\1}_{\alg}) = {\1}_{\alg}$.	
	Therefore uniqueness of the Birkhoff decomposition implies that it suffices to check $\left[ (R \circ \phi^{\convolution -1}) \convolution \phi \right] (\ker \counit) \subseteq \alg_+ = \ker R$, which follows from
	\begin{equation*}
		R \circ \left[ \left( R\circ \phi^{\convolution -1} \right) \convolution \phi \right]
		= \left( R^2 \circ {\phi}^{\convolution -1} \right) \convolution (R \circ \phi)
		= R \circ \left( {\phi}^{\convolution -1} \convolution \phi \right)
		= R \circ u_{\alg} \circ \counit_H
		= e. \qedhere
	\end{equation*}
\end{proof}
As an application of \eqref{eq:character-group} we deduce in particular
\begin{korollar}
	If $H$ is a Hopf algebra and $R=R^2\in\chars{\alg}{\alg}$ a renormalization scheme on the commutative algebra $\alg$, then for any $\phi \in \chars{H}{\alg}$ the Birkhoff decomposition reads
	\begin{equation}
		\phi_- = R \circ \phi \circ S = R \circ {\phi}^{\convolution -1}
		\quad\text{and}\quad
		\phi_+ = \left( R \circ \phi^{\convolution -1} \right) \convolution \phi.
		\label{eq:birkhoff-character}
	\end{equation}
\end{korollar}
Hence for multiplicative schemes, the renormalization happens entirely on the combinatorial side of the Hopf algebra! In contrast, the general case really forces inductive calculation of $\phi_-$ with lots of nested applications of $R$. The much simpler case of \eqref{eq:birkhoff-multiplicative-scheme} is present in physics in the \emph{momentum schemes} we will encounter in the next chapter, leading to superior algebraic properties in comparison to schemes like \emph{minimal subtraction}, where $\Rms \notin \chars{\alg}{\alg}$:
\begin{equation*}
	\Rms \left( z \cdot \frac{1}{z^2} \right)
	= \frac{1}{z}
	\neq 0
	= \Rms \left(z \right) \cdot \Rms \left( \frac{1}{z^2} \right).
\end{equation*}

\section{Rooted Trees}
\label{sec:H_R}
So far we did not give any examples of Hopf algebras! We only mention that the tensor algebra $T(V)$ over a vector space $V$ and the universal enveloping algebra $\mathcal{U}(\mathcal{L})$ of a Lie algebra $\mathcal{L}$ carry Hopf algebra structures in a natural way and refer to \cite{Sweedler} for details. We will have a very brief look at symmetric algebras in section \ref{sec:polynomials}.

However, in this section we introduce the Hopf algebra of rooted trees as it describes the combinatorics of renormalization of nested and disjoint subdivergences\footnote{see section \ref{sec:iterated-insertions}} for a single primitive divergence in {\qft}, which is the content of the \emph{toy model} to be investigated in the following chapter.
\begin{definition}
	A graph theoretic tree $T$ consists of sets $V(T)$ of nodes and $E(T)\subset \setexp{e\subseteq V(T)}{\abs{e} = 2}$ of edges such that $T$ is connected\footnote{%
For any $v,w \in V(T)$ there exists a path $v=v_0 \rightarrow v_1 \rightarrow \ldots \rightarrow v_n = w$ of nodes such that $\set{v_i,v_{i+1}} \in E(T)$ for any $0\leq i<n$.}
	and simply connected.\footnote{%
	$T$ does not contain any \emph{cycles} of edges.}

	We define a \emph{labelled rooted tree} as a pair $(T, r)$ of a graph theoretic tree $T$ and a distinguished node $r \in V(T)$, called the \emph{root} of $(T, r)$.
\end{definition}
An isomorphism $\phi\!:\ (T,r) \rightarrow (T', r')$ of labelled rooted trees is an isomorphism\footnote{%
A bijection $\phi\!:\  V(T) \rightarrow V(T')$ such that $\set{v,w} \in E(T) \gdw \set{\phi(v),\phi(w)} \in E(T')$ for any $v,w \in V(T)$.}
of the graphs $T$ and $T'$ fixing the root $\phi(r) = r'$. We are only interested in isomorphism classes of trees as we do not care about the names of the nodes -- only their connections count. We finally state
\begin{definition}
	A \emph{rooted tree} is an isomorphism class of labelled rooted trees. Let
	\begin{equation}
		\trees = \set{\tree{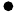}, \tree{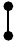}, \tree{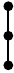}, \tree{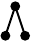}, \tree{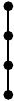}, \tree{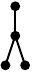}, \tree{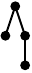}, \tree{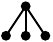}, \ldots}
		\label{eq:trees}
	\end{equation}
	denote the set of rooted trees. A \emph{rooted forest} is a disjoint union of rooted trees,
	\begin{equation}
		\forests = \set{\1} \cupdot \trees \cupdot \set{\tree{+-}\tree{+-}, \tree{+-}\tree{+-}\tree{+-}, \tree{+-}\tree{++--}, \tree{+-}\tree{+-}\tree{+-}\tree{+-}, \tree{+-}\tree{+-}\tree{++--}, \tree{+-}\tree{++-+--}, \tree{+-}\tree{+++---}, \ldots}
		\label{eq:forests}
	\end{equation}
	shall denote the set of rooted forests. Here, $\1 \defas \emptyset$ denotes the empty rooted forest (that does not contain any nodes). Every rooted forest $f$ is the union of a unique multiset of rooted trees denoted by $\pi_0(f)$.
\end{definition}
In the intuitive pictorial representation of rooted forests, as used in \eqref{eq:trees} and \eqref{eq:forests}, we will always draw the roots at the top. Note that there is no order among the children of a node or the trees of a forest, such that
\begin{equation}\label{eq:unordered}
	\tree{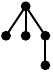} = \tree{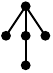} = \tree{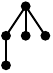}
	\quad \text{and} \quad
	\tree{+-}\tree{++--}\tree{++-+--}
	= \tree{+-} \tree{++-+--} \tree{++--}
	= \tree{++--} \tree{+-} \tree{++-+--}
	= \tree{++--} \tree{++-+--} \tree{+-}
	= \tree{++-+--} \tree{+-} \tree{++--}
	= \tree{++-+--} \tree{++--} \tree{+-}.
\end{equation}
These trees and forests are sometimes called \emph{non-planar}, emphasizing that they do not carry a distinguished planar embedding with them. However, to avoid confusion with graph theory\footnote{where every tree is considered to be \emph{planar}} we prefer to call these \emph{unordered} rooted trees and forests.

On the other hand, one can consider forests with a distinguished total order among the children of any node and among the trees of a forest. Thus the drawings in \eqref{eq:unordered} all represent different \emph{ordered (planar)} rooted trees forests.

\begin{definition}
	The algebra $H_R$ of (unordered) rooted trees is the symmetric algebra $H_R \defas S(\lin\trees) = \K[\trees]$ generated by rooted trees. As a vector space it has the natural basis $\forests$, each forest representing a unique monomial in trees.

	The \emph{grafting operator} $B_+ \in \End(H_R)$ is defined by adding a new root (above all existing roots) to a rooted forest, extended linearly. So for example,
	\begin{equation*}
		B_+ \left( \alpha \1 + \beta \tree{+-} + \gamma \tree{+-}\tree{++--} \right)
		= \alpha \tree{+-} + \beta \tree{++--} + \gamma \tree{++-++---}.
	\end{equation*}
\end{definition}
Note that $H_R = \bigoplus_{n \in \N_0} H_{R,n}$ carries a natural grading\footnote{called \emph{weight} or \emph{degree}} by node number through
\begin{equation}\label{eq:H_R-graduation}
	\forall n\in\N_0\!: \quad
	H_{R,n} = \lin \forests_n
	\quad \text{for} \quad
	\forests_n \defas \setexp{f\in\forests}{\abs{f} \defas \abs{V(f)} = n}.
\end{equation}
Clearly, $B_+$ is homogenous of degree one with respect to this grading. Also note $\im B_+ = \lin \trees$, in particular $B_+\!:\ \forests \rightarrow \trees$ delivers a bijection.

\subsection{The coproduct}
\label{sec:coproduct}
To turn $H_R$ into a bialgebra, we define the coproduct $\Delta$ by requiring
\begin{equation}
	\Delta \circ B_+
	= B_+ \tp \1 + (\id \tp B_+) \circ \Delta,
	\label{eq:B_+-cocycle}
\end{equation}
as this determines $\Delta$ uniquely as a morphism of unital algebras. For example,
\begin{align}
	\Delta (\tree{+-})
	&= \Delta \circ B_+ (\1)
	= B_+(\1) \tp \1 + (\id \tp B_+) \circ \Delta (\1)
%	= \tree{+-} \tp \1 + (\id \tp B_+) (\1 \tp \1)
	= \tree{+-} \tp \1 + \1 \tp \tree{+-} 
	\label{coproduct:()}\\
	\Delta \left( \tree{++-+--}\right)
	&= \Delta \circ B_+ \left({\tree{+-}}^2 \right)
	= B_+ \left( {\tree{+-}}^2 \right) \tp \1 + (\id \tp B_+) \circ \Delta \left( {\tree{+-}}^2 \right) \nonumber\\
	&= \tree{++-+--} \tp \1 + (\id \tp B_+) \left( {\left[ \Delta(\tree{+-}) \right]}^2 \right)
	= \tree{++-+--} \tp \1 + (\id \tp B_+) \left( \tree{+-}\tree{+-} \tp \1 + 2 \tree{+-} \tp \tree{+-} + \1 \tp \tree{+-}\tree{+-} \right) \nonumber\\
	&= \tree{++-+--} \tp \1 + \tree{+-}\tree{+-} \tp \tree{+-} + 2\tree{+-} \tp \tree{++--} + \1 \tp \tree{++-+--}.
	\label{coproduct:(()())}
\end{align}
\begin{proposition}
	The coproduct $\Delta\!:\ H_R \rightarrow H_R \tp H_R$ defined by \eqref{eq:B_+-cocycle} is coassociative.
\end{proposition}
\begin{proof}
	As $\Delta$ is multiplicative by construction, $(\Delta \tp \id) \circ \Delta = (\id \tp \Delta) \circ \Delta$ needs only to be checked on trees. We employ induction over the number of nodes: Let the claim be true for any tree of less than $N$ nodes (hence also on $\forests_{N-1}$). Now consider $t\in\trees_N\defas\setexp{t\in\trees}{\abs{t}=N}$, then $t=B_+(f)$ for $f\in\forests_{N-1}$ such that
	\begin{align*}
		(\Delta \tp \id) \circ \Delta (t)
		&= (\Delta \tp \id) \circ \left[ t \tp \1 + (\id \tp B_+) \circ \Delta (f) \right]\\
		&= \Delta(t) \tp \1 + (\id \tp \id \tp B_+) \circ (\Delta \tp \id) \circ \Delta (f) \\
		&= t \tp \1 \tp \1 + \left[(\id \tp B_+) \circ \Delta (f) \right] \tp \1 + \left\{\id \tp [(\id \tp B_+) \circ \Delta]\right\} \circ \Delta (f) \\
		&= t \tp \1 \tp \1 + \left\{ \id \tp \left[ B_+ \tp \1 + (\id \tp B_+) \circ \Delta  \right] \right\} \circ \Delta (f) \\
		&= (\id \tp \Delta) \circ [ t\tp \1 + (\id \tp B_+) \circ \Delta (f) ]
		= (\id \tp \Delta) \circ \Delta (t). \qedhere
	\end{align*}
\end{proof}
This coproduct can be understood combinatorially as follows: First note how a labelled rooted tree $t$ naturally induces a partial order on the set of its nodes by
	\begin{equation}\label{eq:rel-above}
		\forall v,w \in V(t)\!: \quad
		v \aboveeq w 
		\ \text{iff}\ v\ \text{lies on the path from}\ w\ \text{to the root}\ r(t).
	\end{equation}
	Then set $\indep(t)$ to contain exactly the \emph{independent subsets}\footnote{In the literature one considers instead certain subsets of edges, called \emph{admissible cuts}. However, we prefer the notion of independent sets as it in particular allows to state \eqref{eq:coproduct-cuts} for arbitrary forests. This does not work using the notion of admissible cuts, as long as one does not introduce an artificial \emph{complete cut} for each tree of a forest.
} $W \subseteq V(t)$ of nodes, meaning those $W$ such that for any $v,w\in W$ with $v \neq w$, neither $v\aboveeq w$ nor $v\beloweq w$.

\begin{proposition}
	For any forest $f\in\forests$ we have
	\begin{equation}
		\Delta(f) 
		= \sum_{W \in \indep(t)} P^W (f) \tp R^W (f),
		\label{eq:coproduct-cuts}
	\end{equation}
	where the \emph{pruned part} $P^W(f)$ denotes the forest made out of the subtrees with roots in $W$ and $R^W(f)$ is the forest spanned by the remaining nodes.
\end{proposition}
Graphically, $P^W(f)$ just contains all trees that fall down if one cuts right above each node in $W$. For example consider the tree $t=\tree{++-+--}$, then
\begin{align*}
	& \sum_{W\in I (t)} P^W \left( \treeabc \right) \tp R^W \left( \treeabc \right)
	= P^{\emptyset} \left( \treeabc \right) \tp R^{\emptyset} \left( \treeabc \right)
	+P^{\set{a}} \left( \treeabc \right) \tp R^{\set{a}} \left( \treeabc \right) \\
	& \quad +P^{\set{b}} \left( \treeabc \right) \tp R^{\set{b}} \left( \treeabc \right)
	+P^{\set{c}} \left( \treeabc \right) \tp R^{\set{c}} \left( \treeabc \right)
	+P^{\set{b,c}} \left( \treeabc \right) \tp R^{\set{b,c}} \left( \treeabc \right) \\
	&= \1 \tp \treeabc + \treeabc \tp \1 + \tree{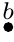} \tp \tree{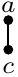} + \tree{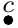} \tp \tree{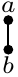} + \tree{+b-}\tree{+c-} \tp \tree{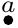}
	= \Delta \left( \treeabc \right).
\end{align*}
Clearly we have to pick a labelled representative for \eqref{eq:coproduct-cuts} to make sense. But as we afterwards pass to isomorphism classes of labelled trees again, this choice does not matter. The above example then delivers \eqref{coproduct:(()())}.

\begin{proof}
	The right-hand side of \eqref{eq:coproduct-cuts} is clearly multiplicative: For any $f,f'\in\forests$, note
	\begin{equation*}
		\indep(f\cdot f')
		= \indep(f \cupdot f')
		= \setexp{W \cupdot W'}{W\in \indep(f)\ \wedge\ W' \in\indep(f')},
	\end{equation*}
	$P^{W\cupdot W'}(f\cdot f') = P^W(f) \cdot P^{W'}(f')$ and $R^{W\cupdot W'}(f\cdot f') = R^W(f) \cdot R^{W'}(f')$ as $\aboveeq$ does not relate any nodes of $f$ with those of $f'$ in $f\cupdot f'$. Hence we can proceed inductively on trees again. Given \eqref{eq:coproduct-cuts} to hold on $\forests_{N-1}$, consider a tree of $N$ nodes $t = B_+(f)$ with $f\in \forests_{N-1}$:
	\begin{equation*}
		\indep(t)
		= \indep \left( B_+(f) \right)
		= \set{r(t)} \cupdot \indep(f)
	\end{equation*}
	is an immediate consequence of $r(t) \aboveeq v$ for any $v\in V(t)$. Hence observe
	\begin{align*}
		\sum_{W\in \indep(t)} P^W(t) \tp R^W(t)
		&= P^{\set{r(t)}} \tp R^{\set{r(t)}} + \sum_{W \in \indep(f)} P^W(t) \tp R^W(t) \\
		&= t \tp \1 + \sum_{W \in \indep(f)} P^W(f) \tp B_+ \left( R^W(f) \right) \\
		&= \left[ B_+ \tp \1 + (\id \tp B_+)\circ \Delta \right] (f) 
		\urel{\eqref{eq:B_+-cocycle}} \Delta \circ B_+(f)
		= \Delta (t). \qedhere
	\end{align*}
\end{proof}

We close the discussion of the coproduct with a few more examples:
	\begin{align}
%		\Delta \left( \tree{+-} \right) &= \1 \tp \tree{+-} + \tree{+-} \tp \1
%			\label{coproduct:()} \\
		\Delta \left( \tree{++--} \right) &= \1 \tp \tree{++--} + \tree{+-} \tp \tree{+-} + \tree{++--} \tp \1
			\label{coproduct:(())} \\
		\Delta \left( \tree{+++---} \right) &= \1 \tp \tree{+++---} + \tree{+-} \tp \tree{++--} + \tree{++--} \tp \tree{+-} + \tree{+++---} \tp \1
			\label{coproduct:((()))} \\
%		\Delta \left( \tree{++-+--} \right) &= \1 \tp \tree{++-+--} + 2 \tree{+-} \tp \tree{++--} + \tree{+-}\tree{+-} \tp \tree{+-} + \tree{++-+--} \tp \1
%			\label{coproduct:(()())} \\
		\Delta \left( \tree{++-++---} \right) &= \1 \tp \tree{++-++---} + \tree{+-} \tp \left( \tree{+++---} + \tree{++-+--} \right) + \left( \tree{+-}\tree{+-} + \tree{++--} \right) \tp \tree{++--} + \wald{\tree{+-}\tree{++--}} \tp \tree{+-} + \tree{++-++---} \tp \1
			\label{coproduct:(()(()))}
	\end{align}
\subsection{The bialgebra}
\begin{definition}
	On $H_R$ we define the functional $\counit \in H_R' \defas \Hom(H_R, \K)$ by
	\begin{equation}
		\text{$\counit (\1) \defas 1$ and $\counit(f) \defas 0$ for any forest $f\in\forests\setminus\set{\1}$.}
		\label{eq:H_R-counit}
	\end{equation}
\end{definition}

\begin{satz}
	$(H_R, \cupdot, \1, \Delta, \counit)$ is a bialgebra.
\end{satz}

\begin{proof}
	We defined $\Delta$ to be a morphism of unital algebras, hence it only remains to check \eqref{eq:koalg-koeins}. For any forest $f\in\forests$, $\counit \left( P^W(f) \right) = 0$ unless $W=\emptyset$ by \eqref{eq:H_R-counit} proves $(\counit \tp \id) \circ \Delta = \id$ as $R^{\emptyset}(f) = f$. Similarly $\counit \left( R^W(f) \right) = 0$ unless $W=\setexp{r(t)}{t \in \pi_0(f)}$ is the set of all roots, pruning the complete forest $P^W(f) = f$ while $R^W(f) = \1$ such that $(\id \tp \counit) \circ \Delta = \id$.
\end{proof}
Clearly $\Delta$ respects the graduation of $H_R$ by $\Delta (H_{R,n}) \subseteq \bigoplus_{k=0}^n H_{R,k} \tp H_{R, n-k}$ as it partitions the nodes into $P^W(f)$ and $R^W(f)$.
\begin{korollar}
	The number of nodes delivers a graduation of the bialgebra $H_R$ of rooted trees. By $H_{R,0} = \K \cdot \1$ it is connected and hence a Hopf algebra through corollary \ref{satz:zsh-bialg-hopf}.
\end{korollar}
The antipode of $H_R$ can be calculated recursively by $S(f) = -f -\sum_f S(f')f'' = -f-\sum_f f'S(f'')$ for any non-empty forest $f\in\forests$ as a consequence\footnote{alternatively recall \eqref{eq:antipode-recursive}} of $e = S \convolution \id = \id\convolution S$:
\begin{align}
	S \left( \tree{+-} \right)
		&= - \tree{+-}
		\label{antipode:()} \\
	S \left( \tree{++--} \right)
		&= - \tree{++--} + \tree{+-}\tree{+-}
		\label{antipode:(())} \\
	S \left( \tree{+++---} \right)
		&= - \tree{+++---} + 2 \tree{+-}\tree{++--} - \tree{+-}\tree{+-}\tree{+-}
		\label{antipode:((()))} \\
	S \left( \tree{++-+--} \right)
		&= - \tree{++-+--} + 2\tree{+-}\tree{++--}- \tree{+-}\tree{+-}\tree{+-}
		\label{antipode:(()())} \\
	S \left( \tree{++-++---} \right)
		&= - \tree{++-++---} + \tree{+-} \left( \tree{+++---} + \tree{++-+--} \right) + \tree{++--}\tree{++--} - 3 \tree{+-}\tree{+-}\tree{++--} + \tree{+-}\tree{+-}\tree{+-}\tree{+-}
		\label{antipode:(()(()))}
\end{align}

\subsection{Tree factorials}
In the study of Feynman rules on $H_R$ in the next chapter we will encounter the \emph{tree factorial}\footnote{It can be considered as a canonical Feynman rule by itself according to \eqref{eq:int-rules}.} defined in
\begin{definition}\label{def:tree-factorial}
	We define the algebra morphism $(\cdot)!\!:\ H_R \rightarrow \K$ by requesting
	\begin{equation}
		\forall f\in\forests\!:\quad
		\left[ B_+(f) \right]! = f! \cdot \abs{B_+(f)}.
		\label{eq:tree-factorial}
	\end{equation}
\end{definition}
This tree factorial fulfils many interesting combinatorial relations like \eqref{eq:tree-factorial-feet}. Examples are given in \eqref{eq:int-rules-correlation-expansion} and section \ref{sec:DSE}. Denoting the subtree rooted at a node $v$ of a forest $f$ by $f_v$, it is immediate to check
\begin{equation}
	\forall f\in\forests\!:\quad
	f! = \prod_{v\in V(f)} \abs{f_v}.
	\label{eq:tree-factorial-product}
\end{equation}

\section{Hochschild cohomology}
Dualizing the classical Hochschild homology of algebras gives a cohomology of coalgebras. A special case of this cohomology, applied to bialgebras (like $H_R$ or the Hopf algebra of Feynman graphs), turns out to capture the connection from single contributions of individual trees (or graphs) in perturbation theory to full \emph{correlation functions} with the help of \emph{Dyson-Schwinger-equations}. In section \ref{sec:DSE} we will have a very brief look on this formalism.

The one-cocycles also play a fundamental role in the definition of Feynman rules through the universal property \eqref{eq:H_R-universal}. Their powerful algebraic properties are the key to the general proofs of locality and finiteness of renormalization, like in section \ref{sec:finiteness}.
\subsection{Cohomology of coalgebras}
\begin{definition}\label{def:comodule}
	Let $C$ be a coalgebra, then a \emph{left $C$-comodule} $M$ is a vector space with a map $\cmodcp_L \in \Hom(M, C \tp M)$ such that
	\begin{equation}
		(\id_{C} \tp \cmodcp_L) \circ \cmodcp_L
		= (\Delta \tp {\id}_{M}) \circ \cmodcp_L
		\quad\text{and}\quad
		(\counit \tp \id_M) \circ \cmodcp_L = \id_M.
		\label{eq:comodule}
	\end{equation}
	These conditions correspond to the commutativity of the diagrams
	\begin{equation}
		\vcenter{\xymatrix@C+0.5cm{
			{M} \ar[r]^-{\cmodcp_L} \ar[d]_{\cmodcp_L} & {C \tp M} \ar[d]^{\id_C\, \tp\, \cmodcp_L}\\
			{C \tp M} \ar[r]_-{\Delta \tp \id_M} & {C \tp C \tp M}
		}}
		\quad\text{and}\quad
		\vcenter{\xymatrix{
			{M} \ar[r]^-{\cmodcp_L} \ar[dr]_{\isomorph} & {C \tp M} \ar[d]^{\counit\, \tp\, \id_M} \\
			& {\K\tp M}
		}}.
		\label{cd:comodule}
	\end{equation}
	Analogously one defines \emph{right $C$-comodules} carrying a map $\cmodcp_R \in \Hom(M, M \tp C)$. Finally a \emph{$C$-bicomodule} is at the same time both a left- and a right-comodule such that $(\id_C \tp \cmodcp_R) \circ \cmodcp_L = (\cmodcp_L \tp \id_C) \circ \cmodcp_R$.
\end{definition}
\begin{definition}\label{def:hochschild}
	Let $C(\Delta,\counit)$ be a coalgebra and $M$ a $C$-bicomodule with left and right comodule structures $\cmodcp_L$ and $\cmodcp_R$. The \emph{Hochschild cochain complex} $(HC^{\cdot}(M), \partial^{\cdot})$ is defined via the cochains
	\begin{equation}
		HC^k(M) \defas \Hom \left(M, C^{\tp k}\right)
		\quad
		\forall k \in \N_0
		\label{eq:hochschild-cochains}
	\end{equation}
	and the \emph{coboundary maps} $\partial^k\!: HC^k(M) \rightarrow HC^{k+1}(M)$ given by
	\begin{equation}
		\partial^k \defas \sum_{i=0}^{k+1} {(-1)}^i d_i^k
		, \quad
		d_i^k (L) \defas \begin{cases}
			(\id \tp L) \circ \cmodcp_L 
				& \text{if $i=0$,} \\
			\left[ {\id}^{\tp i-1} \tp \Delta \tp {\id}^{\tp (k-i)} \right] \circ L
				& \text{if $1 \leq i \leq k$,} \\
			(L \tp \id) \circ \cmodcp_R
				& \text{if $i=k+1$.}
		\end{cases}
		\label{eq:hochschild-coboundary}
	\end{equation}
\end{definition}
\begin{lemma}\label{satz:hochschild-kettenkomplex}
	$(HC^{\cdot}(M), \partial^{\cdot})$ is a cochain complex, that is $\partial \circ \partial = 0$.
\end{lemma}
\begin{proof}
	For fixed $k\in \N_0$, $L \in HC^k(M)$ and any $1\leq i \leq k$, $1\leq j \leq k+1$ we have
	\begin{equation*}
		d^{k+1}_j \circ d^k_i (L) = \begin{cases}
			\big[ \id^{\tp j-1} \tp \Delta \tp \id^{\tp i-1-j} \tp \Delta \tp \id^{\tp k-i} \big] \circ L
				& \text{if $j<i$,} \\
			\big[ \id^{\tp i-1} \tp \Delta^2 \tp \id^{\tp k-i} \big] \circ L
				& \text{if $j \in \set{i, i+1}$,} \\
			\big[\id^{\tp i-1} \tp \Delta \tp \id^{\tp j-i-2} \tp \Delta \tp \id^{\tp k+1-j} \big] \circ L
				& \text{if $j>i+1$,} \\
		\end{cases}
	\end{equation*}
	exploiting coassociativity in the cases $j\in\set{i,i+1}$. Thus $d_j^{k+1} \circ d_i^k = d_i^{k+1} \circ d_{j-1}^k$ whenever $1\leq i < j \leq k+1$. We extend this relation to the case when $j=k+2$ through
	\begin{align*}
		d^{k+1}_{k+2} \circ d^k_i (L)
		&= \left\{ \left[ \left( \id^{\tp i-1} \tp \Delta \tp \id^{\tp k-i} \right) \circ L \right] \tp \id \right\} \circ \psi_R \\
		&= \left[ \id^{\tp i-1} \tp \Delta \tp \id^{\tp k+1-i} \right] \circ (L \tp \id) \circ \psi_R
		 = d^{k+1}_i \circ d^k_{k+1} (L).
	\end{align*}
	Analogously we can derive $d^{k+1}_j \circ d^k_0 = d^{k+1}_0 \circ d^{k}_{j-1}$ for $j>0$. Finally,
	\begin{align*}
		d^{k+1}_0 \circ d^k_{k+1} (L)
		&= \left\{ \id \tp \left[ \left( L \tp \id \right) \circ \cmodcp_R \right] \right\} \circ \cmodcp_L 
		= (\id \tp L \tp \id) \circ (\id \tp \cmodcp_R) \circ \cmodcp_L \\
		&= (\id \tp L \tp \id) \circ (\cmodcp_L \tp \id) \circ \cmodcp_R 
		= \left\{ \left[ \left( \id \tp L \right) \circ \cmodcp_L \right] \tp \id \right\} \circ \cmodcp_R 
		=d^{k+1}_{k+2} \circ d^k_0 (L),
	\end{align*}
	proves $d_j^{k+1} \circ d_i^k = d_i^{k+1} \circ d_{j-1}^k$ for all $0 \leq i < j \leq k+2$ such that
	\begin{align*}
		\partial^{k+1} \circ \partial^k
		&= \sum_{j=0}^{k+2} {(-1)}^j d^{k+1}_j \circ \sum_{i=0}^{k+1} {(-1)}^i d^k_i \\
		&= \sum_{0\leq j\leq i\leq k+1} {(-1)}^{i+j} d^{k+1}_j \circ d^k_i + \sum_{0\leq i\leq j-1\leq k+1} {(-1)}^{i+j} \underbrace{d^{k+1}_j \circ d^k_i}_{d^{k+1}_i \circ d^k_{j-1}}
		 = 0
	\end{align*}
	upon relabelling the indices $j-1 \mapsto i$ and $i \mapsto j$ in the second sum.
\end{proof}
Hence we obtain cohomology groups $HH^k(M) \defas H^k \left( HC^{\cdot}(M),\partial^{\cdot} \right) = \ker \partial^k / \im \partial^{k-1}$. In particular we can consider the natural bicomodule structure $\psi_L=\psi_R=\Delta$ on $M = C$.

\subsection{Cohomology of bialgebras}
On a bialgebra $H$, we may also take $\psi_L = \1 \tp \id$ and/or $\psi_R = \id \tp \1$ as left and right comodule structures\footnote{More generally, we can even consider $\cmodcp_L = g \tp \id$ and/or $\cmodcp_R = \id \tp h$ for any $g,h\in\Grp(H)$.}. All combinations of these indeed yield bicomodule structures on $H$ and define different cohomologies! As we will mainly be interested in Hochschild--$1$--cocycles $L \in HZ^1(H) \defas HC^1(H) \cap \ker \partial$, we list the defining equations for $1$-cocycles in these different cases:
\begin{align}
	\cmodcp_L &= \Delta & \text{and} && \cmodcp_R &= \Delta &\Rightarrow&& \Delta \circ L &= (\id \tp L + L \tp \id) \circ \Delta
	\label{bicomodule:dd}\\
	\cmodcp_L &= \Delta & \text{and} && \cmodcp_R &= \id \tp \1 &\Rightarrow&& \Delta \circ L &= (\id \tp L) \circ \Delta + L \tp \1
%	\label{eq:cocycle}
	\label{bicomodule:d1}\\
	\cmodcp_L &= \1 \tp \id & \text{and} && \cmodcp_R &= \Delta &\Rightarrow&& \Delta \circ L &= \1 \tp L + (L \tp \id) \circ \Delta 
	\label{bicomodule:1d}\\
	\cmodcp_L &= \1 \tp \id & \text{and} && \cmodcp_R &= \id \tp \1 &\Rightarrow&& \Delta \circ L &= \1 \tp L + L \tp \1
	\label{bicomodule:11}
\end{align}
In the first case $1$--cocycles are just coderivations, in the fourth case we get endomorphisms with $\im L \subseteq \Prim(H)$. The asymmetric cases \eqref{bicomodule:d1} and \eqref{bicomodule:1d} are the most interesting to us and in the following we will exclusively consider the bicomodule structure \eqref{bicomodule:d1} ($\cmodcp_L=\Delta$ and $\cmodcp_R=\id \tp \1$). To stress this we denote the Hochschild cochains, cocycles, coboundaries and cohomologies by $HC^k_{\counit}(H)$, $HZ^k_{\counit}(H)$, $HB^k_{\counit}(H)$ and $HH^k_{\counit}(H)$ respectively.

\begin{lemma}\label{satz:cocycle-props}
	For a 1-cocycle $L\in HZ^1_{\counit}(H)$ of a bialgebra $H$, we have $\im L \subseteq \ker \counit$ and $L(\1) \in \Prim(H)$. The evaluation map $\ev_{\1}\!: HZ^1_{\counit} \rightarrow \Prim(H), L \mapsto L(\1)$ factorizes to a well defined map $\widetilde{\ev}_{\1}\!: HH^1_{\counit} \rightarrow \Prim(H), [L] \mapsto L(\1)$.
\end{lemma}
\begin{proof}
	First note
	\begin{align*}
		\counit \circ L 
		&= (\counit \tp \counit) \circ \Delta \circ L 
		= (\counit \tp \counit) \circ \left[ (\id \tp L) \circ \Delta + L \tp \1 \right] \\
		&= \counit \circ \big[ L \circ \underbrace{(\counit \tp \id) \circ \Delta}_{\id} + L \big]
		= 2 \counit \circ L,
	\end{align*}
	delivering the first assertion.	Further,
	\begin{equation*}
		\Delta \circ L(\1) = (\id \tp L) (\1 \tp \1) + L(\1) \tp \1 = L(\1) \tp \1 + \1 \tp L(\1)
	\end{equation*}
	implies the second and it only remains to check that for any $\alpha \in H'$
	\begin{equation*}
		\partial \alpha (\1) = (\id \tp \alpha) \circ \Delta (\1) - \alpha(\1) \tp \1 = \1 \alpha(\1) - \1 \alpha(\1) = 0. \qedhere
	\end{equation*}
\end{proof}

\begin{korollar}\label{satz:B_+-cocycle}
	In $H_R$, the grafting operator $B_+$ is a non-trivial Hochschild-1-cocycle $B_+ \in HZ^1_{\counit}(H_R)$ by \eqref{eq:B_+-cocycle} and $0\neq [B_+] \in HH^1_{\counit}(H_R)$ as $\widetilde{\ev}_{\1}([B_+]) = B_+(\1) = \tree{+-} \neq 0$.
\end{korollar}

\subsection{The universal property of \texorpdfstring{$H_R$}{H\_R}}
The Hopf algebra $H_R$ of rooted trees fulfils a universal property described in
\begin{satz}\label{satz:H_R-universal}
	Let $\alg$ be any commutative unital algebra and $L \in \End(\alg)$. Then there exists a unique morphism $\unimor{L}\!: H_R \rightarrow \alg$ of unital algebras such that
	\begin{equation}
		\unimor{L} \circ B_+ = L \circ \unimor{L},
		\quad \text{equivalently} \quad
		\vcenter{\xymatrix{
		{H_R} \ar[r]^{\unimor{L}} \ar[d]_{B_+} & {\alg} \ar[d]^{L} \\
			{H_R} \ar[r]_{\unimor{L}} & {\alg}
		}}
		\quad \text{commutes.}
		\label{eq:H_R-universal}
	\end{equation}
	Moreover, if $\alg$ is a bialgebra and $L \in HZ^1_{\counit}(\alg)$, then $\unimor{L}$ is a morphism of coalgebras (and thus bialgebras) as well. Should $\alg$ further allow for an antipode, then $\unimor{L}$ is even a morphism of Hopf algebras.
\end{satz}

\begin{proof}
	Uniqueness follows inductively: Starting with $\unimor{L}(\1) = \1_{\alg}$, suppose $\unimor{L}$ to be uniquely defined on $\forests_n$ for some $n\in\N_0$. Then \eqref{eq:H_R-universal} fixes $\unimor{L} (t) = L \circ \unimor{L}(f)$ for any tree $t=B_+(f)$ with $n+1$ nodes ($f \in \forests_n$) and thus also for any general forest $f\in\forests_{n+1}$ by
	\begin{equation*}
		\unimor{L}(f)
		= \unimor{L} \left( \prod_{t\in \pi_0(f)} t \right)
		= \prod_{t \in \pi_0(f)} \unimor{L} (t). \tag{$\ast$}
	\end{equation*}
	%, by $\big( \unimor{L} \tp \unimor{L} \big) \circ m_{\alg} = \unimor{L} \circ m_H$ we also have $\unimor{L}$ nailed down for any forest of $n+1$ nodes.
	Now given uniqueness of $\unimor{L}$ on $\forests$ and by linearity also on $H_R$ as a whole, the above procedure at the same time serves as an inductive definition proving existence.\footnote{Note that we need the commutativity of $\alg$ here, otherwise the right-hand side of $(\ast)$ is not well defined!}
	
	If $L$ is a cocycle, by lemma \ref{satz:cocycle-props} we have $\unimor{L} (\lin \trees) = \unimor{L} (\im B_+) \subseteq L\big(\im \unimor{L} \big) \subseteq \ker \counit_{\alg}$. Hence by multiplicativity of $\counit_{\alg}$ we find $\counit_{\alg} \circ \unimor{L} (f) = 0 = \counit_{H_R} (f)$ for any forest $f \in \forests\setminus\set{\1}$ using $(\ast)$. Together with $\counit_{\alg} \circ \unimor{L} (\1) = \counit_{\alg} (\1_{\alg}) = 1 = \counit_{H_R} (\1)$, we obtain
	\begin{equation*}
		\counit_{\alg} \circ \unimor{L} = \counit_{H_R}.
	\end{equation*}
	To prove that $\unimor{L}$ is morphism of (counital) coalgebras it thus remains to show
	\begin{equation*}
		\Delta_{\alg} \circ \unimor{L} = \left( \unimor{L} \tp \unimor{L} \right) \circ \Delta_{H_R},	
		\tag{$\sharp$}
	\end{equation*}
	which can be done inductively on trees again (as both sides of $(\sharp)$ are algebra morphisms)! So suppose $(\sharp)$ holds for all forests $f \in \forests_n$ (start of the induction is trivial at $f=\1$), then it also holds on $\forests_{n+1}$ as for any tree $t=B_+ (f)$ with $f\in\forests_n$ we have
	\settowidth{\wurelwidth}{\eqref{eq:H_R-universal}}
	\begin{align*}
		\Delta_{\alg} \circ \unimor{L} (t)
		&\wurel{} \Delta_{\alg} \circ \unimor{L} \circ B_+ (f)
		\urel{\eqref{eq:H_R-universal}} \Delta_{\alg} \circ L \circ \unimor{L} (f)
		\urel{\eqref{bicomodule:d1}} \left[ (\id_{\alg} \tp L) \circ \Delta_{\alg} + L \tp \1_{\alg} \right] \circ \unimor{L} (f) \\
		&\wurel{} (\id_{\alg} \tp L) \circ \Delta_{\alg} \circ \unimor{L} (f) + \big[ L \circ \unimor{L} (f) \big] \tp \1_{\alg} \\
		&\wurel{$(\sharp)$} \big[\unimor{L} \tp ( L \circ \unimor{L} ) \big] \circ \Delta_{H_R} (f) + \big[ \unimor{L} \circ B_+ (f) \big] \tp \unimor{L} (\1) \\
		&\wurel{\eqref{eq:H_R-universal}} \big( \unimor{L} \tp \unimor{L} \big) \circ \left[ (\id_{H_R} \tp B_+) \circ \Delta_{H_R} + B_+ \tp \1 \right] (f) \\
		&\wurel{\eqref{eq:B_+-cocycle}} \big( \unimor{L} \tp \unimor{L} \big) \circ \Delta_{H_R} \circ B_+ (f)
		= \big( \unimor{L} \tp \unimor{L} \big) \circ \Delta_{H_R} (t).
	\end{align*}
	For the Hopf algebra case we analogously prove \mbox{$S_{\alg} \circ \unimor{L} = \unimor{L} \circ S$} inductively on trees, again exploiting that both sides in this equation are algebra morphisms\footnote{in general only antimorphisms, but $\alg$ is commutative}. Suppose it is true on $\forests_n$ (trivial start at $n=0$), then for any tree $t=B_+(f)$ with $f\in\forests_n$
	\settowidth{\wurelwidth}{\eqref{eq:antipode-cocycle}}
	\begin{align*}
		S_{\alg} \circ \unimor{L} (t)
		&\wurel{} S_{\alg} \circ \unimor{L} \circ B_+ (f)
		\urel{\eqref{eq:H_R-universal}} S_{\alg} \circ L \circ \unimor{L} (f)
		\urel{\eqref{eq:antipode-cocycle}} - (S_{\alg} \convolution L) \circ \unimor{L} (f) \\
		&\wurel{$(\sharp)$} - (S_{\alg} \circ \unimor{L}) \convolution (L \circ \unimor{L} ) (f)
		\urel{\eqref{eq:H_R-universal}} - (\unimor{L} \circ S) \convolution (\unimor{L} \circ B_+) (f)
		\urel{$\ (\natural)\ $} - \unimor{L} \circ (S \convolution B_+) (f) \\
		&\wurel{\eqref{eq:antipode-cocycle}} \unimor{L} \circ S \circ B_+ (f)
		 = \unimor{L} \circ S (t).
	\end{align*}
	Here we employed that $\unimor{L}$ is an algebra morphism in $(\natural)$ and used the induction hypothesis \mbox{$S_{\alg} \circ \unimor{L} = \unimor{L} \circ S$} in the step before, on the left-hand side of the convolution product (note \mbox{$\Delta (f) \in \forests_n \tp \forests_n$}). The helpful equation 
	\begin{equation}\label{eq:antipode-cocycle}
		S\circ L = - S \convolution L,
	\end{equation}
	valid for the antipode $S$ of an arbitrary Hopf algebra and any 1-cocycle $L\in HZ^1_{\counit}(H)$, follows from lemma \ref{satz:cocycle-props} by 
	\begin{align*}
		u \circ \underbrace{\counit \circ L}_0
		&= (S \convolution \id) \circ L
		 = m \circ (S \tp \id) \circ \Delta \circ L 
		 = m \circ (S \tp \id) \circ \left[ ( \id \tp L) \circ \Delta + L \tp \1 \right]\nonumber\\
		&= m \circ (S \tp L) \circ \Delta + (S \circ L) \cdot \1
		 = S \convolution L + S \circ L.
		 \qedhere
	\end{align*}
\end{proof}
This natural algebra morphism $\unimor{L}$ is easily understood as follows: Any element of $H_R$ is a unique linear combination of forests, which are themselves expressible as iterations of $B_+$ and $m$, applied to $\1$, in a unique way. As an example consider
\begin{equation*}
	\tree{++-+--} + 3 \tree{++--} - \tree{+-}
	= B_+ \left( {\left[B_+(\1) \right]}^2 \right) + 3 B_+ \circ B_+ (\1) - B_+ (\1).
\end{equation*}
Now $\unimor{L}\!:\ H_R \rightarrow \alg$ just replaces every $B_+$ by $L$, $m$ by $m_{\alg}$ and $\1$ by $\1_{\alg}$! Hence we find
\begin{equation*}
	\unimor{L} \left( \tree{++-+--} + 3\tree{++--} - \tree{+-} \right)
	= L \left(  {\left[ L(\1_{\alg}) \right]}^2 \right) + 3 L \circ L (\1_{\alg}) - L(\1_{\alg}).
\end{equation*}
In this sense, $H_R$ is the free commutative algebra that is generated by a generic endomorphism (represented through $B_+$)! One should think of $B_+$ just as a placeholder for a specific endomorphism $L$.

More precisely, consider a category whose objects are pairs $(\alg, L_{\alg})$ of commutative unital algebras $\alg$ and endomorphisms $L_{\alg} \in \End(\alg)$, whereas the morphisms from $(\alg[A], L_{\alg[A]})$ to $(\alg[B], L_{\alg[B]})$ are given by morphisms $\rho\!:\ \alg[A]\rightarrow \alg[B]$ of unital algebras such that $\rho \circ L_{\alg[A]} = L_{\alg[B]} \circ \rho$.
Then the only elements of $\alg$ that we can naturally construct from $(\alg, L_{\alg})$ are the linear combinations of
\begin{equation*}
	\1_{\alg},\ L(\1_{\alg}),\ { \left[ L(\1_{\alg}) \right] }^2,\ L \circ L (\1_{\alg}),\ \ldots,
\end{equation*}
the results of iterations of $m_{\alg}$ and $L_{\alg}$ applied to the distinguished element $\1_{\alg}$.
By definition, in the case of the object $(H_R, B_+)$ all those elements (each corresponding to a forest) are linearly independent. In particular, theorem \ref{satz:H_R-universal} just shows that $(H_R, B_+)$ is the \emph{initial object} in the category just described! It further proves that it remains an initial object for the subcategories where we restrict to bialgebras (Hopf algebras) $\alg$, Hochschild cocycles $L_{\alg}$ and bialgebra (Hopf algebra) morphisms $\rho$.

Analogously, in the case of general unital algebras (not necessarily commutative ones) the initial object is the Hopf algebra of \emph{planar} (or \emph{ordered}) rooted trees. The distinguished total order among the children of a node prescribes the order in which the multiplication is to be performed.

These considerations are special instances of the much more general theory of \emph{operads}, for instance see \cite{Chapoton:Operads}.

\subsection{Automorphisms of \texorpdfstring{$H_R$}{H\_R}}
In particular we may apply the universal property to $\alg \defas H_R$ as target algebra itself, but with a cocycle $L \neq B_+$. Naturally we can modify $B_+$ by coboundaries leading to

\begin{definition}\label{def:auto}
	For any $\alpha \in H_R'$, using theorem \ref{satz:H_R-universal} we define
	\begin{equation}
		\auto{\alpha} \defas \unimor{B_+ + \partial\alpha}\!:\ H_R \rightarrow H_R
		\label{eq:auto}
	\end{equation}
	to be the unique algebra endomorphism of $H_R$ such that
		$\auto{\alpha} \circ B_+ = \left[ B_+ + \partial \alpha \right] \circ \auto{\alpha}$.
\end{definition}
	For instance, using the coproducts in section \ref{sec:coproduct} and \eqref{eq:hochschild-coboundary} to evaluate $\partial\alpha$, check
\begin{align}
	\auto{\alpha} \left( \tree{+-} \right) 
		&= \auto{\alpha} \circ B_+ (\1)
		 = B_+ (\1) + (\partial\alpha) (\1) 
		 = B_+ (\1)
		 = \tree{+-}
		 \label{auto:()}\\
%\nonumber\\
	\auto{\alpha} \left( \tree{++--} \right)
		&= \auto{\alpha} \circ B_+ \left( \tree{+-} \right)
		 = \left( B_+ + \partial\alpha \right) \auto{\alpha} \left( \tree{+-} \right)
%		 = \left[ B_+ + \partial\alpha \right] \left( \tree{+-} \right)
		 = \tree{++--} + \partial\alpha \left( \tree{+-} \right)
		 = \tree{++--} + \alpha(\1) \tree{+-}
		 \label{auto:(())}\\
%\nonumber\\
	\auto{\alpha} \left( \tree{+++---} \right)
		&= \auto{\alpha} \circ B_+ \left( \tree{++--} \right)
		 = \left( B_+ + \partial\alpha \right) \left[ \tree{++--} + \alpha(\1)\tree{+-} \right]
		 = \tree{+++---} + \alpha(\1) \tree{++--} + \partial\alpha \left( \tree{++--} \right) + \alpha(\1) \partial\alpha\left( \tree{+-} \right) \nonumber\\
		&= \tree{+++---} + 2 \alpha(\1) \tree{++--} + \left\{ {\left[ \alpha(\1) \right]}^2 + \alpha\left( \tree{+-} \right) \right\} \tree{+-}
		\label{auto:((()))} \\
	\auto{\alpha} \left( \tree{++-+--} \right)
		&= \auto{\alpha} \circ B_+ \left( \tree{+-}\tree{+-} \right)
		 = \left( B_+ + \partial\alpha \right) \left[ \auto{\alpha}(\tree{+-})\auto{\alpha}(\tree{+-}) \right]
		 = \left[ B_+ + \partial\alpha \right] \left( \tree{+-}\tree{+-} \right) \nonumber\\
		&= \tree{++-+--} + \partial\alpha \left( \tree{+-}\tree{+-} \right) 
		 = \tree{++-+--} + 2 \alpha \left( \tree{+-} \right) \tree{+-} + \alpha (\1) \tree{+-}\tree{+-}.
		 \label{auto:(()())}
\end{align}
Now arises the natural question of how the morphisms $\unimor{L}$ induced by \eqref{eq:H_R-universal} change under a variation of the cocycle $L$ by a coboundary. We give the answer in

\begin{satz}\label{satz:change-coboundary-equals-auto}
	Let $H$ be any commutative bialgebra, $L \in HZ^1_{\counit}(H)$ a 1-cocycle and further $\alpha \in H'$ a functional. Then for $\unimor{L},\unimor{L+\partial \alpha}\!:\ H_R \rightarrow H$ given through theorem \ref{satz:H_R-universal} and $\auto{\alpha \circ \unimor{L}}\!:\ H_R \rightarrow H_R$ from definition \ref{def:auto}, we have
	\begin{equation}
		\unimor{L + \partial \alpha} = \unimor{L} \circ \auto{\left[\alpha\, \circ\, \unimor{L}\right]},
		\quad\text{equivalently}\quad
		\vcenter{\xymatrix{
			{H_R} \ar[r]^{\unimor{L + \partial\alpha}} \ar[d]_{\auto{\alpha \circ \unimor{L}}}  & {H} \\
			{H_R} \ar[ur]_{\unimor{L}} &
			}}
		\quad\text{commutes.}
		\label{eq:change-coboundary-equals-auto}
	\end{equation}
\end{satz}

\begin{proof}
	As both sides of \eqref{eq:change-coboundary-equals-auto} are algebra morphisms, it is sufficient to prove it inductively on trees: Let it be true for a forest $f\in\forests$ (start the induction at $f=\1$), then it holds as well for the tree $ B_+(f)$ by
	\settowidth{\wurelwidth}{\eqref{eq:H_R-universal}}
	\begin{align*}
		\unimor{L} \circ \auto{\left[\alpha \circ \unimor{L}\right]} \circ B_+ (f)
		&\wurel{\eqref{eq:H_R-universal}} \unimor{L} \circ \left[ B_+ + \partial \left( \alpha \circ \unimor{L} \right) \right] \circ \auto{\left[\alpha \circ \unimor{L}\right]} (f) \\
		&\wurel{} \bigg\{ \underbrace{\unimor{L} \circ B_+}_{L \circ \unimor{L}}
		 					+\underbrace{\unimor{L} \circ \left[
								                         \partial \left( \alpha \circ \unimor{L} \right)
																				 \right]}_{(\partial \alpha) \circ \unimor{L}}
			 \bigg\} \circ \auto{\left[\alpha \circ \unimor{L}\right]} (f) \\
		&\wurel{} \left\{ L + \partial \alpha \right\} \circ 
		 			\underbrace{
							\unimor{L} \circ \auto{\left[\alpha \circ \unimor{L}\right]} (f)
					}_{
							\unimor{L + \partial\alpha} (f)
					}
		\urel{\eqref{eq:H_R-universal}} \unimor{L + \partial\alpha} \circ B_+ (f).
	\end{align*}
	Here we used the identity
	\begin{align*}
		(\partial\alpha) \circ \phi
		&= (\id_B \tp \alpha) \circ \Delta_B \circ \phi - \1_B \cdot \alpha \circ \phi 
		 = (\id_B \tp \alpha) \circ (\phi \tp \phi) \circ \Delta_A - \phi(\1_A) \cdot \alpha \circ \phi \\
		&= \phi \circ \Big\{ [\id_A \tp (\alpha \circ \phi) ] \circ \Delta_A - \1_A \cdot \alpha \circ \phi \Big\}
		 = \phi \circ \partial \left( \alpha \circ \phi \right),
	\end{align*}
	valid for arbitrary bialgebra morphisms $\phi\!:\ A \rightarrow B$ and $\alpha \in B'$, applied to $\phi \defas \unimor{L}$.
\end{proof}

Examples for this result are worked out in section \ref{sec:more-feynman-rules}. Generally, theorem \ref{satz:change-coboundary-equals-auto} says that the effect of twisting $L$ by $\partial\alpha$ on the resulting morphism from theorem \ref{satz:H_R-universal} can completely be restored and understood on the side of $H_R$ alone, by the endomorphism $\auto{\alpha \circ \unimor{L}}$.  This is in fact an automorphism of $H_R$ as shown in

\begin{satz}
	The map $\auto{\cdot}\!:\ H_R' \rightarrow \End_{\text{Hopf}}(H_R)$, taking values in the space of Hopf algebra endomorphisms of $H_R$, fulfils the following properties:
	\begin{enumerate}
		\item For $\alpha\in H_R'$ and any forest $f\in\forests$, $\auto{\alpha}(f)$ differs from $f$ only by lower order forests:
			\begin{equation}
				\forall f\in\forests\!: \quad \auto{\alpha} (f) \in f + H_R^{\abs{f}-1}
				= f + \bigoplus_{n=0}^{\abs{f}-1} H_{R,n}.
				\label{eq:auto-leading-term}
			\end{equation}
		\item $\auto{\cdot}$ maps $H_R'$ into $\Aut_{\text{Hopf}} (H_R)$, the group of Hopf algebra automorphisms.
		\item The automorphisms of this form are closed under composition, saying that for any \mbox{$\alpha,\beta \in H_R'$} there exists a $\gamma \in H_R'$ with $\auto{\alpha} \circ \auto{\beta} = \auto{\gamma}$. Concretely we can take
			\begin{equation}
				\gamma = \alpha + \beta \circ {\auto{\alpha}}^{-1}.
				\label{eq:auto-composition}
			\end{equation}
		\item The maps $\partial^0\!:\ H_R' \rightarrow HZ^1_{\counit}(H_R)$ and $\auto{\cdot}\!:\ H_R' \rightarrow \Aut_{\text{Hopf}}(H_R)$ are injective.
		\item Therefore the subgroup $\im \auto{\cdot} = \setexp{\auto{\alpha}}{\alpha \in H_R'} \subset \Aut_{\text{Hopf}}(H_R)$ induces a group structure on $H_R'$ with neutral element $0$ and group law
			\begin{equation}
				\alpha \autoconc \beta 
				\defas \auto{\cdot}^{-1} \left( \auto{\alpha} \circ \auto{\beta} \right)
				= \alpha + \beta \circ \auto{\alpha}^{-1} ,\quad \alpha^{\autoconc -1} = -\alpha \circ \auto{\alpha}.
				\label{eq:auto-group}
			\end{equation}
	\end{enumerate}
\end{satz}

\begin{proof}
	For examples to \eqref{eq:auto-leading-term} see \eqref{auto:()}, \eqref{auto:(())} and \eqref{auto:(()())}. The general proof is done inductively starting with $\auto{\alpha}\left( \tree{+-} \right) = \tree{+-}$. So suppose \eqref{eq:auto-leading-term} holds for some forests $f, f'\in \forests$, then it does so for $f \cdot f'$ too by
	\begin{equation*}
		\auto{\alpha} (f\cdot f')
		= \auto{\alpha}(f) \cdot \auto{\alpha} (f')
		\in \left( f + H_R^{\abs{f}-1} \right)\cdot \left( f' + H_R^{\abs{f'}-1} \right)
		\subseteq f\cdot f' + H_R^{\abs{f\cdot f'}-1}.
	\end{equation*}
	Using $\partial\alpha (H_R^n) \subseteq H_R^n$, we further achieve \eqref{eq:auto-leading-term} for the tree $B_+(f)$ by
	\begin{equation*}
		\auto{\alpha} \circ B_+(f)
		=\left[ B_+ + \partial\alpha \right] \circ \auto{\alpha} (f)
		\subseteq \left[ B_+ + \partial\alpha \right] \left( f + H_R^{\abs{f}-1} \right)
		\subseteq B_+(f) + H_R^{\abs{f}}.
	\end{equation*}
	This proves inductively \eqref{eq:auto-leading-term} and also the surjectivity of $\auto{\alpha}$: Starting with $H_R^0 = \K \cdot \1 \subseteq \im \auto{\alpha}$, suppose that $H_R^n \subseteq \im \auto{\alpha}$ for some $n\in\N_0$. Then taking any forest $f \in \forests_{n+1}$ we just proved
	\begin{equation*}
		\auto{\alpha} (f)
		\in f + H_R^n,
		\quad \text{therefore} \quad
		f \in \auto{\alpha}(f) + H_R^n \subseteq \im \auto{\alpha}.
	\end{equation*}
	For the injectivity of $\auto{\alpha}$ suppose $0 \neq x \in \ker \auto{\alpha}$ with $x = \bigoplus_{n\in\N_0} x_n$ using homogeneous components $x_n\in H_{R,n}$. Take $N\in\N_0$ such that $x_n=0\ \forall n>N$ and $x_N \neq 0$, then
	\begin{equation*}
		0 = \auto{\alpha} (x) = \sum_{n=0}^{N} \auto{\alpha} (x_n)
		\urel[\in]{\eqref{eq:auto-leading-term}} \underbrace{\sum_{n=0}^{N-1} \auto{\alpha}(x_n)}_{\subseteq H_R^{N-1}} + x_{N} + H_R^{N-1}
	\end{equation*}
	implies the contradiction $x_N \subseteq H_R^{N-1}$, thus we must have $\ker \auto{\alpha}=\set{0}$. The bijectivity of $\auto{\alpha}$ and thus $\auto{\alpha}\in\Aut_{\text{Hopf}}(H_R)$ being proven, the inverse $\auto{\alpha}^{-1}$ in \eqref{eq:auto-composition} is well defined and we can apply \eqref{eq:change-coboundary-equals-auto} to get \eqref{eq:auto-composition} as
	\begin{align*}
		\auto{\left[\alpha+\beta \,\circ\, \auto{\alpha}^{-1}\right]}
		&= \unimor{\left[B_+ + \partial\alpha \right] + \partial\left( \beta \,\circ\, \auto{\alpha}^{-1} \right)}
		\urel{\eqref{eq:change-coboundary-equals-auto}} \unimor{\left[B_+ + \partial\alpha\right]} \circ \auto{\big[\beta \ \circ\ \auto{\alpha}^{-1} \ \circ\ \unimor{\left(B_+ + \partial\alpha\right)}\big]} \\
		&= \auto{\alpha} \circ \auto{\left[\beta \ \circ\ \auto{\alpha}^{-1} \ \circ\ \auto{\alpha}\right]}
		= \auto{\alpha} \circ \auto{\beta}.
	\end{align*}
	Now consider $\alpha,\beta \in H_R'$ with $\auto{\alpha} = \auto{\beta}$, then
	
%	To see injectivity of $\partial^0\!:\ H_R' \rightarrow HZ^1_{\counit}(H_R)$, suppose $\auto{\alpha} = \id$ and observe $\partial\alpha = 0$ by
	\begin{equation*}
		0
		= (\auto{\alpha} - \auto{\beta}) \circ B_+
		= (B_+ + \partial\alpha) \circ \auto{\alpha} - (B_+ + \partial\beta) \circ \auto{\beta}
		= (\partial\alpha - \partial\beta) \circ \auto{\alpha}
	\end{equation*}
	implies $\partial\alpha = \partial\beta$ by the surjectivity of $\auto{\alpha}$. Hence the injectivity of $\auto{\cdot}$ reduces to that of $\partial^0\!:\ H_R' \rightarrow HZ^1_{\counit}(H_R)$. In contrast to $\auto{\cdot}$, this map is linear and we only need to consider $\partial\alpha = 0$. First check how for any $n\in\N_0$, $\alpha\left( {\tree{+-}}^n \right) = 0$ follows from
	\begin{equation*}
		\forall n\in\N_0\!: \quad
		0
		= \partial\alpha \left( {\tree{+-}}^{n+1} \right)
		= \sum_{i=0}^n \binom{n+1}{i} \alpha \left( {\tree{+-}}^i \right) {\tree{+-}}^{n+1-i}.
	\end{equation*}
	Given an arbitrary forest $f\in\forests$ and $n\in\N$, the expression
	\begin{align*}
		0 
		&= \partial\alpha \left( {\tree{+-}}^n f \right)
		= f \underbrace{\alpha\left( {\tree{+-}}^n \right)}_0 + \sum_f \sum_{i=0}^n \binom{n}{i}{\tree{+-}}^i f' \alpha\left( {\tree{+-}}^{n-i} f'' \right) 
		  + \sum_{i=1}^n \binom{n}{i} \bigg[ {\tree{+-}}^i f \underbrace{\alpha\left( {\tree{+-}}^{n-i} \right)}_0 + {\tree{+-}}^i \alpha\left( f {\tree{+-}}^{n-i} \right) \bigg]
	\end{align*}
	simplifies upon projection onto $\K \tree{+-}$ to
	\begin{equation*}
		\alpha\left( f {\tree{+-}}^{n-1} \right)
		= - \frac{1}{n}\sum_{\substack{f \\ f'=\tree{+-}}} \alpha\left( {\tree{+-}}^n f'' \right).
	\end{equation*}
	Iteration of this formula allows us to express $\alpha(f)$ as a linear combination of the values $\setexp{\alpha\big( {\tree{+-}}^k \big)}{ 0 \leq k \leq \abs{f}}$. These all vanish, hence we proved $\alpha = 0$ and therefore the injectivity of $\auto{\cdot}\!: \ H_R' \rightarrow \Aut_{\text{Hopf}}(H_R)$.

	Therefore we can pull the group structure from $\Aut_{\text{Hopf}}(H_R)$ back onto $H_R'$, resulting in \eqref{eq:auto-group} as a consequence of \eqref{eq:auto-composition}.
\end{proof}

\section{Decorated rooted trees}
Though we will not need the Hopf algebra $H_R(\decor)$ of decorated\footnote{Simply speaking, we just add a label drawn from the set $\decor$ to every node of a tree (or forest). In particular, the grafting operator $B_+^d$ now carries an index $d\in\decor$ specifying this label for the new root it attaches to a forest.} rooted trees (with decorations drawn from a set $\decor$) in the sequel, we still want to remark briefly that the above results generalize (along with their proofs) immediately to the decorated setup:
\begin{itemize}
	\item Let $\alg$ be a commutative algebra and $L_{\cdot}\!:\ \decor \rightarrow \End(\alg)$ a $\decor$-indexed set of endomorphisms. Then there exists a unique algebra morphism $\unimor{L_{\cdot}}: H_R(\decor) \rightarrow \alg$ such that
		\begin{equation}
			\forall d \in \decor\!: \quad \unimor{L_{\cdot}} \circ B_+^d = L_d \circ \unimor{L_{\cdot}},
			\label{eq:H_R-dekoriert-universal}
		\end{equation}
		which turns out to be a morphism of bialgebras if $\alg$ is a bialgebra and $\im L_{\cdot} \subseteq HZ^1_{\counit}(\alg)$ are 1-cocycles. Finally, if $\alg$ is even a Hopf algebra, then $\unimor{L_{\cdot}}$ is a morphism of Hopf algebras.
	\item Given a family $\alpha_{\cdot}\!:\ \decor \rightarrow H_R'(\decor)$ of functionals on $H_R(\decor)$, one can consider the family $\decor\ni d\mapsto L_d \defas B_+^d + \partial \left( \alpha_d \right)$ (shorthand notation $L_{\cdot} \defas B_+^{\cdot} + \partial \alpha_{\cdot}$) of cocycles and obtains an endomorphism $\auto{\alpha_{\cdot}} \defas \unimor{B_+^{\cdot} + \partial\alpha_{\cdot}}$ of $H_R(\decor)$ via \eqref{eq:H_R-dekoriert-universal}. It is an automorphism of Hopf algebras and it does not change the terms of leading weight.
	\item The automorphisms of this kind are closed under composition and induce a group structure on $\bigoplus_{d \in \decor} H_R'(\decor)$ by
		\begin{equation*}
			{\left[ \alpha_{\cdot} \autoconc \beta_{\cdot} \right]}_d
			= \alpha_d + \beta_d \circ \unimor{\alpha_{\cdot}}^{-1} 
			= \alpha_d + \beta_d \circ \unimor{\left[- \alpha_{\cdot} \circ \unimor{\alpha_{\cdot}}\right]} \quad \forall d \in \decor.
		\end{equation*}
	\item If $H$ is a bialgebra, $L_{\cdot}\!:\ \decor \rightarrow HZ^1_{\counit}(H)$ a $\decor$-indexed set of cocycles and $\alpha_{\cdot}\!: \decor \rightarrow H'$ a set of functionals, the effect of twisting the $L_{\cdot}$ by the coboundaries $\partial\alpha_{\cdot}$ is captured by an automorphism of $H_R(\decor)$ via
		\begin{equation*}
			\unimor{L_{\cdot} + \partial\alpha_{\cdot}}
			= \unimor{L_{\cdot}} \circ \auto{\left[\alpha_{\cdot} \circ \unimor{L_{\cdot}}\right]},
		\end{equation*}
		just as in the case \eqref{eq:change-coboundary-equals-auto} without decorations.
\end{itemize}

\section{The Hopf algebra of polynomials}
\label{sec:polynomials}
In the next chapter, we will encounter polynomials as the target algebra of renormalized Feynman rules. Hence it is worth to study
\begin{definition}
	The symmetric algebra $S(V)=\bigoplus_{n=0}^{\infty} S^n(V)$ of a vector space $V$ is a graded Hopf algebra, defined by setting $\Delta (v) = v\tp \1 + \1 \tp v$ for any $v\in V$.
\end{definition}
\begin{korollar}
	$S(V)$ is connected, commutative and cocommutative. For any $v\in V$ and $n\in\N_0$ we have $\Delta \left( v^n \right) = \sum_{i=0}^n \binom{n}{i} v^i \tp v^{n-i}$. Further note $\Prim \left( S(V) \right) = V$.
\end{korollar}
By the Milnor-Moore-Theorem (see \cite{MilnorMoore}) it turns out, that in fact every connected, commutative and cocommutative Hopf algebra is isomorphic to the symmetric algebra over its primitive elements.
\begin{lemma}
	In one variable ($\dim V = 1$), $S(V) \isomorph {\K}[x]$ comes along with a natural Hochschild 1-cocycle
\begin{equation}
	\int_0 \in HZ^1_{\counit} (\K[x]),\quad
	f \mapsto \int_0 f \defas \left[ x \mapsto \int_0^x f(y)\ \dd y \right].
	\label{eq:int-cocycle}
\end{equation}
\end{lemma}
\begin{proof}
	For any $n\in\N_0$ consider the monomial $\frac{x^n}{n!} \in \K[x]$ and observe
	\begin{align*}
		\Delta \int_0 \left( \frac{x^n}{n!} \right)
		&= \Delta \left( \frac{x^{n+1}}{(n+1)!} \right)
		 = \frac{1}{(n+1)!} \sum_{k=0}^{n+1} \binom{n+1}{k} x^k \tp x^{n+1-k} \\
		&= \sum_{k=0}^{n+1} \frac{x^k}{k!} \tp \frac{x^{n+1-k}}{(n+1-k)!} = \frac{x^{n+1}}{(n+1)!} \tp \1 + \sum_{k=0}^n \frac{x^k}{k!} \tp \int_0 \left(\frac{x^{n-k}}{(n-k)!} \right) \\
		&= \int_0 \left( \frac{x^n}{n!} \right) \tp \1 + \left( \id \tp \int_0 \right) \bigg[ \frac{1}{n!} \underbrace{\sum_{k=0}^n \binom{n}{k} x^k \tp x^{n-k}}_{\Delta (x^n)} \bigg] \\
		&= \left[ \int_0 \tp \1 + \left(\id \tp \int_0 \right) \circ \Delta \right] \left( \frac{x^n}{n!} \right). \qedhere
	\end{align*}
\end{proof}
As for $B_+$ and $H_R$, we easily check that $\int_0$ is not a coboundary by $\int_0 1 = x \neq 0$. In fact, $\int_0$ is essentially the only nontrivial cocycle through
\begin{satz}
	The space of Hochschild 1-cocycles in $\K[x]$ decomposes into
	\begin{equation}
		HZ^1_{\counit}(\K[x]) = \K \cdot \int_0 \ \oplus\ \underbrace{\partial \left( \K[x]' \right)}_{HB^1_{\counit}(\K[x])}.
		\label{eq-polys-cycles}
	\end{equation}
	In particular this implies $HH^1_{\counit} (\K[x]) = \K \cdot [ \int_0 ]$ being one-dimensional.
\end{satz}

\begin{proof}
	First consider a cocycle $L \in HZ^1_{\counit} (\K[x])$ with $L(1) = 0$. We prove inductively the existence of a functional $\alpha \in \K[x]'$ with $L(x^n) = (\partial \alpha)(x^n)$ for all $n\leq N$. The start at $N=0$ is trivial as $L(1) = 0 = \left( \partial \alpha \right) (1)$ for any $\alpha$, which also implies
	\begin{equation*}
%		\forall n \leq N: \quad L \left( x^n \right) = \partial \alpha \left( x^n \right)
		\cored \circ L = (\id \tp L)\circ \cored
		\quad\text{and}\quad
		\cored (\partial\alpha) = \left[ \id \tp (\partial\alpha) \right] \circ \cored
	\end{equation*}
	by \eqref{bicomodule:d1}. Assuming the induction hypothesis for $N\in \N_0$,
	\begin{align*}
		   \tilde{\Delta} \circ L \left( x^{N+1} \right) 
		&= \sum_{i=1}^N \binom{N+1}{i} x^{N+1-i} \tp L \left( x^i \right) \\
		&= \sum_{i=1}^N \binom{N+1}{i} x^{N+1-i} \tp \partial\alpha \left( x^{i} \right) 
		 = \tilde{\Delta} \circ \partial \alpha \left( x^{N+1} \right)
	\end{align*}
	implies $(L - \partial \alpha) \left( x^{N+1} \right) \in \Prim (\K[x]) = \K \cdot x$. Thus let $\lambda \in \K$ be the scalar such that $(L-\partial\alpha) \left( x^{N+1} \right) = \lambda \cdot x$ and adjust the functional to
	\begin{equation*}
		\forall n \in \N_0: \quad \alpha' \left( x^n \right) \defas 
		\begin{cases}
			\alpha \left( x^n \right) & \text{if $n \neq N$,}\\
			\alpha\left( x^n \right) + \frac{\lambda}{N+1} & \text{if $n=N$,}\\
		\end{cases}
	\end{equation*}
	resulting in $L\left( x^{N+1} \right) = \partial \alpha' \left( x^{N+1} \right)$ by construction. For all $n\leq N$ we further maintain $L(x^n) = (\partial\alpha)(x^n) = (\partial \alpha') (x^n) $ as $(\partial\alpha')(x^n)$ only depends on the values $\alpha' (x^r)$ for $r \leq n < N$, where $\alpha'(x^r)=\alpha(x^r)$. Inductively we see how each value $\alpha(x^n)$ is determined uniquely.

	This finishes the proof that any $L\in HZ^1_{\counit}(\K[x])$ with $L(1) = 0$ lies in $HB^1_{\counit}(\K[x])$. Considering an arbitrary cocycle $L$, lemma \ref{satz:cocycle-props} ensures $L(1) \in \Prim(\K[x]) = \K \cdot x$. So $L(1) = \lambda \cdot x = \lambda \int_0 1$ for some $\lambda \in \K$ and $L= \lambda \int_0 \oplus (L-\lambda \int_0)$ with $(L - \lambda \int_0)(1) = 0$.
\end{proof}

\subsection{Characters}
As a character $\phi \in \chars{{\K}[x]}{\K}$ is uniquely determined by its value $\phi (x)$, we immediately see
\begin{lemma}
	The characters of the polynomial algebra $\K[x]$ are the evaluations
	\begin{equation}
		\chars{{\K}[x]}{\K} = \setexp{\ev_{\lambda}}{\lambda \in \K},
		\label{eq:poly-characters}
	\end{equation}
	sending a polynomial $p(x) \in \K[x]$ to $\ev_{\lambda} (p) \defas p(\lambda)$.
\end{lemma}
Given that $\K[x]$ is a connected Hopf algebra, by \eqref{satz:character-group} these characters actually form a group under the convolution product. Concretely we have
\begin{lemma}
	The group structure on $\chars{\K[x]}{\K}$ induced by convolution is
	\begin{equation}
		\forall a,b \in \K\!: \quad \ev_a \convolution \ev_b = \ev_{a+b}.
		\label{eq:poly-character-group}
	\end{equation}
\end{lemma}
\begin{proof}
	Simply calculate for any $a,b \in \K$ and $n\in\N_0$
	\begin{align*}
		\left[ \ev_a \convolution \ev_b \right] \left( x^n \right) 
		&= { \left[ \left( \ev_a \tp \ev_b \right) \circ \Delta (x) \right] }^n
		= { \left[ \ev_a(\1) \cdot \ev_b(x) + \ev_a(x) \cdot \ev_b(\1) \right] }^n  \\
		&= { ( b + a ) }^n
		= {\left[ \ev_{a+b} (x) \right] }^n 
		= \ev_{a+b} \left( x^n \right). \qedhere
	\end{align*}
\end{proof}

\chapter{A detailed example: Kreimer's toy model}
\label{chap:toymodel}

The Hopf algebra $H_R$ of rooted trees, introduced in section \ref{sec:H_R}, turns out to be sufficient to formulate and understand renormalization problems in any quantum field theory. It precisely models the structure of both nested and disjoint subdivergences occurring in multidimensional integrals, while the famous remaining problem of overlapping divergences is resolved into a linear combination of rooted forests (see \cite{Kreimer:Overlapping}).

In his works, Dirk Kreimer employed several setups of Feynman rules defined on this Hopf algebra to serve as illuminating examples. We investigate one of those in this chapter to familiarize the reader with the concept of renormalization and its algebraic properties.

First of all, we see how Feynman rules may be defined on $H_R$ in a natural way utilizing the universal property \eqref{eq:H_R-universal} of $H_R$.
As a special case of this construction, we define Kreimer's toy model in section \ref{sec:toymodel}. Note that we will not consider Feynman rules originating from \emph{iterated integrals}, another setup occurring in some of his papers. The physical origin of the toy model is lined out in section \ref{sec:toymodel-origin}: a brief look at {\qft} exhibits it as the sub sector of iterated propagator insertions!\footnote{The techniques presented here can be generalized to the full renormalization problem of {\qft}, see \cite{Yeats}.}

As it is typical for {\qft}, the naive toy model is ill-defined as such and needs a \emph{regularization}. We use \emph{analytic regularization}, although in theorem \ref{satz:finiteness} we show how this choice of regulator is irrelevant for the (physical limit of the) renormalized results in the employed scheme.

We renormalize the regularized toy model in section \ref{sec:renormalization}, as prescribed by the Birkhoff decomposition using the momentum scheme. Our results from the previous chapter allow for a complete combinatoric description of the full renormalized Feynman rules in section \ref{sec:toymodel-rules-by-cocycles}, after taking the physical limit. Amazingly, they turn out to be of the simplest kind we studied in section \ref{sec:more-feynman-rules}!

After exploiting this special structure to obtain the reduction of higher order to first order contributions as in \eqref{eq:toymodel-higher-orders}, we have a short look on the behaviour of the renormalized \emph{correlation function} of the toy model, arising from combinatoric \emph{Dyson-Schwinger equations}. These considerations culminate in the \emph{renormalization group} in section \ref{sec:rge} and non-perturbative approaches.

A final remark shall be made on a problem of the generalization of the steps performed to {\qft}, originating from \emph{higher degrees of divergence}.

\section{Construction of Feynman rules}
We want to define \emph{Feynman rules}, a synonym for morphisms $\phi\!:\ H_R \rightarrow \alg$ of unital algebras to some commutative target algebra $\alg$. As $H_R$ is free commutative as an algebra, the space of such Feynman rules is isomorphic to the space of maps $\trees \rightarrow \alg$ via restriction. Thus the most general Feynman rules can take arbitrary values on trees.

However, physical Feynman rules are much less arbitrary and in fact of a very special algebraic flavour. All examples we will study arise from \eqref{eq:H_R-universal}, that is they obey
\begin{equation*}
	\phi \circ B_+ = L \circ \phi
\end{equation*}
for some $L \in \End(\alg)$. It is precisely this special form of Feynman rules that ensures the physically relevant properties like finiteness and locality of renormalization, as will be discussed in section \ref{sec:finiteness}.

The simplest rules of this kind originate from $\alg \defas \K$, then $L \in \End(\K) = \K \cdot \id$ multiplies by some constant $a \in \K$ and it is easy to check that for any forest
\begin{equation}
	\forall f\in \forests\!: \quad \unimor{a \cdot \id}(f) = a^{\abs{f}},
	\label{eq:scalar-rules}
\end{equation}
essentially counting the nodes. Here, $\unimor{a\cdot\id}$ is the morphism from theorem \ref{satz:H_R-universal}.

\subsection{External parameters}
In quantum field theory, the functions assigned to the combinatoric objects (trees or graphs) typically depend on a finite number of \emph{external parameters} like momenta of the external particles of a Feynman graph. Therefore, $\alg$ is in general some algebra of functions of those parameters. As the simplest example for a single parameter consider
\begin{lemma}
	Let $\alg \defas {\K}[x]$ and $\int_0 \in \End(\alg)$ from \eqref{eq:int-cocycle} induce Feynman rules
	\begin{equation*}
		\intrules \defas \unimor{\int_0}\!:\ H_R \rightarrow {\K}[x],
		\quad \intrules \circ B_+ = \int_0 \ \circ\ \intrules
	\end{equation*}
	through the universal property \eqref{satz:H_R-universal}. Then for any forest $f\in \forests$ we have 
	\begin{equation}\label{eq:int-rules}
		\intrules (f) = \frac{x^{\abs{f}}}{f!}.
	\end{equation}
\end{lemma}

\begin{proof}
	The inductive proof (start at $f=\1$) supposes \eqref{eq:int-rules} to be true for all $f\in \forests_{\leq n}$. Then \eqref{eq:int-rules} also holds for any true forest $f\in\forests_{n+1}\setminus\trees$ (so $\abs{\pi_0 (f)} > 1$) as
	\begin{equation*}
		  \intrules (f) 
		= \prod_{t \in \pi_0 (f)} \intrules (t) 
		= \prod_{t \in \pi_0 (f)} \frac{x^{\abs{t}}}{t!} 
		= \frac{x^{\sum_{t\in \pi_0(f)} \abs{t}}}{\prod_{t \in \pi_0(f)} t!} 
		= \frac{x^{\abs{f}}}{f!},
	\end{equation*}
	exploiting $\abs{t} \leq n$ for any $t\in\pi_0(f)$ to use the induction hypothesis. It remains to consider a tree $t = B_+ (f)$ for some $f\in \forests_n$ in
	\begin{equation*}
		\intrules (t)
		= \intrules \circ B_+ (f) 
		= \int_0 \circ \:\intrules (f) 
		= \int_0^x \frac{y^{\abs{f}}}{f!} \ \dd y
		= \frac{x^{\abs{f}+1}}{(\abs{f}+1) \cdot f!} 
		= \frac{x^{\abs{B_+ (f)}}}{\left( B_+ f \right)!}
		= \frac{x^{\abs{t}}}{t!}. \qedhere
	\end{equation*}
\end{proof}
Note that these Feynman rules are very special as they provide not only a morphism of algebras, but rather a morphism of Hopf algebras (we took $L=\int_0$ to be a cocycle). Defining the evaluated characters
\begin{equation}
	\forall a \in \K\!:\quad
	\intrules_a \defas \ev_a \circ \intrules\!:\ H_R \rightarrow \K, f \mapsto \restrict{\intrules(f)}{a},
	\label{eq:int-rules-char}
\end{equation}
which are \emph{not} of the basic form \eqref{eq:scalar-rules} as for the additional factors $\frac{1}{f!}$, we obtain the remarkable
\begin{proposition}\label{satz:int-rules-rge}
	\begin{equation}
		\forall a,b \in \K:\quad \intrules_a \convolution \intrules_b = \intrules_{a + b}.
		\label{eq:int-rules-rge}
	\end{equation}
\end{proposition}
For example consider
\settowidth{\wurelwidth}{\eqref{coproduct:(()())}} 
\begin{align*}
	\intrules_a \convolution \intrules_b \left( \tree{++-+--} \right) 
	&\wurel{\eqref{coproduct:(()())}} \intrules_a \left( \tree{++-+--} \right) + 2 \intrules_a \left( \tree{+-} \right) \intrules_b \left( \tree{++--} \right) + \intrules_a \left( \tree{+-}\tree{+-} \right) \intrules_b \left( \tree{+-} \right) + \intrules_b \left( \tree{++-+--} \right) \\
	&\wurel{\eqref{eq:int-rules}} \frac{a^3}{3} + 2 a \frac{b^2}{2} + a^2 b + \frac{b^3}{3} 
	 = \frac{(a+b)^3}{3}
	 \urel{\eqref{eq:int-rules}} \intrules_{a+b} \left( \tree{++-+--} \right),
\end{align*}
demonstrating an amazing compatibility of the Feynman rules with the combinatoric structure of trees through the convolution product.
Essentially, \eqref{eq:int-rules-rge} is the \emph{renormalization group equation} in the momentum scheme (see sections \ref{sec:higher-orders} and \ref{sec:rge}).

Using \eqref{eq:int-rules} and \eqref{eq:poly-character-group}, proposition \ref{satz:int-rules-rge} is an immediate corollary of
\begin{lemma}\label{satz:convolution-rge}
	Let $\phi: C \rightarrow D$ be a morphism of coalgebras and $\alg$ an algebra. Then
	\begin{equation}
		\tilde{\phi}\!:
		\quad \convalg{D}{\alg} \rightarrow \convalg{C}{\alg},
		\quad \alpha \mapsto \alpha \circ \phi
		\label{eq:convolution-morphism}
	\end{equation}
	is a morphism of associative unital algebras:
	\begin{equation*}
		\forall \alpha,\beta \in \convalg{D}{\alg}\!: \quad
		\tilde{\phi} (\alpha) \convolution \tilde{\phi} (\beta) = \tilde{\phi} (\alpha \convolution \beta).
	\end{equation*}
	If $C$ and $D$ are bialgebras and $\phi$ a morphism of such, then $\tilde{\phi}$ maps characters $\alpha \in \chars{D}{\alg}$ to characters $\alpha \circ \phi \in \chars{C}{\alg}$. In particular we obtain a homomorphism $\chars{D}{\alg} \rightarrow \chars{C}{\alg}$ of groups for connected $C$, $D$ and commutative $\alg$ by lemma \ref{satz:convolution-group}.
\end{lemma}

\begin{proof}
	We only have to check that for any $\alpha,\beta \in \convalg{D}{\alg}$,
	\begin{align*}
		  (\alpha \circ \phi) \convolution (\beta \circ \phi) 
		&= m_{\alg} \circ \left[ (\alpha \circ \phi) \tp (\beta \circ \phi) \right] \circ \Delta_C 
		 = m_{\alg} \circ (\alpha \tp \beta) \circ (\phi \tp \phi) \circ \Delta_C \\
		&= m_{\alg} \circ (\alpha \tp \beta) \circ \Delta_D \circ \phi 
		 = (\alpha \convolution \beta) \circ \phi. \qedhere
	\end{align*}
\end{proof}

\subsection{On symmetry factors}
It is interesting to note that among the vast pool of possible Feynman rules defined by 
\begin{equation*}
	\forall t\in\trees\!: \quad
	\rho(t) = a_t x^{\abs{t}},
\end{equation*}
{\qft} chooses those very special ones arising through \eqref{eq:H_R-universal}. In the above example, this special choice of $a_t = \frac{1}{t!}$ leads to the strong result \eqref{eq:int-rules-rge} -- clearly this group law fails unless the coefficients $a_t$ fulfill a plethora of combinatoric relations!

The occurrence of purely combinatoric factors like $\frac{1}{t!}$ in \eqref{eq:int-rules} is a typical feature of Feynman rules in {\qft}, where every graph comes along with a peculiar \emph{symmetry factor}. Those are determined combinatorially for each Feynman graph of perturbation theory. Though they can cause confusion to the beginner and might appear as a nuisance in calculations, their correct treatment is of outmost importance!

On one hand we just observed how the precise combinatorics of coefficients is crucial to achieve algebraic relations like \eqref{eq:int-rules-rge}.
On the other hand, naive renormalization calculations\footnote{An example of this cancellation of non-local divergences is worked out in section 5.2 of \cite{Collins}.} in {\qfts} produce counterterms (at intermediate steps) that depend non-locally on external parameters. Amazingly all those non-local contributions cancel in the overall counterterms obtained through the renormalization recursion to be described in section \ref{sec:renormalization}. Clearly, this property of \emph{locality} depends crucially on the precise combinatorics and relations among the counterterms.

Phrased differently, the precise form of symmetry factors allows to identify the counterterm Lagrangian based method of renormalization and the graph-by-graph procedure\footnote{Section 5.6 of \cite{Collins} gives an outline of the proof.}, which is described by the Birkhoff decomposition (see \cite{CK:RH1}) and can be proved to provide locality as in \cite{Factorization}.

As it will turn out in the sequel, it is precisely the fact that the Feynman rules $\phi$ fulfil the universal property $\phi \circ B_+ = \psi \circ \phi$ for some linear operator $\psi \in \End(A)$ that allows for inductive proves of finiteness and locality.
Luckily, the Feynman rules of physical {\qfts} are precisely of this form, with $\psi$ denoting a loop integral over insertions of subdivergences into a primitive skeleton graph.

Finally note that in the Hopf algebraic approach to renormalization using the Birkhoff decomposition, the symmetry factors are not included into the Feynman rules as explained in \cite{CK:RH1}. Their share of combinatorics is introduced afterwards through combinatorial Dyson-Schwinger equations (see section \ref{sec:DSE}).

\subsection{Variation of lower order terms by coboundaries}
\label{sec:more-feynman-rules}
Instead of considering the cocycle $\int_0$ from above, one might instead twist it by some coboundary as in
\begin{definition}
	For any functional $\alpha \in \K[x]'$ consider the cocycle
	\begin{equation}
		\int_0 + \partial \alpha\!:\ x^n \mapsto \frac{x^{n+1}}{n+1} + \sum_{k < n} \binom{n}{k} \alpha \left( x^k \right) x^{n-k}
		\label{eq:int-rules-cocycle}
	\end{equation}
	and call $\intrules[\alpha] \defas \unimor{\int_0 + \partial\alpha}\!:\ H_R \rightarrow \K[x]$ the Feynman rules induced by \eqref{eq:H_R-universal}.
\end{definition}
These Feynman rules differ from $\intrules$, but they still arise from a cocycle and hence enjoy \eqref{eq:int-rules-rge} upon evaluation at different parameter values as well!

To understand the effect of $\partial\alpha$, set $\alpha_n \defas \alpha(x^n)$ and consider the examples
\begin{align}
	\intrules[\alpha] \left( \tree{+-} \right)
		&= \left[ \int_0 + \partial\alpha \right] (1)
		 = \int_0 1
		 = x
		 = \intrules \left( \tree{+-} \right) \\
	\intrules[\alpha] \left( \tree{++--} \right)
		&= \left[ \int_0 + \partial\alpha \right] (x)
		 = \frac{x^2}{2} + \alpha_0 \cdot x
		 = \intrules \left\{ \tree{++--} + \alpha(1) \tree{+-} \right\} \\
	\intrules[\alpha] \left( \tree{+++---} \right)
		&= \left[ \int_0 + \partial\alpha \right] \left( \frac{x^2}{2} + \alpha_0 \cdot x \right)
		 = \frac{x^3}{6} + \alpha_0 \cdot \frac{x^2}{2} + \alpha_0 \cdot \frac{x^2}{2} + \alpha_1 \cdot x + \alpha_0^2 \cdot x \nonumber\\
		&= \intrules \left\{ \tree{+++---} + 2\alpha_0 \tree{++--} + \left[ \alpha_1 + \alpha_0^2 \right] \tree{+-} \right\}
		\label{int-boundary:((()))}\\
	\intrules[\alpha] \left( \tree{++-+--} \right)
		&= \left[ \int_0 + \partial\alpha \right] \left( x^2 \right)
		 = \frac{x^3}{3} + \alpha_0 \cdot x^2 + 2\alpha_1 \cdot x
		 =\intrules \left\{ \tree{++-+--} + \alpha_0 \tree{+-}\tree{+-} + 2\alpha_1 \tree{+-} \right\}
		\label{int-boundary:(()())}
\end{align}
Apparently the \emph{leading terms} of $\intrules[\alpha]$ and $\intrules$ match, which follows from $\intrules[\alpha] = \intrules \circ \auto{\left(\alpha\, \circ\, \intrules\right)}$ by \eqref{eq:change-coboundary-equals-auto} together with \eqref{eq:auto-leading-term} in full generality!
%This is of particular relevance towards the distinction between \emph{angle variables} and the \emph{energy scale} in \qft, see \cite{}.
In the examples \eqref{int-boundary:((()))} and \eqref{int-boundary:(()())}, using \eqref{auto:((()))} and \eqref{auto:(()())} we check
\begin{align*}
	\auto{\alpha\, \circ\, \intrules} \left( \tree{++-+--} \right)
		&= \tree{++-+--} + 2 \left[ \alpha \circ \intrules \left( \tree{+-} \right) \right] \tree{+-} + \left[ \alpha \circ \intrules (\1) \right] \tree{+-}\tree{+-}
		 = \tree{++-+--} + 2 \alpha_1 \tree{+-} + \alpha_0 \tree{+-}\tree{+-} \\
	\auto{\alpha\, \circ\, \intrules} \left( \tree{+++---} \right)
		&= \tree{+++---} + 2 \left[ \alpha \circ \intrules(\1) \right] \tree{++--} + \left\{ {\left[ \alpha \circ \intrules (\1) \right] }^2 +  \left[ \alpha \circ \intrules \left( \tree{+-} \right) \right] \right\} \tree{+-} \\
		&= \tree{+++---} + 2\alpha_0 \tree{++--} + \left[ \alpha_0^2 + \alpha_1 \right] \tree{+-}
\end{align*}
and thus indeed verify $\intrules[\alpha] \left( \tree{++-+--} \right) = \intrules \circ \auto{\left(\alpha\, \circ\, \intrules\right)} \left( \tree{++-+--} \right)$ and $\intrules[\alpha] \left( \tree{+++---} \right) = \intrules \circ \auto{\left(\alpha\, \circ\, \intrules\right)} \left( \tree{+++---} \right)$ explicitly.

Amazingly enough, it will turn out in section \ref{sec:toymodel-rules-by-cocycles} that the physical limit of the renormalized Feynman rules of Kreimer's toy model are of this very simple kind if one uses the momentum scheme!

\section{Kreimer's toy model of \qft}
\label{sec:toymodel}
So far we only focused on Feynman rules that suggest themselves naturally on a purely algebraic level. However, those occurring in {\qft} are in general of a different kind: They map into a target algebra $\alg = C^{\infty}\big( (0,\infty) \big)$ of rather general functions (in our case depending on a single external parameter, the \emph{scale} $s$) and arise through the universal property \eqref{eq:H_R-universal} by means of integrations like
	\begin{equation}
		\phi_s \circ B_+
		= \int_0^{\infty} \frac{\phi_{\zeta}}{s + \zeta}\: \dd\zeta,
		\quad\text{explicitly}\quad
		\phi \circ B_+
		= \left[ s \mapsto \int_0^{\infty} \frac{\phi_{\zeta}}{s+\zeta} \dd\zeta \right].
%		= \int_0^{\infty} \frac{\phi_{s \zeta}}{\zeta + 1}\: d\zeta
		\label{eq:physical-rules}
	\end{equation}
	The physicist will recognize the striking resemblance of $\int_0^{\infty} \dd\zeta$ to an integration over a \emph{loop momentum} $\zeta$ and the similarity of the integral kernel $\frac{1}{\zeta+s}$ to a \emph{propagator} like\footnote{As will become clear in section \ref{sec:toymodel-origin}, we should rather compare to massless propagators and consider $\frac{1}{(p+k)^2}$ or $\frac{1}{\slashed{p} + \slashed{k}}$ instead, where the external and internal momenta $p$ and $k$ correspond to $s$ and $\zeta$.}
	$\frac{1}{p^2 + m^2}$ or $\frac{1}{\slashed{p} + m}$.
	Indeed, as we show in section \ref{sec:toymodel-origin}, rules similar to \eqref{eq:physical-rules} do actually arise in {\qft}! This motivates their study in the sequel.

However, as it stands \eqref{eq:physical-rules} is not well defined as already the integral for $\phi_s (\tree{+-})$ is logarithmically divergent! This is a typical feature of {\qfts} and in fact the reason why renormalization is necessary at all: The Feynman rules are naturally given as divergent integrals over well-defined integrands. Making sense out of these is precisely the subject of renormalization!

\subsection{Analytic regularization}
\label{sec:regularization}
Hence we have to first find a way of quantifying the divergences and give \eqref{eq:physical-rules} a precise meaning. This process is called \emph{regularization} and can be performed in many different ways.

We focus on divergences of the integral originating solely from the unboundedness of the integration domain\footnote{%
This is ensured in {\qft}: Employing \emph{Wick rotation}, the theory is transformed into Euclidean space, eliminating all propagator poles. We will not discuss infrared divergences here.%
}, that is, we exclude singularities in the integrand itself. Then we can achieve a finite value for the integral in

\begin{definition}
	The Feynman rule of the toy model is defined as the \emph{analytic regularization} of \eqref{eq:physical-rules} through theorem \ref{satz:H_R-universal} and
	\begin{equation}
		\toy_s \circ B_+ 
		= \int_0^{\infty} \frac{\zeta^{-\reg}}{\zeta + s}\ \toy_{\zeta} \ \dd\zeta 
		= \int_0^{\infty} \frac{(s\zeta)^{-\reg}}{\zeta+1}\ \toy_{s\zeta} \ \dd\zeta.
		\label{eq:toymodel}
	\end{equation}
	More generally, we allow to replace $\frac{1}{\zeta+1}$ in the integrand by another suitably regular function $f(\zeta) \in \bigo{\zeta^{-1}}$ for $\zeta \rightarrow \infty$.\footnote{Here we restrict ourselves to the case of logarithmically divergent integrands.}
\end{definition}
So $\toy[]$ maps into an algebra of functions $\toy_s$ of both $s$ and $\reg$. Note that this is a general phenomenon: any regularization procedure introduces a new (artificial) parameter\footnote{This is not to be confused with the mass scale $\mu$ needed in dimensional regularization. In fact, we need the same in analytic regularization if we treat $s$ and $\zeta$ to be dimensional: $\left( \frac{\zeta}{\mu} \right)^{-\reg}$ only makes sense for a dimensionless base.} (here $\reg$).

Clearly these rules give finite values for $\Re\reg>0$ and if $f$ is analytic, they will even be meromorphic in $\reg$ around zero (however, check remark 3. below). In fact, by defining the \emph{Mellin transform}\footnote{In standard notation the Mellin transform is $\{\mathcal{M}f\}(\reg) \defas \int_0^{\infty} \zeta^{\reg-1} f(\zeta)\ \dd\zeta$, thus $F(\reg) \defas \{\mathcal{M}f\}(1-\reg)$.}
\begin{equation}
	F(\reg) \defas
	\int_0^{\infty} f(\zeta) \zeta^{-\reg} \ \dd\zeta
	= \sum_{n=-1}^{\infty} \coeff{n} {\reg}^n,
	\label{eq:mellin-trafo}
\end{equation}
we obtain a meromorphic function $F \in \reg^{-1}\K[[\reg]]$ with a pole of first order at $\reg\rightarrow 0$ that captures all analytic information in a combinatorial manner by
\begin{proposition}
	For any forest $t \in \mathcal{F}$ we have\footnote{Compare with \eqref{eq:tree-factorial-product} to see how this generalizes the tree factorial!}
	\begin{equation}
		\toy_{s} (t) = s^{-\reg\abs{t}} \prod_{v \in V(t)} F \left( \reg \abs{t_v} \right).
		\label{eq:toymodel-mellin}
	\end{equation}
\end{proposition}

\begin{proof}
	As both sides of \eqref{eq:toymodel-mellin} are clearly multiplicative, it is enough to prove the claim inductively for trees. Let it be valid for some forest $x\in\forests$, then for $t = B_+ (x)$ observe
	\begin{align*}
		\toy_{s} \circ B_+ (x)
		&= \int_0^{\infty} (s\zeta)^{-\reg} f(\zeta)\ \toy_{s\zeta} (x) \ \dd\zeta
		 = \int_0^{\infty} (s\zeta)^{-\reg} f(\zeta) (s\zeta)^{-\reg\abs{x}} \prod_{v \in V(x)} F \left( z \abs{x_v} \right) \ \dd\zeta \\
		&= s^{-\reg \left( 1+ \abs{x} \right)} \prod_{v \in V(x)} F \left( \reg \abs{x_v} \right) \int_0^{\infty} f(\zeta) \zeta^{-\reg \left( 1 + \abs{x} \right)} \ \dd\zeta \\
		&= s^{-\reg \abs{B_+ (x)}} \left[ \prod_{v \in V(x)} F \left( \reg \abs{x_v} \right) \right] F \left( \reg \abs{B_+ (x)} \right)
		= s^{-\reg \abs{t}} \prod_{v \in V(t)} F \left( \reg \abs{ t_v } \right) \qedhere.
	\end{align*}
\end{proof}
This result allows us to calculate the Feynman rules algebraically (without having to perform any integrations). As examples consider
\begin{align*}
	\toy_s \left( \tree{+-} \right)
	&= s^{-\reg} F(\reg) &
	\toy_s \left( \tree{++--} \right)
	&= s^{-2\reg} F(\reg) F(2\reg) \\
	\toy_s \left( \tree{+++---} \right)
	&= s^{-3\reg} F(\reg) F(2\reg) F(3\reg) &
	\toy_s \left( \tree{++-+--} \right)
	&= s^{-3\reg} {\left[ F(\reg) \right]}^2 F(3\reg).
\end{align*}
For the original toy model $f(\zeta) = \frac{1}{1+\zeta}$ we obtain
\begin{align}
	F(\reg)
	&= \int_0^{\infty} \frac{\zeta^{-\reg}}{\zeta+1} \ \dd\zeta
	= B(\reg, 1-\reg)
	= \Gamma(\reg) \Gamma(1-z)
	= \frac{\pi}{\sin (\pi \reg)}\\
	&= \frac{1}{\reg} \Gamma(1+\reg)\Gamma(1-\reg)
	= \frac{1}{\reg} \exp \left\{ \sum_{n \in \N} \frac{\zeta(2n)}{n} z^{2n} \right\},
	\label{eq:mellin-prop}
\end{align}
with Euler's beta function
\begin{equation*}
	B(x,y)
	\defas \int_0^1 t^{x-1} {(1-t)}^{y-1}\ \dd t
	= \int_0^{\infty} \dd\alpha\ \alpha^{x-1} \int_0^{\infty} \dd\beta\ \beta^{y-1}\ \delta(1-\alpha-\beta)
	= \frac{\Gamma(x)\Gamma(y)}{\Gamma(x+y)}
\end{equation*}
and his gamma function $\Gamma(\reg) = \int_0^{\infty} x^{\reg-1} e^{-x}\ \dd x$.
\hide{
\begin{align*}
	\toy_s \left( \tree{+-} \right)
	&= s^{-\reg} B(\reg, 1-\reg) \\
	\toy_s \left( \tree{++--} \right)
	&= s^{-2\reg} B(\reg, 1-\reg) B(2\reg, 1-2\reg) \\
	\toy_s \left( \tree{+++---} \right)
	&= s^{-3\reg} B(\reg, 1-\reg) B(2\reg, 1-2\reg) B(3\reg, 1-3\reg) \\
	\toy_s \left( \tree{++-+--} \right)
	&= s^{-3\reg} {B(\reg, 1-\reg)}^2 B(3\reg, 1-3\reg)
\end{align*}
}
A couple of remarks are in order:
\begin{enumerate}
	\item Whereas the original integral \eqref{eq:toymodel} converges only for $\Re \reg > 0$, the analyticity of $F$ allows for an analytic continuation in $\reg$ into some region where $\Re \reg < 0$. The same happens in the popular dimensional regularization, see section \ref{sec:toymodel-dimreg}.
	\item Notice how due to the presence of a single scale only, the dependency on the external parameter $s$ is very simple and given through the plain power $s^{-\reg\abs{f}}$.
	\item Denoting the radius of convergence of the Laurent series \eqref{eq:mellin-trafo} by $r$, we obtain $\frac{r}{\abs{t}}$ as the radius of convergence of the Laurent series of $\toy_s(t)$ by considering the contribution $F(\reg\cdot\abs{t})$ in \eqref{eq:toymodel-mellin} for a tree $t\in\trees$. In particular we can not find a $\reg \neq 0$ such that $\toy_s(f)$ converges as Laurent series for every $f \in H_R$!
	\item The highest order pole of $\toy_s (f)$ comes from multiplying all the poles of $F$'s, so
		\begin{align}
			\toy_s (f)
			&\in s^{-\reg\abs{f}} \prod_{v \in V(f)} \left\{ \frac{\coeff{-1}}{\reg\abs{f_v}} + \K[[\reg]] \right\} 
			\subset \Big( 1 + \K[[\reg\ln s]] \Big) \left\{ \prod_{v \in V(f)} \frac{\coeff{-1}}{\reg\abs{f_v}} + \reg^{1-\abs{f}}\K[[\reg]] \right\} \nonumber\\
			&= \frac{1}{f!} {\left( \frac{\coeff{-1}}{\reg} \right)}^{\abs{f}} + \reg^{1-\abs{f}} \K[[\reg,\reg \ln s]].
			\label{eq:toymodel-leading-pole}
		\end{align}
		Hence the leading divergence of $\toy_s(f)$ is independent of $s$ and given by the tree factorial \eqref{eq:toymodel-leading-pole}.
\end{enumerate}

\section{Renormalization of the toy model}
\label{sec:renormalization}
\subsection{General concept of renormalization}
As seen in the previous chapter, the mathematical concept of Birkhoff decomposition allows for a decomposition of Feynman rules $\phi$ into \emph{renormalized} rules $\phi_R \defas \phi_+$ and the so-called \emph{counterterms} $Z \defas \phi_-$. Recall the \emph{Bogoliubov character} (also \emph{$\bar{R}$-operation})
\begin{equation*}
	\bar{\phi} (x) \defas
	\phi(x) + \sum_x \phi_- (x') \phi(x'') 
	= \phi(x) + [ \phi_- \convolution \phi - \phi_- - \phi ](x) 
	= \phi_+(x) - \phi_-(x)
\end{equation*}
we already encountered in \eqref{eq:rbar}. It defines the Birkhoff decomposition recursively by \eqref{eq:birkhoff-dec} and essentially renormalizes the \emph{subdivergences} as we shall see in theorem \ref{satz:subdivergences}.

Obviously the result of renormalization depends crucially on how \mbox{$\alg = \alg_+ \oplus \alg_-$} is splitted (specifying the projection $R\!: \alg \twoheadrightarrow \alg_-$). This is called the choice of a \emph{renormalization scheme}.

We already mentioned the \emph{minimal subtraction scheme} \eqref{eq:ms-splitting}, which is applicable to the toy model as we regularized it to deliver meromorphic functions in $\reg$. However, we will study the \emph{momentum scheme} instead as it is far better behaved algebraically (see chapter \ref{chap:conclusion} for details).

\subsection{Momentum scheme}
For massive\footnote{%
In massless theories, naive application of the momentum scheme typically introduces infrared divergences (see section 3.6.4 in \cite{Collins}).}
theories, the momentum scheme is a very convenient and (as it will turn out) algebraically distinguished renormalization scheme. It is based upon the following observation: Consider the rational function
\begin{equation*}
	f(\zeta, s) \defas \frac{1}{\zeta+s}
\end{equation*}
occurring in the integrand of \eqref{eq:toymodel}. For fixed $s$ we have $f(\zeta,s) \in \landautheta{\zeta^{-1}}$ and hence the integral $\toy_s(\tree{+-})$ is $\log$-divergent. However, for an arbitrary \emph{subtraction point} $\rp$,
\begin{equation}
	f(\zeta,s) - f(\zeta,\rp)
	= \frac{1}{\zeta+s} - \frac{1}{\zeta+\rp}
	= \frac{\rp - s}{(\zeta+s)(\zeta+\rp)}
	\label{eq:momsch-example}
\end{equation}
lies in $\bigo{\zeta^{-2}}$ whence \mbox{$\int_0^{\infty} \left[ f(\zeta,s) - f(\zeta, \rp) \right] \ \dd\zeta$} is convergent! This generalizes to higher degrees of divergence (section \ref{sec:higher-sdds}), though we will for now only be considering the $\log$-divergent case.

\begin{definition}
	On the target algebra $\alg$ of regularized Feynman rules depending on a single external variable $s$, define the \emph{momentum scheme} by evaluation at $s=\rp$:
	\begin{equation}
		\End(\alg) \ni \momsch{\rp} \defas \ev_{\rp} = \left( \alg \ni f \mapsto {\left. f \right|}_{s=\rp} \right).
		\label{eq:momentum-scheme}
	\end{equation}
\end{definition}
In particular, $\momsch{\rp}$ is a character of $\alg$ and we thus may use \eqref{eq:birkhoff-character}! We define the counterterm \mbox{$\toyZ \defas {\left( \toy \right)}_- = \momsch{\rp} \circ \toy \circ S = \toy_{\rp} \circ S$} and the renormalized Feynman rules $\toyR \defas {\left( \toy \right)}_+$ via the Birkhoff decomposition induced by $\momsch{\rp}$. Note that the counterterms $\toyZ$ do not depend on $s$ and we suppress the dependency on $\reg$ in the notation.

This results in the following values for the first trees:
\begin{align}
	\toyZ \left( \tree{+-} \right)
		&= \momsch{\rp} \circ \toy \circ S \left( \tree{+-} \right) 
		 = \ev_{\rp} \circ \toy \left( -\tree{+-} \right)
		 = - \toy_{\rp} \left( \tree{+-} \right)
		 = - \rp^{-\reg} F(\reg)
		 \nonumber\\
	 \toyR[s] \left( \tree{+-} \right)
		&= \toy_s \left( \tree{+-} \right) + \toyZ \left( \tree{+-} \right)
		 = \left( s^{-\reg} - \rp^{-\reg} \right) F(\reg)
		 \label{toyR:()}\\
	\toyZ \left( \tree{++--} \right)
		&= \toy_{\rp} \circ S \left( \tree{++--} \right)
		 = \toy_{\rp} \left( -\tree{++--} + \tree{+-}\tree{+-} \right)
		 = - \rp^{-2\reg}F(\reg)F(2\reg) + \rp^{-2\reg}F^2(\reg) \nonumber\\
		&= \rp^{-2\reg} F(\reg) \left[ F(\reg) - F(2\reg) \right] 
		\nonumber\\
	\toyR[s] \left( \tree{++--} \right)
		&= \toy_s \left( \tree{++--} \right) + \toyZ \left( \tree{+-} \right) \toy_s \left( \tree{+-} \right) + \toyZ \left( \tree{++--} \right) \nonumber\\
	&= \left( s^{-2\reg} - \rp^{-2\reg} \right) F(\reg)F(2\reg) - \left( s^{-\reg} - \rp^{-\reg} \right) \rp^{-\reg} F^2(\reg)
	\label{toyR:(())}
\end{align}

\section{The physical limit}
In the step of regularization, we introduced the artificial (non-physical) parameter $\reg$ into the Feynman rules $\toy_s$. The goal of renormalization is to take the \emph{physical limit}
\begin{equation}
	\toyphy \defas \lim_{\reg \rightarrow 0} \toyR
	\label{eq:toymodel-physical}
\end{equation}
of the renormalized Feynman rules $\toyR$, corresponding to the situation without a regulator (the original theory). The notation $\toyphy$ is unambiguous as $\reg=0$ only makes sense for the renormalized Feynman rules.
Expanding \eqref{toyR:()} and \eqref{toyR:(())} in $\reg$ we obtain
\begin{align}
	\toyphy[s] \left( \tree{+-} \right)
		&= \lim_{\reg \rightarrow 0} \left[ \left( s^{-\reg} - \rp^{-\reg} \right) F(\reg) \right]
		 = \lim_{\reg \rightarrow 0} \left[ \left( -\reg \ln \tfrac{s}{\rp} + \bigo{\reg^2} \right) \cdot \left( \tfrac{\coeff{-1}}{\reg} + \bigo{\reg^0} \right) \right] \nonumber\\
		&= - \coeff{-1} \ln \tfrac{s}{\rp}
		\quad\text{and}
		\label{toyR-physical:()}\displaybreak[0]\\
	\toyphy[s] \left( \tree{++--} \right)
		&= \lim_{\reg \rightarrow 0} \left\{ \left[ -2\reg \ln \tfrac{s}{\rp} + 2\reg^2\left( \ln^2 s - \ln^2 \rp \right) + \bigo{z^3} \right]\cdot \left[ \tfrac{\coeff[2]{-1}}{2\reg^2} + \tfrac{3\coeff{0} \coeff{-1}}{2\reg} + \bigo{\reg^0} \right] \right. \nonumber\\
		&\quad \left. - \left[ -\reg \ln \tfrac{s}{\rp} + \tfrac{\reg^2}{2} \left( \ln^2 s + 2\ln s \ln\rp - 3\ln^2 \rp \right) + \bigo{\reg^3} \right] \cdot \left[ \tfrac{\coeff[2]{-1}}{\reg^2} + 2\tfrac{\coeff{-1} \coeff{0}}{\reg} + \bigo{\reg^0} \right] \right\} \nonumber\displaybreak[0]\\
		&= \frac{\coeff[2]{-1}}{2} \ln^2 \tfrac{s}{\rp} - \coeff{-1} \coeff{0} \ln \tfrac{s}{\rp}.
		\label{toyR-physical:(())}
\end{align}
These calculations obviously become increasingly lengthy, we just state
\begin{align}
	\toyphy[s] \left( \tree{+++---} \right)
		&= - \frac{\coeff[3]{-1}}{6} \ln^3 \tfrac{s}{\rp} + \coeff[2]{-1} c_0 \ln^2 \tfrac{s}{\rp} - \coeff{-1} \left( \coeff[2]{0} + \coeff{-1} \coeff{1} \right) \ln \tfrac{s}{\rp} 
		\label{toyR-physical:((()))}\\
	\toyphy[s] \left( \tree{++-+--} \right)
		&= -\frac{\coeff[3]{-1}}{3} \ln^3 \tfrac{s}{\rp} + \coeff[2]{-1} \coeff{0} \ln^2 \tfrac{s}{\rp} - 2\coeff[2]{-1}\coeff{1} \ln \tfrac{s}{\rp}
		\label{toyR-physical:(()())}
\end{align}
and observe that $\toyphy[s]$ takes values in $\K[\ln \tfrac{s}{\rp}]$, mapping any forest $f$ to a polynomial in $\ln \tfrac{s}{\rp}$ of degree $\abs{f}$. Note that we have no constant parts, except for $\toyphy(\1) = 1$. So far we did not prove the existence of the limit $\toyphy$ at all (the reader should check how in \eqref{toyR-physical:(())} the contributions $\propto \reg^{-1}$ cancel inside the limit)! We will remedy this in the following two sections resulting in proposition \ref{satz:finiteness}.

\subsection{Renormalization of subdivergences}
\label{sec:subdivergences}
\begin{satz}\label{satz:subdivergences}
	Let $H$ be a connected bialgebra, $\alg$ a commutative algebra with an endomorphism $L \in \End(\alg)$ and consider the Feynman rules $\phi \defas \unimor{L}$ induced by \eqref{eq:H_R-universal}. Given a renormalization scheme $R \in \End(\alg)$ such that
	\begin{equation}
		L \circ m_{\alg} \circ ( \phi_- \tp \id )
		= m_{\alg} \circ (\phi_- \tp L),
		\label{eq:counterterm-scalars}
	\end{equation}
	that is to say, $L$ is linear over the counterterms, we have
	\begin{equation}
		\bar{\phi}_R \circ B_+ = L \circ \phi_+.
		\label{eq:rbar-cocycle}
	\end{equation}
\end{satz}
\begin{proof}
	This is a straightforward consequence of the cocycle property:
	\begin{align*}
		\bar{\phi}_R \circ B_+
		&= \left( \phi_+ - \phi_- \right) \circ B_+
		 = \left( \phi_- \convolution \phi - \phi_- \right) \circ B_+
		 = m_{\alg} \circ (\phi_- \tp \phi) \circ \Delta \circ B_+ - \phi_- \circ B_+ \\
		&= m_{\alg} \circ (\phi_- \tp \phi) \circ \left[ (\id \tp B_+) \circ \Delta + B_+ \tp \1 \right] - \phi_- \circ B_+ \\
		&= m_{\alg} \circ \left[ \phi_- \tp (\phi \circ B_+) \right] \circ \Delta + \left( \phi_- \circ B_+ \right) \cdot \phi(\1) - \phi_- \circ B_+ \\
		&= \phi_- \convolution \left( \phi \circ B_+ \right) 
		 = \phi_- \convolution \left( L \circ \phi \right)
		 = m_{\alg} \circ (\id \tp L) \circ (\phi_- \tp \phi) \circ \Delta \\
		&= L \circ m_{\alg} \circ (\phi_- \tp \phi) \circ \Delta
		 = L \circ \left( \phi_- \convolution \phi \right)
		 = L \circ \phi_+ \qedhere.
	\end{align*}
\end{proof}
First of all note that the condition \eqref{eq:counterterm-scalars} is fulfilled in our case: The counterterms $\toyZ$ are independent of the parameter $s$ such that they can be moved out of the integral in \eqref{eq:toymodel}! This actually applies in general to renormalization of {\qfts}: It is the nature of the counterterms to not depend on any of the external variables.

Even if the divergence of a Feynman graph does depend on external momenta (this happens for the quadratically divergent scalar propagator in renormalizable \qfts), the Hopf algebra $H$ is defined in such a way that the counterterms are evaluations on certain \emph{external structures}, given by distributions. So in any case, $\phi_-$ maps to scalars independent of the external momenta.
%In terms of Feynman graphs, this is also evident, as the external legs of subdivergence do not naturally correspond to the external parameters (external legs) of the original graph!

In fact it is the whole point of the Hopf algebra approach to renormalization to put the momentum dependence into indices on contracted vertices living in the right-hand side of the coproduct. For details of this concept we refer to \cite{CK:RH1}.

The result \eqref{eq:rbar-cocycle} is most powerful and shows that the renormalized value of a tree $B_+(f)$ can be gained out of the knowledge of the renormalized value $\toyR(f)$ only! In particular it allows for inductive proofs of properties of $\toyR$ and also $\toyphy$, without having to consider the unrenormalized Feynman rules or their counterterms at all!

Though we will restrict ourselves to the toy model here, the method employed in proposition \ref{satz:finiteness} can be easily extended to prove finiteness and also locality of renormalization in a very general setting, see \cite{Factorization}.

\subsection{Finiteness and BPHZ}
\label{sec:finiteness}
\begin{proposition}\label{satz:finiteness}
	The physical limit $\toyphy[s]$ of the toy model exists and maps $H_R$ into the polynomials $\K[\ln \tfrac{s}{\rp}]$.
\end{proposition}

\begin{proof}
	We prove the claim inductively, starting with the trivial case of the empty forest $\toyphy[s](\1) = 1$. As $\toyR$ is a character, so will be $\toyphy$ (if existent) and we may restrict to trees $t=B_+(x)$. Assuming the claim to hold for $x\in\forests$, we can take the limit
	\settowidth{\wurelwidth}{\eqref{eq:rbar-cocycle}}
	\begin{align}
		\toyphy[s] (t)
		&\wurel{\eqref{eq:rbar-cocycle}} \lim_{\reg \rightarrow 0} (\id - \momsch{\rp}) \left[ s \mapsto \int_0^{\infty} \frac{f(\zeta/s)}{s} \zeta^{-\reg}\ \toyR[\zeta](x)\ \dd\zeta \right] \nonumber\\
		&\wurel{} \lim_{\reg \rightarrow 0} \int_0^{\infty} \left[ \frac{f(\zeta/s)}{s} - \frac{f(\zeta/\rp)}{\rp} \right] {\zeta}^{-\reg}\ \toyR[\zeta] (x)\ \dd\zeta \nonumber\\
		&\wurel{} \int_0^{\infty} \left[ \frac{f(\zeta/s)}{s} - \frac{f(\zeta/\rp)}{\rp} \right]\ \toyphy[\zeta](x)\ \dd\zeta
		\label{eq:BPHZ}
	\end{align}
	using dominated convergence: the term in square brackets lies in $\bigo{\zeta^{-2}}$ like \eqref{eq:momsch-example} and by assumption $\toyphy[\zeta](x) \in \bigo{\ln^N \zeta}$ for some $N\in\N$, wherefore this integral is convergent! Thus knowing the limit $\reg \rightarrow 0$ of $\toyR[s](t)$ to exist, we proved that all pole terms in the Laurent series of $\toyR[s](t)$ must cancel and identify $\toyphy[s](t)$ with the $\propto \reg^0$ term.
	
	Inspection of \eqref{eq:toymodel-mellin} and the scheme $R_{\rp}$ reveals that this coefficient is a polynomial in $\ln s$ and $\ln\rp$ of order $\abs{t}$, as these logarithms come with a factor $\reg$ (expanding $s^{-\reg}$) which needs to cancel with a pole term $\tfrac{c_{-1}}{\reg\abs{t_v}}$ from some $F(\reg\abs{t_v})$ in order to contribute to the constant ($\propto \reg^0$) term -- compare with \eqref{toyR-physical:()} and \eqref{toyR-physical:(())}!

%	Finally, by $\ev_{\rp} \circ (\id - \momsch{\rp}) = 0$ and \eqref{eq:birkhoff-dec} we observe $\toy_{\rp} = 0$ on $\ker \counit$. Hence $\toyphy[s](t)$ must be a polynomial in $\ln s - \ln \rp = \ln \tfrac{s}{\rp}$.
	Finally, we observe that $\toyphy[s]$ is only a function of $\frac{s}{\rp}$. Starting with $\toyphy[s](\1)=1$ this follows inductively from \eqref{eq:BPHZ}, substituting $\zeta\mapsto\rp\zeta$ such that
	\begin{equation*}
		\int_0^{\infty} \left[ \frac{f(\zeta/s)}{s} - \frac{f(\zeta/\rp)}{\rp} \right]\ \toyphy[\zeta](x)\ \dd\zeta
		= \int_0^{\infty} \left[ \frac{f(\zeta\tfrac{\rp}{s})}{\tfrac{s}{\rp}} - f(\zeta) \right]\ \toyphy[\rp\zeta](x)\ \dd\zeta.
	\end{equation*}
	Note that $\toyphy[\rp\zeta](f)$ is independent of $\rp$ by the induction hypothesis.
\end{proof}
Using \eqref{eq:BPHZ}, the physical limit of the renormalized Feynman rules can be obtained inductively by convergent integrations after performing the subtraction at $s=\rp$ on the level of the integrand. This method is known under the name \emph{BPHZ{\footnotemark} scheme}. In particular note that it does not need any regulator at all!
\footnotetext{named by Nikolay Nikolaevich Bogoliubov, Ostap Stepanovych Parasiuk, Klaus Hepp and Wolfhart Zimmermann}

As an important consequence, the physical limit of the toy model does not depend on the chosen regulator, as long as one employs the momentum scheme (which leads to the BPHZ method)! So should we instead of analytic regularization employ a cutoff regulator $\Lambda$
\begin{equation*}
	\phi^{(\Lambda)} \circ B_+ \defas \left[ s \mapsto \int_0^{\Lambda} \frac{1}{s + \zeta}\ \phi^{(\Lambda)}_{\zeta}\ \dd\zeta \right]
\end{equation*}
instead, the renormalized Feynman rules using the scheme $\momsch{\rp}$ again lead to \eqref{eq:BPHZ} and thus the same physical limit $\toyphy = \lim_{\Lambda \rightarrow \infty} \phi^{(\Lambda)}$!

\subsection{Feynman rules induced by cocycles}
\label{sec:toymodel-rules-by-cocycles}
Now we find ourselves in the familiar setup of section \ref{sec:more-feynman-rules}: the renormalized Feynman rules $\toyphy$ map $H_R$ into the Hopf algebra $\K[x]$ of polynomials, such that evaluation at $x=\ln\frac{s}{\rp}$ delivers the value $\toyphy[s]$. Thus it is natural to ask whether $\toyphy$ arises through the universal property of $H_R$. This is actually the case as shown in

\begin{satz}\label{satz:toymodel-universal}
	Defining the functional $\toyform \in \K[x]'$ and the cocycle $\toycc \in HZ^1_{\counit}(\K[x])$ by
	\begin{equation}
		\toycc \defas -\coeff{-1} \int_0 + \partial \toyform
		\quad\text{and}\quad
		\toyform \left( x^n \right) \defas n!\, (-1)^n \coeff{n}
		\quad\text{for any $n\in \N_0$,}
		\label{eq:cocycle-toymodel}
	\end{equation}
	where the $c_n$ are the coefficients from \eqref{eq:mellin-trafo}, we have
	\begin{equation}
		\toyphy = \unimor{\toycc}
		\label{eq:toymodel-universal}
	\end{equation}
	for the Hopf algebra morphism $\unimor{\toycc}$ from theorem \ref{satz:H_R-universal}.
\end{satz}

\begin{proof}
	First investigate how logarithms contributing to subdivergences evolve for $\rp=1$:
	\begin{align*}
		&\lim_{\reg \rightarrow 0} (\id - \momsch{1}) \left[ s \mapsto \int_0^{\infty} f(\zeta) {\left( s\zeta \right)}^{-\reg} \ln^n \left( s\zeta \right)\;\dd\zeta \right]
		= {\left( -\frac{\partial}{\partial \reg} \right) }_{\reg=0}^n
	%	\restrict{\left( -\partial_z \right)^n}{0}
		(\id - \momsch{1}) \int_0^{\infty} f(\zeta) {\left( s\zeta \right)}^{-\reg}\; \dd\zeta \\
		&= {\left( -\frac{\partial}{\partial \reg} \right)}_{\reg=0}^n \left( s^{-\reg} - 1 \right) \int_0^{\infty} f(\zeta) \zeta^{-\reg}\; \dd\zeta
		 = {\left( -\frac{\partial}{\partial \reg} \right)}_{\reg=0}^n 
		 \left[
		 		\left\{ \frac{s^{-\reg}-1}{\reg} \right\} 
				\Big\{ \reg F(\reg) \Big\}
		 \right]\\
		&= {(-1)}^n \sum_{k=0}^n \binom{n}{k} 
		 		\left\{ {\left( \frac{\partial}{\partial \reg} \right)}_{\reg=0}^k \frac{s^{-\reg}-1}{\reg} \right\} 
				\cdot 
				\left\{ {\left( \frac{\partial}{\partial \reg} \right)}_{\reg=0}^{n-k} \big[ \reg F(\reg) \big] \right\} \\
		&= {(-1)}^n \sum_{k=0}^n \binom{n}{k} k! \frac{{\left( - \ln s \right)}^{k+1}}{(k+1)!} (n-k)!\, c_{n-k-1} 
		 = \ev_{\ln s} \left[ \sum_{k=0}^n \frac{n!\, x^{k+1}}{(k+1)!} {(-1)}^{n-k-1} c_{n-k-1}  \right] \\
%		&= - \frac{c_{-1}}{n+1} \ln^{n+1} s + \sum_{k=0}^{n-1} \frac{n!}{(k+1)!} (-1)^{n-k-1} c_{n-k-1} \ln^{k+1} s \\
%		&= - c_{-1} \int_0 \ln^n s + \sum_{l=1}^n \frac{n!}{l!} (-1)^{n-l} c_{n-l} \ln^l s
		&= \ev_{\ln s} \left[-c_{-1} \frac{x^{n+1}}{n+1} + \sum_{i=1}^n \binom{n}{i} x^i {(-1)}^{n-i} c_{n-i} (n-i)! \right]
%		= -c_{-1} \int_0 \ln^n s + \partial \toyform \left( \ln^n s \right).
		 = \ev_{\ln s} \circ \toycc \left( x^n \right).\tag{$\ast$}
	\end{align*}
	Here we expanded the holomorphic functions
	\begin{equation*}
		\frac{s^{-\reg}-1}{\reg} = \sum_{n=0}^{\infty} \frac{ {(-\ln s)}^{n+1} }{(n+1)!} \reg^n
		\quad \text{and} \quad
		\reg F(\reg) = \sum_{n=0}^{\infty} c_{n-1} \reg^n
	\end{equation*}
	and exploited the renormalization scheme $\momsch{\rp}$ to only evaluate $s$ and not to act on $\reg$, in particular it commutes with $\frac{\partial}{\partial \reg}$.

	By linearity the above holds also if we replace $\ln^n (s\zeta)$ in the integrand by any polynomial in $\K[\ln (s\zeta)]$! This allows us to prove \eqref{eq:toymodel-universal} inductively on trees as usual (both sides of \eqref{eq:toymodel-universal} are algebra morphisms): let it hold for a forest $x \in \forests$, then observe
	\settowidth{\wurelwidth}{\eqref{eq:BPHZ}}
	\begin{align*}
		\toyphy[s] \circ B_+ (x)
%		&= \lim_{\reg\rightarrow 0} (\id - \momsch{1}) \left[ s \mapsto \int_0^{\infty} \frac{f\left( \zeta / s \right)}{s} \zeta^{-\reg} \ \toyR[\zeta] (f)\ \dd\zeta \right] \\
%		&= \lim_{\reg \rightarrow 0} \int_0^{\infty} \left[ \frac{f\left( \zeta / s \right)}{s} - \frac{f(\zeta)}{1} \right] \zeta^{-\reg} \ \toyR[\zeta] (f)\ \dd\zeta \\
		&\wurel{\eqref{eq:BPHZ}} \int_0^{\infty} \left[ \frac{f\left( \zeta / s \right)}{s} - \frac{f(\zeta)}{1} \right]  \ \toyphy[\zeta] (x)\ \dd\zeta 
		= \lim_{\reg \rightarrow 0} \int_0^{\infty} \left[ \frac{f\left( \zeta / s \right)}{s} - \frac{f(\zeta)}{1} \right] {\zeta}^{-\reg}  \ \toyphy[\zeta] (x)\ \dd\zeta \\
		&\wurel{} \lim_{\reg \rightarrow 0} (\id - \momsch{1}) \left[ s \mapsto \int_0^{\infty} f(\zeta) {\left( s\zeta \right)}^{-\reg}\ \toyphy[s\zeta] (x)\ \dd\zeta \right] \\
		&\wurel{} \ev_s \circ \toycc \left[ \toyphy (x) \right]
		 = \ev_s \circ \toycc \circ \unimor{\toycc} (x)
		 = \ev_s \circ \unimor{\toycc} \circ B_+ (x).
	\end{align*}
	We exploited the convergence of \eqref{eq:BPHZ} and reintroduced $\zeta^{-\reg}$ into the integrand in order to apply $(\ast)$.
\end{proof}

The above result translates the task of renormalization of the toy model into a very simple combinatoric recursion, without any need for series expansions. To illustrate the benefit we rederive the earlier examples \eqref{toyR-physical:()} to \eqref{toyR-physical:(()())}, now using theorem \ref{satz:toymodel-universal}:
\begin{align*}
	\toyphy \left( \tree{+-} \right)
	&= \unimor{\toycc} \left( \tree{+-} \right)
	 = \unimor{\toycc} \circ B_+ (\1)
	 = \left[ -\coeff{-1} \int_0 + \partial \toyform \right] (1)
	 = - \coeff{-1} \int_0 1
	 = - \coeff{-1}\, x \\
	\toyphy \left( \tree{++--} \right)
	&= \unimor{\toycc} \circ B_+ \left( \tree{+-} \right)
	 = \toycc \circ \unimor{\toycc} \left( \tree{+-} \right)
	 = \left[ -\coeff{-1}\int_0 + \partial\toyform \right] \left( -\coeff{-1} x \right)
	 = \coeff[2]{-1} \frac{x^2}{2} - \coeff{-1} \coeff{0}\, x \\
	\toyphy \left( \tree{+++---} \right)
	&= \left[ -\coeff{-1}\int_0 + \partial\toyform \right] \left( \coeff[2]{-1} \frac{x^2}{2} - \coeff{-1} \coeff{0}\, x \right)
	= - \coeff[3]{-1} \frac{x^3}{6} + \coeff[2]{-1} \coeff{0} \frac{x^2}{2}+ \coeff[2]{-1} \toyform(x)\,x \\
	 &\quad + \toyform(1) \left( \coeff[2]{-1} \frac{x^2}{2} - \coeff{-1} \coeff{0}\, x \right) 
	 = - \coeff[3]{-1} \frac{x^3}{6} + \coeff[2]{-1} \coeff{0}\, x^2 - \left( \coeff{-1} \coeff[2]{0} + \coeff[2]{-1} \coeff{1} \right) x \\
	\toyphy \left( \tree{++-+--} \right)
	&= \unimor{\toycc} \circ B_+ \left( \tree{+-}\tree{+-} \right)
	 = \toycc \circ \unimor{\toycc} \left( \tree{+-}\tree{+-} \right) 
	 = \left[ -\coeff{-1}\int_0 + \partial\toyform \right] \left\{ {\left( -\coeff{-1}\,x \right) }^2 \right\} \\
	&= -\coeff[3]{-1} \frac{x^3}{3} + \coeff[2]{-1} \left[ \toyform(1)\,x^2 + 2\toyform(x)\,x \right]
	 = -\coeff[3]{-1} \frac{x^3}{3} + \coeff[2]{-1} \coeff{0}\,x^2 - 2\coeff[2]{-1}\coeff{1}\,x.
\end{align*}
As $\toycc$ is a cocycle, by theorem \ref{satz:H_R-universal} we note the
\begin{korollar}\label{satz:toyphy-hopfmor}
	The physical limit $\toyphy\!: H_R \rightarrow \K[x]$ of the renormalized Feynman rules \eqref{eq:toymodel} of the toy model is a morphism of Hopf algebras.
\end{korollar}
This result implies a range of important consequences -- in particular the renormalization group \eqref{eq:rge-physical} -- which we will briefly discuss in the upcoming sections. For now let us remark:
\begin{enumerate}
%	\item As $\toycc$ is a cocycle, $\toyphy$ is actually a morphism of Hopf algebras and therefore it fulfils \eqref{eq:int-rules-rge}:
%		\begin{equation*}
%			(\ev_a \circ \toyphy) \convolution (\ev_b \circ \toyphy) = \ev_{a+b} \circ \toyphy.
%		\end{equation*}
	\item Up to the lower order modifications $\partial \toyform$, $\toycc$ is just $-\coeff{-1}\int_0$. Hence the highest order term of $\toyphy[s]$ -- the \emph{leading $\log$} -- is the same as for $\unimor{\left[-\coeff{-1}\int_0\right]}$:
		\begin{equation}
			\forall f \in \mathcal{F}\!: \quad
			\toyphy[s] (f) \in \frac{ {\left( - \coeff{-1} \ln \tfrac{s}{\rp} \right)}^{\abs{f}}}{f!} + \bigo{\ln^{\abs{f}-1} \tfrac{s}{\rp}}.
			\label{eq:toymodel-leading-log}
		\end{equation}
		Note how this leading $\log$ corresponds to the leading divergence \eqref{eq:toymodel-leading-pole}!
	\item By theorem \ref{satz:change-coboundary-equals-auto}, the deviation of $\toyphy$ from $\unimor{-\coeff{-1}\int_0}$ is given by an automorphism of $H_R$ that adds only lower order corrections:
		\begin{equation*}
			\toyphy 
			= \unimor{\left[-\coeff{-1}\int_0\right]} \circ \auto{\left[\toyform\, \circ\, \unimor{\left(-\coeff{-1}\int_0 \right)}\right]}.
		\end{equation*}
	\item Consider a change of $f$ and thus its Mellin transform $F$ to different functions $f'$ and $F'$, keeping $c_{-1}$ fixed but probably altering the other coefficients $\coeff{n}$ of $F$. Then again the difference in the resulting Feynman rules is captured by an automorphism of $H_R$ as
		\begin{equation*}
			{\toyphy}' 
			= \unimor{\toycc'}
			= \unimor{\toycc + \partial (\delta \toyform)}
			= \unimor{\toycc} \circ \auto{\left[ \delta\toyform\, \circ\, \unimor{\toycc} \right]}
			= \toyphy \circ \auto{\left[ \delta\toyform\, \circ\, \toyphy \right]},
		\end{equation*}
		where $\delta\toyform \defas \toyform' - \toyform$ denotes the change in the $\coeff{n}$ for $n \in \N_0$. In other words, altering $F$ corresponds to addition of a coboundary!
\end{enumerate}

\section{The structure of higher orders}
\label{sec:higher-orders}
In \eqref{eq:int-rules-rge} we discovered a one-parameter subgroup
\begin{equation*}
	\K \ni a \mapsto \intrules_a \in \chars{H_R}{\K},\quad
	\intrules_a \convolution \intrules_b = \intrules_{a+b}
\end{equation*}
of the convolution group of characters of $H_R$, which imposes strong combinatorial constraints on $\intrules$. The \emph{(infinitesimal) generator} of this subgroup is determined by the equation \mbox{$\intrules_a = \exp_{\convolution} (a \log_{\convolution} \intrules_1) \urel{\eqref{eq:conv-power}} \intrules_1^{\convolution a}$} and evaluates with \eqref{eq:exp-diff} to
\begin{equation}
	\log_{\convolution} \intrules_1
	= \restrict{\frac{\partial}{\partial a}}{0} \intrules_a
	= \restrict{\frac{\partial}{\partial a}}{0} \ev_a \circ \intrules
	= \partial_0 \circ \intrules
	= Z_{\tree{+-}},
	\label{eq:intrules-generator}
\end{equation} 
such that $\intrules_a = \intrules_1^{\convolution a} = \exp_{\convolution} (a Z_{\tree{+-}})$. Here we introduced the map
\begin{equation}
	\partial_0 \defas \restrict{\frac{\partial}{\partial x}}{0}\!: \quad
	\K[x] \rightarrow \K, \quad
	x^n \mapsto \delta_{1,n} =
		\begin{cases}
			1 & \text{if $n=1$,} \\
			0 & \text{else},
		\end{cases}
	\label{eq:diff0-poly}
\end{equation}
extracting the coefficient of $x=x^1$ in a polynomial as well as the functional
\begin{equation}
	Z_{\tree{+-}} \in H_R' \quad\text{by}\quad
	Z_{\tree{+-}} (f) = \delta_{f,\tree{+-}}
\end{equation}
for any forest $f\in\forests$, using Kronecker's $\delta$ again. The last equality in \eqref{eq:intrules-generator} is immediate as $\intrules$ only produces a contribution proportional to $x$ when applied to a forest of a single node. Observation \eqref{eq:intrules-generator} generalizes to
\begin{proposition}
	Let $H$ be any connected bialgebra and $\phi\!: H \rightarrow \K[x]$ a morphism of bialgebras.\footnote{It is an easy exercise to prove inductively, using \eqref{eq:antipode-recursive}, that this already implies $\phi$ to be a morphism of Hopf algebras.} Then $\log_\convolution \phi$ is simply given by the term proportional to $x$:
	\begin{equation}
		\log_{\convolution} \phi = x \cdot \partial_0 \circ \phi.
	\end{equation}
\end{proposition}
\begin{proof}
	As $\phi$ is a morphism of coalgebras, by proposition \ref{satz:convolution-rge} and \eqref{eq:poly-characters} we obtain a one-parameter group $\K\ni a \mapsto \ev_a \circ \phi \in \chars{H}{\K}$ in
	\begin{equation*}
		(\ev_a \circ \phi) \convolution (\ev_b \circ \phi)
		= \ev_{a+b} \circ \phi.
	\end{equation*}
	Hence as in \eqref{eq:intrules-generator} we immediately conclude\footnote{%
	More generally note that $(\log_{\convolution} \psi) \circ \phi = \log_{\convolution} (\psi \circ \phi) = \psi \circ \log_{\convolution} \phi$ for any algebra morphism $\psi$ and a coalgebra morphism $\phi$ with $\phi(\1) = \1$. We in fact showed $\log_{\convolution} \ev_a = a \log_{\convolution} \ev_1 = a \partial_0$ in \eqref{eq:intrules-generator}.}
	\begin{equation*}
		a \partial_0 \circ \phi
		= a \log_{\convolution} \left( \ev_1 \circ \phi \right)
		= \log_{\convolution} \left( \ev_a \circ \phi \right)
		\urel{($\ast$)} \ev_a \circ \log_{\convolution} \phi,
	\end{equation*}
	where $(\ast)$ is a consequence of the character property of $\ev_a$ and \eqref{eq:log-conv} by
	\begin{align*}
		{(\ev_a \circ \phi - e)}^{\convolution n}
		&= m_{\K}^{n-1} \circ \ev_a^{\tp n} \circ (\phi-e)^{\tp n} \circ \Delta_H^{n-1} \\
		&= \ev_a \circ m_{\K[x]}^{n-1} \circ (\phi-e)^{\tp n} \circ \Delta_H^{n-1}
		= \ev_a \circ (\phi-e)^{\convolution n}.
		\qedhere
	\end{align*}
\end{proof}
In particular, such bialgebra morphisms are completely determined by the functional in $H'$ that extracts the coefficient of $x$! In the above example, \mbox{$\log_{\convolution} \intrules = x \cdot Z_{\tree{+-}}$} gives
\begin{equation}\label{eq:int-rules-exp}
	\intrules
	= \exp_{\convolution} \left( x \cdot Z_{\tree{+-}} \right)
	= \sum_{n=0}^{\infty} \frac{x^n}{n!} Z_{\tree{+-}}^{\convolution n}.
\end{equation}
This entails a direct proof of a combinatoric relation\footnote{This is equation $(124)$ of \cite{Kreimer:ChenII}, where leaves are called \emph{feet}.} among tree factorials in
\begin{korollar}
	For any forest $f\in\forests$ let $\leaves(f)\subseteq V(f)$ denote the set of \emph{leaves} of $f$, being those nodes without children. Then we have the identity
	\begin{equation}\label{eq:tree-factorial-feet}
		\frac{\abs{f}}{f!} = \sum_{v\in\leaves(f)} \frac{1}{\left[R^{\set{v}}(f)\right]!}.
	\end{equation}
\end{korollar}
\begin{proof}
	By \eqref{eq:int-rules} and \eqref{eq:int-rules-exp} we note $\frac{1}{f!} = \frac{1}{\abs{f}!} Z_{\tree{+-}}^{\convolution \abs{f}}(f)$, hence
	\begin{equation*}
		\frac{\abs{f}}{f!}
		= \frac{1}{\left( \abs{f} -1 \right)!} \sum_{f} Z_{\tree{+-}}(f_1)Z_{\tree{+-}}^{\convolution \abs{f}-1}(f_2)
		= \sum_{\substack{f\\ f_1=\tree{+-}}} \frac{1}{\abs{f_2}!}Z_{\tree{+-}}^{\convolution \abs{f_2}} (f_2)
		\urel{\eqref{eq:coproduct-cuts}} \sum_{v\in\leaves(f)} \frac{1}{\left[ R^{v}(f) \right]!}. \qedhere
	\end{equation*}
\end{proof}
These considerations may seem trivial for $\intrules$, but they also apply to the physical limit $\toyphy$ of the toy model through corollary \ref{satz:toyphy-hopfmor}! We introduce the functional
\begin{equation}
	H_R' \ni
	\toylog \defas \log_{\convolution} \ev_1\, \circ\, \toyphy
	= \restrict{\frac{\partial}{\partial a}}{0} \ev_a\:\circ\:\toyphy
	= \partial_0\: \circ\: \toyphy
	\label{eq:beta-toymodel}
\end{equation}
and obtain $\toyphy = \exp_{\convolution} (x\cdot \toylog)$. Note{\footnotemark} that $\toylog$ is an \emph{infinitesimal character}, that is to say \mbox{$\toylog \circ m = \toylog \tp \counit + \counit\tp\toylog$} wherefore $\toylog$ vanishes on any forest that is not a tree. From \eqref{toyR-physical:()} to \eqref{toyR-physical:(()())} we read off
\footnotetext{A simple proof can be found as proposition II.4.2 in \cite{Manchon}.}%
\begin{align*}
	\toylog \left( \tree{+-} \right)
	&= -\coeff{-1}
	&
	\toylog \left( \tree{++--} \right)
	&= -\coeff{-1}\coeff{0}
	&
	\toylog \left( \tree{+++---} \right)
	&= -\coeff{-1}\coeff[2]{0} - \coeff[2]{-1} \coeff{1}
	&
	\toylog \left( \tree{++-+--} \right)
	&= - 2\coeff[2]{-1}\coeff{1}
\end{align*}
and show how $\toylog$ determines the higher powers of $x$ in two examples:
\begin{align*}
	\toyphy \left( \tree{++--} \right)
	&= \exp_{\convolution} (x\cdot\toylog) \left( \tree{++--} \right)
	 \urel{\eqref{eq:exp-conv}} \left[ e + x\toylog + x^2\frac{\toylog\convolution\toylog}{2} \right] \left( \tree{++--} \right) \\
	&= 0 + x\toylog\left( \tree{++--} \right) + x^2 \frac{\toylog^2\left( \tree{+-} \right)}{2} 
	 = - \coeff{-1}\coeff{0}\,x + \coeff[2]{-1} \frac{x^2}{2}, \\
	\toyphy \left( \tree{++-+--} \right)
	&= \exp_{\convolution} (x\cdot\toylog) \left( \tree{++-+--} \right)
	 \urel{\eqref{eq:exp-conv}} \left[ e + x\toylog + x^2\frac{\toylog \convolution \toylog}{2} + x^3\frac{\toylog \convolution \toylog \convolution \toylog}{6} \right] \left( \tree{++-+--} \right) \\
	&= 0 + x \toylog \left( \tree{++-+--} \right) 
		+ x^2\frac{\toylog \tp \toylog}{2} \left( 2 \tree{+-} \tp \tree{++--} + \tree{+-}\tree{+-} \tp \tree{+-} \right)
		+ x^3\frac{\toylog \tp \toylog \tp \toylog}{6} \left( 2\tree{+-} \tp \tree{+-} \tp \tree{+-} \right) \\
		&= \toylog^3\left(\tree{+-}\right) \frac{x^3}{3} + x^2 \toylog\left(\tree{+-}\right) \toylog\left(\tree{++--}\right)  - 2\coeff[2]{-1}\coeff{1}\,x
	 = -\coeff[3]{-1} \frac{x^3}{3} + \coeff[2]{-1} \coeff{0}\,x^2 - 2\coeff[2]{-1}\coeff{1}\,x.
\end{align*}
A very similar phenomenon happens with the counterterms in a minimal subtraction scheme: It turns out that the single poles $\propto \reg^{-1}$ determine the full counterterm already. This is the content of the \emph{scattering formula} proved in \cite{CK:RH2}.

Our case is much simpler as the higher power contributions to $\toyphy$ are explicitly given by merely taking convolution powers of $\toylog$, that is
\begin{equation}\label{eq:toymodel-higher-orders}
	\toyphy = \sum_{n=1}^{\infty} x^n \toylog_n
	\quad\text{for the functionals}\quad
	\toylog_n = \frac{\toylog^{\convolution n}}{n!} \in H_R'.
\end{equation}
Note how in the above example, the fragment $\tree{+-}\tree{+-} \tp \tree{+-}$ of $\Delta \left( \tree{++-+--} \right)$ does not contribute to the quadratic terms $\frac{x^2}{2} \toylog \convolution \toylog$, as $\toylog$ vanishes on proper products. This is an important result we shall exploit in \eqref{eq:rge}, leading to the \emph{renormalization group equation} \eqref{eq:rge-physical}!

\section{Correlation functions and Dyson-Schwinger equations}
\label{sec:DSE}
So far we considered the renormalized Feynman rules as a whole. Now {\qft} teaches that instead of individual contributions (of single trees or more generally graphs), the quantities that bear physical meaning are the \emph{correlation functions}. These are determined as asymptotic expansions through the formal power series in the coupling constant by summation of the renormalized contributions of all trees (Feynman graphs in \qft), like in \eqref{eq:correlation}.

This concept fits beautifully into the Hopf algebraic setting by the aid of fixed point equations in Hochschild cohomology. For details about these \emph{combinatorial Dyson-Schwinger equations} in general we refer to \cite{BergbauerKreimer}. Here we will only discuss the special case of our toy model, where the Dyson-Schwinger equation reads\footnote{At the end of section \ref{sec:iterated-insertions} we show why \eqref{eq:DSE-propagator} suits the toy model.}
\begin{equation}
	X(\alpha) = \alpha B_+ \left( \frac{\1}{\1 - X(\alpha)} \right) \defas \alpha B_+ \left( \sum_{n \in \N_0} {\left[ X(\alpha) \right]}^n \right)
	\label{eq:DSE-propagator}
\end{equation}
for a formal power series $X(\alpha) \in H_R [[\alpha]]$ in a parameter $\alpha$. More concretely, $X(\alpha) = \sum_{n\in\N} a_n \alpha^n$ for homogeneous $a_n \in H_{R,n}$ determined in
\begin{proposition}\label{satz:DSE-solution}
	The unique solution of \eqref{eq:DSE-propagator} is given by
	\begin{equation}
		\forall n\in\N\!: \quad a_n = \sum_{t\in\trees_n} \dsecoeff(t) \cdot t.
		\label{eq:DSE-solution}
	\end{equation}
	Here, $a_n$ sums all trees of $n$ nodes and weights them with a factor $\dsecoeff(t)$ that counts the number of ordered rooted trees that yield $t$ upon forgetting the ordering. It fulfils the recursion
	\begin{equation*}
		\dsecoeff(t) = \binom{n_1 + \ldots + n_r}{n_1\ \cdots\ n_r} \dsecoeff^{n_1}(t_1) \ldots \dsecoeff^{n_r}(t_r),
	\end{equation*}
	if one writes $t=B_+\left( t_1^{n_1} \ldots t_r^{n_r} \right)$ such that $t_i \neq t_j$ whenever $i \neq j$.
\end{proposition}
\begin{proof}
	The existence and uniqueness of $a_n$ solving \eqref{eq:DSE-propagator} follows immediately by counting powers of $\alpha$. We can directly read off the recursion
	\begin{equation}\label{eq:a_n-recursion}
		a_{n+1} = B_+ \left( \sum_{k=0}^n\ \sum_{i_1+\ldots+i_k = n} a_{i_1} \ldots a_{i_k} \right).
	\end{equation}
	This proves \eqref{eq:DSE-solution} and the claim about $\sigma$, as each term of this sum corresponds to a unique ordered rooted tree! For a much more general exposition see \cite{Foissy:DSE}, especially theorem 11 therein.
\end{proof}
So $\dsecoeff\left( \tree{+-} \right)
	= \dsecoeff \left( \tree{++--} \right)
	= \dsecoeff \left( \tree{+++---} \right)
	= \dsecoeff \left( \tree{++-+--} \right)
	= \dsecoeff\left(\tree{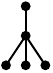} \right)
	= \dsecoeff\left( \tree{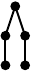} \right)
	= \dsecoeff\left( \tree{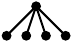} \right)
	=1$, 
whereas
\begin{equation*}
	\dsecoeff \left( \tree{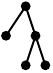} \right)
	= \abs{\set{\tree{++-++-+---}, \tree{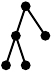}}}
	=2
	\quad \text{and} \quad
	\dsecoeff \left( \tree{+++--+-+--} \right)
	= \abs{\set{ \tree{+++--+-+--},\tree{++-++--+--},\tree{++-+-++---} }}
	= 3.
\end{equation*}
Here we wrote the different \emph{ordered} (\emph{planar}) rooted trees corresponding to $\dsecoeff$'s argument inside the sets. The first terms of $X$ are thus
\begin{equation}\label{eq:DSE-examples}\begin{split}
	a_1	&= \tree{+-} \quad
	a_2 = \tree{++--} \quad
	a_3 = \tree{+++---} + \tree{++-+--} \quad
	a_4 = \tree{++++----} + \tree{+++-+---} + 2 \tree{++-++---} + \tree{++-+-+--} \\
	a_5 &= \tree{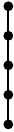} + \tree{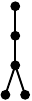} + 2\tree{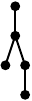} + \tree{+++-+-+---}+ \tree{+++--++---} + 2\tree{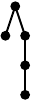} + 2\tree{++-++-+---} + 3\tree{+++--+-+--} + \tree{++-+-+-+--}.
\end{split}\end{equation}
\begin{definition}
	The \emph{correlation function} $G(\alpha)$ is the application of the Feynman rules $\phi\!: H_R\rightarrow\alg$ to the solution $X(\alpha)$ of \eqref{eq:DSE-propagator}, producing the formal power series
	\begin{equation}
		G (\alpha) \defas \phi \left( X(\alpha) \right)
		\defas \sum_{n\in\N} \phi (a_n) \alpha^n
		\in {\alg} [[\alpha]].
		\label{eq:correlation}
	\end{equation}
\end{definition}
Considering the Feynman rules $\intrules$ from \eqref{eq:int-rules}, using \eqref{eq:DSE-examples} check
\begin{align*}
	\frac{\intrules(a_1)}{x}
		&= \frac{1}{\tree{+-}\,!} = 1
	\qquad
	\frac{\intrules(a_2)}{x^2}
		 = \frac{1}{\tree{++--}\,!}
		 = \frac{1}{2}
	\qquad
	\frac{\intrules(a_3)}{x^3}
		 = \frac{1}{\tree{+++---}\,!} + \frac{1}{\tree{++-+--}\,!}
		 = \frac{1}{6} + \frac{1}{3} 
		 = \frac{1}{2} \\
		 \frac{\intrules(a_4)}{x^4}
		&= \frac{1}{\tree{++++----}\,!} + \frac{1}{\tree{+++-+---}\,!} + \frac{2}{\tree{++-++---}\,!} + \frac{1}{\tree{++-+-+--}\,!}
		 = \frac{1}{24} +\frac{1}{12} + \frac{2}{8} + \frac{1}{4} 
		 = \frac{5}{8} \\
	\frac{\intrules(a_5)}{x^5} 
		&= \frac{1}{120} + \frac{1}{60} + \frac{2}{40} + \frac{1}{20} + \frac{1}{20} + \frac{2}{30} + \frac{2}{15} + \frac{3}{10} +\frac{1}{5}
		 = \frac{7}{8}
\end{align*}
such that we obtain the first terms of the series
\begin{equation}
	G (\alpha, x) = \alpha x + \frac{1}{2} (\alpha x)^2 + \frac{1}{2} (\alpha x)^3 + \frac{5}{8} (\alpha x)^4 + \frac{7}{8} (\alpha x)^5 + \bigo{{(\alpha x)}^6}.
	\label{eq:int-rules-correlation-expansion}
\end{equation}

\begin{proposition}
	For the Feynman rules $\intrules$ form \eqref{eq:int-rules}, the correlation function \eqref{eq:correlation} is absolutely convergent for $\abs{\alpha x} < \frac{1}{2}$ and evaluates to
	\begin{equation}
		G(\alpha, x) 
		= \frac{2\alpha x}{1+\sqrt{1-2\alpha x}}
		= 1-\sqrt{1-2\alpha x}.
		\label{eq:correlation-int-rules}
	\end{equation}
\end{proposition}

\begin{proof}
	We will first prove that using the \emph{Catalan numbers} $C_n$, we have for any $n\in\N_0$
	\begin{equation}
		\frac{\intrules(a_{n+1})}{x^{n+1}}
			= 2^{-n} C_n
			= 2^{-n}\binom{2n}{n} \frac{1}{n+1}
			= 2^{-n}\frac{(2n)!}{n!(n+1)!}.
			\tag{$\ast$}
	\end{equation}
	The well-known recursion $C_{n} = \sum_{i=1}^n C_{i-1} C_{n-i}$ allows for an inductive proof: suppose $(\ast)$ holds for $n\leq N$. Denoting the sets of \emph{ordered} rooted trees and forests by $\orderedtrees$ and $\orderedforests$, note that any $t\in\orderedtrees_{N+1}$ contributing to $a_{N+1}$ is of the form $t = B_+(t' f)$ for unique $t'\in\orderedtrees$ and $f\in\orderedforests$. Therefore
	\begin{align*}
		\frac{\intrules \left(a_{N+2} \right)}{x^{N+2}}
		&= \sum_{t\in\trees_{N+2}} \frac{\dsecoeff(t)}{t!}
		 = \sum_{t\in\orderedtrees_{N+2}} \frac{1}{t!}
		 = 	\sum_{t' \in \orderedtrees_{\leq N+1}}\ 
				\sum_{f \in \orderedforests_{N+1 - \abs{t'}}}
				\frac{1}{\left[ B_+\left( t'\cdot f \right) \right]!} \\
		&= \frac{1}{N+2}
			\sum_{k=1}^{N+1}
			\left( \sum_{t' \in \orderedtrees_k}
						\frac{1}{t'!}
			\right)
			\cdot
			\left( \sum_{f \in \orderedforests_{N+1-k}}
							\frac{N+2-k}{[B_+(f)]!}
			\right) \\
		&= \frac{1}{N+2}
					\sum_{k=1}^{N+1}
					2^{-(k-1)} C_{k-1}
					(N+2-k)
					2^{-(N+1-k)} C_{N+1-k} \\
		&= \frac{2^{-N}}{N+2}
				\sum_{k=1}^{N+1}(N+2-k) C_{k-1} C_{N+1-k} 
		 = \frac{2^{-(N+1)}}{N+2} \sum_{k=1}^{N+1}(N+2) C_{k-1} C_{N+1-k}  \\
		&= 2^{-(N+1)} C_{N+1},
	\end{align*}
	where we used $\sum_{k=1}^{N+1}(N+2-k)  C_{k-1} C_{N+1-k} = \sum_{k=1}^{N+1} k C_{k-1} C_{N+1-k}$ by the substitution $k \mapsto N+2-k$. Now recall the classical generating function
	\begin{equation*}
		\sum_{n\in\N_0} C_n x^n
		= \frac{1-\sqrt{1-4x}}{2x}
		= \frac{2}{1+\sqrt{1-4x}}.\qedhere
	\end{equation*}
\end{proof}
\begin{korollar}
	By \eqref{eq:toymodel-leading-log}, the leading-$\log$ contribution\footnote{Recall that to obtain the actual values, evaluate at $x=\ln\tfrac{s}{\rp}$.} to the toy model \eqref{eq:toymodel} is
	\begin{equation*}
		G(\alpha, x)_{\text{leading-$\log$}}
		= -\frac{2\coeff{-1}\alpha x}{1+\sqrt{1+2\coeff{-1}\alpha x}}.
	\end{equation*}
\end{korollar}
Note how in these cases, $G(\alpha,x)$ does only depend on the product $\alpha\cdot x$. This is a consequence of the fact that $\intrules$ respects the graduations, mapping homogeneous elements to homogeneous polynomials of the same degree. For the toy model $\toyphy$ this is not the case anymore (except for $\toyform\neq 0$).

\subsection{Differential equations and the renormalization group}
\label{sec:rge}
We will now exploit the special structure of $X(\alpha)$. Consider
\begin{definition}
	We denote by \mbox{$\Plin\!: H_R \rightarrow H_R$} the projection onto the linear span of trees. Thus \mbox{$\Plin(f)=0$} for any forest \mbox{$f\in\forests\setminus\trees$} and \mbox{$\Plin(t) = t$} for trees \mbox{$t\in\trees$}.
\end{definition}

\begin{proposition}
	For the solution $X(\alpha)$ of \eqref{eq:DSE-propagator}, we find that
	\begin{equation}\label{eq:delta-lin-DSE}
		(\Plin \tp \id) \circ \Delta \left( X(\alpha) \right)
%		= X(\alpha) \tp \sum_{n=1}^{\infty} (2n-1) a_n \alpha^n
		= X(\alpha) \tp \1 + X(\alpha) \tp (2\alpha\partial_{\alpha}-1) X(\alpha).
	\end{equation}
\end{proposition}

\begin{proof}
	It is a remarkable fact that the coefficients $a_n$ of $X(\alpha)=\sum_{n=1}^{\infty} a_n \alpha^n$ do generate a Hopf subalgebra! This was first observed in \cite{BergbauerKreimer} (theorem 3) studying a different Dyson-Schwinger equation, for our case we refer to the general discussion in \cite{Foissy:DSE}. Explicitly, proposition (15) therein gives
	\begin{equation*}
		\Delta (X(\alpha))
		= X(\alpha) \tp \1 + \sum_{n=1}^{\infty} {\left[\1-X(\alpha)\right]}^{1-2n} \tp a_n \alpha^n,
	\end{equation*}
	which reduces by counting powers of $\alpha$ to
	\begin{equation*}
		\Delta (a_n) = a_n \tp \1 + \sum_{k=1}^n \left[ \sum_{r=0}^n \binom{1-2k}{r}(-1)^r \sum_{i_1+\ldots+i_r = n-k} a_{i_1} \ldots a_{i_r} \right] \tp a_k.
	\end{equation*}
	The reader is invited to prove this result inductively using the Hochschild-closedness of $B_+$ and \eqref{eq:a_n-recursion}. Now simply conclude
	\begin{equation*}
		(\Plin \tp \id)\circ \Delta (a_n \alpha^n) 
		=  a_n\alpha^n \tp \1 + \sum_{k=1}^{n-1} a_{n-k} \alpha^{n-k} \tp (2k-1)  a_k \alpha^k. \qedhere
	\end{equation*}
\end{proof}
\begin{korollar}
	As $Z_{\tree{+-}}$ and $\toylog$ vanish on products, we obtain for any $n\in\N$
	\begin{align}
		Z_{\tree{+-}}^{\convolution n+1} \left( X(\alpha) \right)
%		= \left( Z_{\tree{+-}} \tp Z_{\tree{+-}}^{\convolution n} \right) \circ (\Plin \tp \id) \circ \Delta \left( X(\alpha) \right)
		&= Z_{\tree{+-}}\left(X(\alpha)\right) (2\alpha\partial_{\alpha} -1) Z_{\tree{+-}}^{\convolution n} \left( X(\alpha) \right)
			\quad\text{and} \\
		\toylog^{\convolution n+1} \left( X(\alpha) \right)
		&= \toylog\left( X(\alpha) \right) (2\alpha\partial_{\alpha}-1)\toylog^{\convolution n} \left( X(\alpha) \right).
		\label{eq:rge}
	\end{align}
\end{korollar}
As $Z_{\tree{+-}}\left( X(\alpha) \right) = \alpha$, we deduce $Z_{\tree{+-}}^{\convolution 2} \left(X(\alpha)\right) = \alpha(2\alpha\partial_{\alpha} -1)\alpha = \alpha^2$ and recursively
\begin{equation*}
	Z_{\tree{+-}}^{\convolution n+1} \left( X(\alpha) \right)
	= \alpha^{n+1} (2n-1)(2n-3)\cdots(1)
	= \alpha^{n+1} \frac{(2n)!}{2^n n!},
\end{equation*}
proving $\intrules(a_{n+1}) = 2^{-n} C_n x^{n+1}$ again.
The equations \eqref{eq:rge} are the physicist's \emph{renormalization group equations}, relating the dependence of $G(\alpha,x)$ on $\alpha$ with that on $x$:
\begin{align}\label{eq:rge2}
	\frac{\partial}{\partial x} G(\alpha,x)
	&\urel{\eqref{eq:correlation}} \frac{\partial}{\partial x} \left[ \toyphy\left( X(\alpha) \right) \right]
	\urel{\eqref{eq:toymodel-higher-orders}} \frac{\partial}{\partial x} \left[ \exp_{\convolution} \left( x\cdot \toylog \right) \right] \left( X(\alpha) \right)
	\urel{\eqref{eq:exp-diff}} \left(\toylog \convolution \toyphy\right) \left[ X(\alpha) \right]
	\nonumber\\
	&\urel{\eqref{eq:delta-lin-DSE}} \toylog \left( X(\alpha) \right)+\toylog \left( X(\alpha) \right) (2\alpha\partial_{\alpha}-1)G(\alpha,x).
\end{align}
To relate this to the notation common in physics\footnote{Compare \eqref{eq:rge-physical} and \eqref{eq:rge-solved} with (7.3.15) and (7.3.21) in \cite{Collins}!}, define \mbox{$\widetilde{G}(\alpha,\ln\frac{s}{\rp}) \defas G(\alpha,\ln\frac{s}{\rp})-1$} as well as $\widetilde{\toylog}(\alpha)\defas \toylog\big( X(\alpha) \big)$ and introduce the \emph{running coupling} $\alpha(\rp)$ as the solution of
\begin{equation}\label{eq:running-coupling}
	\rp\frac{\dd}{\dd\rp} \alpha(\rp)
	= 2\alpha(\rp)\widetilde{\toylog}\big(\alpha(\rp)\big)
	= \beta \left( \alpha(\rp) \right)
\end{equation}
for the \emph{$\beta$-function} $\beta(\alpha) \defas 2\alpha\widetilde{\toylog}\big(\alpha\big)$. In these terms, \eqref{eq:rge2} boils down to
\begin{equation}\label{eq:rge-physical}
	\rp\frac{\dd}{\dd\rp} \widetilde{G} \left( \alpha(\rp),\ln\tfrac{s}{\rp} \right)
	= \widetilde{\toylog}\big( \alpha(\rp) \big) \widetilde{G} \left( \alpha(\rp),\ln\tfrac{s}{\rp} \right).
\end{equation}
After integration, this tells us explicitly that the correlation functions for different renormalization points $\rp$ differ merely by an overall factor, as long as one chooses the values of $\alpha$ as determined by the running coupling:
\begin{equation}\label{eq:rge-solved}
	\widetilde{G}\left( \alpha(\rp_2),\ln\tfrac{s}{\rp_2} \right)
	= \widetilde{G}\left( \alpha(\rp_1),\ln\tfrac{s}{\rp_1} \right)
	\cdot \exp \left[ \int_{\rp_1}^{\rp_2} \widetilde{\toylog}\big( \alpha(\rp) \big) \tfrac{\dd\rp}{\rp} \right]
	\urel{\eqref{eq:running-coupling}} \widetilde{G}\left( \alpha(\rp_1),\ln\tfrac{s}{\rp_1} \right) \cdot \sqrt{\tfrac{\alpha(\rp_2)}{\alpha(\rp_1)}}.
\end{equation}
In particular, $\widetilde{G}\left(\alpha(\rp),\ln\frac{s}{\rp}\right) / \sqrt{\alpha(\rp)}$ is independent of $\rp$. Setting $\rp=s$ yields
\begin{equation}
	\widetilde{G}\left( \alpha(\rp),\ln\tfrac{s}{\rp} \right)
	= -\sqrt{\frac{\alpha(\rp)}{\alpha(s)}}.
\end{equation}
So in full generality, we obtain $\widetilde{G}\left(\alpha,\ln\frac{s}{\rp}\right) = -\sqrt{\frac{\alpha}{\alpha(s)}}$ where $\alpha(s)$ is determined through
\begin{equation}
	\ln\frac{s}{\rp} = \int_{\alpha}^{\alpha(s)} \frac{\dd\alpha'}{\beta(\alpha')}.
\end{equation}
Hence apparently, $\beta(\alpha)$ determines the asymptotic behaviour of $\widetilde{G}$ and even more profoundly, whether it is defined for all $s>0$ at all!

\subsection{Non-perturbative formulations}
\label{sec:non-perturbative}
We finally exploit that the Feynman rules under consideration are defined through the universal property of $H_R$ and as such obey a specific behaviour on $B_+$. In the case of $\intrules$ from \eqref{eq:int-rules},
\begin{equation}\label{eq:DSE-intrules}
	G(\alpha,x)
	\urel{\eqref{eq:correlation}}
		\intrules \left( X(\alpha) \right)
	\urel{\eqref{eq:DSE-propagator}}
		\alpha \intrules\circ B_+ \left(\frac{\1}{\1-X(\alpha)}\right)
	\urel{\eqref{eq:int-rules}}
		\alpha \int_0^{x} \frac{\dd x'}{1-G(\alpha,x')}
\end{equation}
yields the differential equation $\partial_{z} G(z) = \frac{1}{1-G(z)}$ for $z\defas\alpha x$ which readily integrates to \eqref{eq:correlation-int-rules}, using the boundary conditions $\restrict{G}{\alpha=0} = 0$. This example shows how we can obtain the full correlation function as the solution to an integral equation \eqref{eq:DSE-intrules}, without having to consider the perturbation series at all!

In fact, we cheated in the above derivation: $\intrules\left( X(\alpha) \right) = \sum_{n=1}^{\infty}\intrules(a_n)\alpha^n$ is only a formal power series. Let $\frac{\1}{\1-X(\alpha)} = \sum_{n=0}^{\infty} f_n \alpha^n$ for forests $f_n\in\forests$, that is $a_n = B_+(f_{n-1})$, then
\settowidth{\wurelwidth}{\eqref{eq:int-rules}}
\begin{align*}
%	\alpha\intrules\circ B_+ \left(\frac{\1}{\1-X(\alpha)}\right)
	\intrules\left( X(\alpha)\right)
%	&=\sum_{n=0}^{\infty} \alpha^{n+1} \intrules \circ B_+(f_n)
	&\wurel{\eqref{eq:int-rules}} \sum_{n=0}^{\infty} \alpha^{n+1} \int_0 \intrules(f_n) 
	 \urel{\eqref{eq:a_n-recursion}} \alpha \int_0 \sum_{k=0}^{\infty} \sum_{i_1,\ldots, i_k = 1}^{\infty} \alpha^{i_1+\ldots+i_k}\intrules\left(a_{i_1}\right)\ldots\intrules\left(a_{i_k}\right) \\
	&\wurel{} \alpha \int_0 \sum_{k=0}^{\infty} {\left[ \intrules\big( X(\alpha)\big)\right]}^k
	 = \alpha \int_0 \frac{1}{1-\intrules\left( X(\alpha) \right)}
\end{align*}
is only valid if the series in $\alpha$ is actually absolutely convergent! In \eqref{eq:correlation-int-rules} we saw this to be the case, however this assumption is going to fail in {\qft} and the toy model -- the perturbation series is really only an asymptotic one. Further note that in the above, we extended the algebraic integral operator $\int_0\in\End(\K[x])$ to the analytic integration of general functions: $\frac{1}{1-\intrules\big(X(\alpha)\big)}$ is not a polynomial in $x$ anymore.

However, apparently the naive calculation \eqref{eq:DSE-intrules} suggests itself as a \emph{natural} way of formulating a non-perturbative equation! That is, we take it as the \emph{definition} of the non-perturbative theory. In the case of the toy model, it leads to
\begin{equation*}
	G\left(\alpha,x\right)
	= \alpha \unimor{\left[ -\coeff{-1}\int_0 + \toycc\right]} \circ B_+ \left(\frac{1}{1-X(\alpha)}\right)
	= \alpha \left[ -\coeff{-1}\int_0 + (\partial\toyform)\right] \left(\frac{1}{1-G(\alpha, \cdot)}\right).
\end{equation*}
Here we need to define the action of $\partial\toyform$ on general functions (non-polynomials). By
\begin{align*}
	(\partial\toyform) \left( x^n \right)
	&= \sum_{k=0}^{n-1} \binom{n}{k} \toyform \left( x^{k} \right) x^{n-k}
	= \sum_{k=0}^{n-1} (-1)^{k} c_{k} \frac{n!}{(n-k)!} x^{n-k} \\
	&= \sum_{k=0}^{n-1} (-1)^{k} c_{k} \partial_x^{k} x^n
	= P \circ \left[ \sum_{k=0}^{\infty} (-1)^k c_k \partial_x^k \right] \left( x^n \right)
\end{align*}
we can identify $\partial\toyform$ with a differential operator, while $P$ subtracts the constant terms $\propto x^0$ (coming from $k=n$) as $P(f)\defas f - \restrict{f}{x=0}$. By a differentiation we turn this integro-differential equation into the differential equation
\begin{align*}
	\partial_x G\left(\alpha,x \right)
	&\wurel{} - \alpha \left[ \coeff{-1} + \sum_{k=0}^{\infty} (-1)^{k+1}c_k \partial_x^{k+1} \right] \frac{1}{1-G(\alpha,x)} \\
	&\wurel{\eqref{eq:mellin-trafo}} - \alpha \left[\reg F(\reg)\right]_{\reg = - \partial_x} \left( \frac{1}{1-G(\alpha, x)} \right).
\end{align*}
At this point we stop and remark that we can combine this equation together with \eqref{eq:rge} to obtain a single equation for the scalar function $\widetilde{\toylog}(\alpha)$, that allows to compute the coefficients $\widetilde{\toylog}(\alpha) = \sum_{n=1}^{\infty} d_n \alpha^n$ order by order as polynomials in the Mellin transform coefficients $\coeff{k}$ of $F(\reg)$. For details and in particular the application to \emph{quantum electrodynamics}, we refer to \cite{Yeats}.

\section{Higher degrees of divergence}
\label{sec:higher-sdds}
So far we restricted ourselves to logarithmically divergent integrals only, whereas in {\qft} the divergences can acquire higher power\footnote{For example, a boson propagator is quadratically divergent in any renormalizable \qft.} behaviour.
In the realm of our toy model this situation is exemplified by
\begin{equation*}
	\int_0^{\infty} f(\zeta,s)\ \dd\zeta
	\quad\text{for}\quad
	f(\zeta,s) \defas \frac{\zeta}{\zeta + s},
\end{equation*}
which is linearly divergent as the integrand is in $\bigo{\zeta^0}$ for $\zeta\rightarrow\infty$. The subtraction at $s=\rp$ like in \eqref{eq:momsch-example} yields the integrand
\begin{equation*}
	f(\zeta,s) - f(\zeta,\rp)
	= \frac{(\rp-s) \zeta}{(\zeta+s)(\zeta+\rp)},
\end{equation*}
which lies in $\bigo{\zeta^{-1}}$ and thus still results in a divergent integral! Therefore the simple subtraction scheme $\momsch{\rp}$ from \eqref{eq:momentum-scheme} is not sufficient anymore. However, if we further subtract the linear term of the Taylor expansion of $f$ at $\rp$,
\begin{equation*}
	f(\zeta,s) - f(\zeta, \rp) - (s-\rp) \restrict{\frac{\partial}{\partial s}}{\rp} f(\zeta, s)
	= \frac{(s - \rp)^2 \zeta}{(\zeta+s)(\zeta+\rp)^2}
\end{equation*}
is in $\bigo{\zeta^{-2}}$ and hence delivers a finite integral under $\int_0^{\infty} \dd\zeta$. This method works to arbitrary orders and is the foundation of the general momentum scheme\footnote{If one subtracts already on the level of integrands themselves and omits a regulator, one calls this again the BPHZ scheme.}.
\begin{definition}
	On $\alg \defas C^{\infty} (\K^n)$, for $s\in \N_0$ the operator of Taylor expansion\footnote{For simplicity we expand around zero, however the whole argument to come clearly stays valid for arbitrary subtraction points $\rp \neq 0$!} is
	\begin{equation}
		T_{s} \in \End(\alg),\quad
		T_{s} f \defas \left( \K^n \ni x \mapsto \sum_{\abs{\beta} \leq s} \frac{x^{\beta}}{\beta!}\partial^{\beta}_0 f \right),
		\label{eq:taylor}
	\end{equation}
	using multiindices $\beta=(\beta_1,\ldots,\beta_n) \in \N_0^n$ with the common notations $\beta\leq\alpha$ iff $\beta_i \leq \alpha_i$ for all $i$, $\abs{\beta} \defas \beta_1+\ldots+\beta_n$ as well as
	\begin{equation*}
		x^{\beta} \defas \prod_{1 \leq k \leq n} x_k^{\beta_k}
		, \qquad
		\beta! \defas \prod_{1 \leq k \leq n} \beta_k!
		\qquad\text{and}\qquad
		\partial^{\beta}_0 \defas \prod_{1 \leq k \leq n} \left. \frac{\partial^{\beta_k}}{\partial x_k^{\beta_k}} \right|_{x_k = 0}.
	\end{equation*}
\end{definition}
We can now implement the general momentum scheme using these projections $T_n$, but as seen above we have to pick the correct $n$ to obtain a finite result. In {\qft}, we are given a grading $\sdd$ (called \emph{superficial degree of divergence}, see section \ref{sec:toymodel-origin}) that precisely gives the power behaviour of the divergence and prescribes the order $n$ of Taylor polynomial (in the external parameters) to subtract.

Therefore, given some graduation $H=\bigoplus_n H_n$ as an algebra, we define a Birkhoff decomposition recursively through $\phi_-(\1) \defas \1_{\alg}$ and
\begin{equation}
	\phi_- (x) \defas - T_{\abs{x}} \left[ \phi(x) + \sum_{x} \phi_-(x') \phi(x'')  \right]
	\label{eq:birkhoff-indexed}
\end{equation}
for homogeneous $x \in \ker \counit$ of degree $\abs{x}$. Note the analogue of the Rota-Baxter relation \eqref{eq:Rota-Baxter} in
\begin{satz}
	The Taylor expansion operators fulfil for any $s,t\in \N_0$ and $f,g\in\alg$
	\begin{equation}
		(T_s f) (T_t g) = T_{s+t} \left[ (T_s f) g + f (T_t g) - f g \right].
%		m_{\alg} \circ (T_s \tp T_t) = T_{s+t} \circ m_{\alg} \circ \left( T_s \tp \id + \id \tp T_t - \id \tp \id \right).
		\label{eq:taylor-baxter}
	\end{equation}
\end{satz}
\begin{proof}
	Using the Leibniz rule $\partial \circ m_{\alg} = m_{\alg} \circ \left( \partial \tp \id + \id \tp \partial \right)$ and
	\begin{equation*}
		\partial_0^{\alpha} T_s 
		= \partial_0^{\alpha} \sum_{\abs{\beta} \leq s} \frac{x \mapsto x^{\beta}}{\beta!} \partial_0^{\beta} 
		= \sum_{\abs{\beta} \leq s} \frac{\partial_0^{\alpha} (x \mapsto x^{\beta})}{\beta!} \partial_0^{\beta} 
		= \begin{cases}
				\partial_0^{\alpha} & \text{if $\abs{\alpha} \leq s$,} \\
				0 & \text{else,}
			\end{cases}
		\tag{$\ast$}
	\end{equation*}
	by \eqref{eq:taylor} it suffices to check for any multiindex $\abs{\alpha} \leq s+t$ that
	\begin{align*}
		& \partial_0^{\alpha} \left[ (T_s\, f) g + f (T_t\, g) - f g \right]
			= \sum_{\beta \leq \alpha} \binom{\alpha}{\beta} m_{\alg} \circ \left( \partial_0^{\beta} \tp \partial_0^{\alpha-\beta} \right) \left[ (T_s\, f) g + f (T_t\, g) - f g \right] \\
		& = \sum_{\beta \leq \alpha} \binom{\alpha}{\beta} \left[ \left\{ \partial_0^{\beta} T_s\, f \right\} \left\{\partial_0^{\alpha-\beta}\, g \right\} + \left\{ \partial_0^{\beta}\, f \right\} \left\{ \partial_0^{\alpha-\beta} T_t\, g \right\} - \left\{ \partial_0^{\beta}\, f \right\} \left\{ \partial_0^{\alpha-\beta}\, g \right\} \right] \\
& = \sum_{\beta \leq \alpha} \binom{\alpha}{\beta}  \left\{\partial_0^{\beta}T_s\, f\right\} \left\{\partial_0^{\alpha-\beta}T_t\, g\right\}
			= \partial_0^{\alpha} \left[ (T_s\, f) \cdot (T_t\, g) \right].
	\end{align*}
	Here we used that in the middle line, by $(\ast)$ the contributions with $\abs{\beta} > s$ or $\abs{\alpha-\beta} > t$ give zero. For example, if $\abs{\beta} > s$ note $\abs{\alpha-\beta} = \abs{\alpha} - \abs{\beta} < t$ such that
	\begin{equation*}
		\Big(\underbrace{\partial_0^{\beta} T_s}_0\, f \Big) \Big(\partial_0^{\alpha-\beta}\, g \Big) + \Big( \partial_0^{\beta}\, f \Big) \Big(\underbrace{\partial_0^{\alpha-\beta} T_t}_{\partial_0^{\alpha-\beta}}\, g\Big) - \Big( \partial_0^{\beta}\, f \Big) \Big( \partial_0^{\alpha-\beta}\, g \Big)
		= 0
		=  \Big( \underbrace{\partial_0^{\beta} T_s}_0 f\Big) \Big(\partial_0^{\alpha-\beta} T_t g \Big).
	\end{equation*}
	Hence only terms with $\abs{\beta}\leq s$ and $\abs{\alpha-\beta}\leq t$ remain, but then we get 
	\begin{equation*}
		\Big(\underbrace{\partial_0^{\beta} T_s}_{\partial_0^{\beta}}\, f \Big) \Big(\partial_0^{\alpha-\beta}\, g \Big) + \Big( \partial_0^{\beta}\, f \Big) \Big(\underbrace{\partial_0^{\alpha-\beta} T_t}_{\partial_0^{\alpha-\beta}}\, g\Big) - \Big( \partial_0^{\beta}\, f \Big) \Big( \partial_0^{\alpha-\beta}\, g \Big)
		= \underbrace{\Big( \partial_0^{\beta}\, f \Big)}_{\partial_0^{\beta} T_s\,f} \underbrace{\Big( \partial_0^{\alpha-\beta}\, g \Big)}_{\partial_0^{\alpha-\beta} T_t\, g}.\qedhere
	\end{equation*}
\end{proof}
The above relations imply that also the generalized momentum scheme defined by \eqref{eq:birkhoff-indexed} respects characters in
\begin{satz}
	Let $H$ be a connected bialgebra, graded as an algebra and $\phi \in \chars{H}{\alg}$ an algebra morphism to some commutative algebra $\alg$.
	Further let \mbox{$T_{\cdot}\!: \N_0 \rightarrow \End(\alg)$} be an \emph{indexed renormalization scheme}, that is a family of endomorphisms such that
	\begin{equation}
%		\forall n,m \in \N_0:\ 
		m_{\alg} \circ ( T_n \tp T_m ) = T_{n+m} \circ m_{\alg} \circ \left[ T_n \tp \id + \id \tp T_m - \id \tp \id \right]
		\label{eq:indexed-scheme}
	\end{equation}
	for all $n,m\in\N_0$. Then the counterterms $\phi_-$ defined by \eqref{eq:birkhoff-indexed} (and thus $\phi_+ \defas \phi_- \convolution \phi$ as well) are algebra morphisms.
\end{satz}
\begin{proof}
	The proof is the same as for \eqref{eq:birkhoff-dec}, we only replace \eqref{eq:Rota-Baxter} by \eqref{eq:taylor-baxter}: For homogeneous $x,y \in \ker \counit$,
	\settowidth{\wurelwidth}{\eqref{eq:birkhoff-indexed}}
	\begin{align*}
		\phi_- (x\cdot y)
		&\wurel{\eqref{eq:birkhoff-indexed}} -T_{\abs{x\cdot y}} \left[ \phi(x\cdot y) + \sum_{x\cdot y} \phi_- \left( \{xy\} ' \right) \phi \left( \{xy\}'' \right) \right] \\
		&\wurel{} T_{\abs{x}+\abs{y}} \left[ \left\{ T_{\abs{x}} \bar{\phi}(x) \right\} \bar{\phi}(y) + \bar{\phi}(x) \left\{ T_{\abs{y}} \bar{\phi}(y) \right\} - \bar{\phi}(x) \bar{\phi}(y) \right] \\
		&\wurel{\eqref{eq:indexed-scheme}} \left[ T_{\abs{x}} \bar{\phi}(x) \right] \cdot \left[ T_{\abs{y}} \bar{\phi}(y) \right] 
		\urel{\eqref{eq:birkhoff-indexed}} \phi_-(x) \cdot \phi_-(y). \qedhere
	\end{align*}
\end{proof}
Note how this embodies a vast generalization of our original definition \ref{def:birkhoff} of renormalization!

We conclude remarking that these results extend to the full algebra of all Feynman graphs (as opposed to only the superficially divergent ones), when $\sdd$ takes values in $\Z$. For that purpose, as $\sdd(\Gamma)<0$ indicates no overall divergence set $T_n = 0$ whenever $n<0$. This still ensures \eqref{eq:indexed-scheme} for all $n,m \in \Z$ (both sides vanish if $n$ or $m$ are negative) and we will thus still arrive at counterterms and renormalized Feynman rules that are algebra morphisms.

\section{Massless Yukawa theory and the toy model}
\label{sec:toymodel-origin}
So far we looked at Feynman rules from a purely algebraic point of view. It is the goal of this section to relate our results to physics, in particular through demonstrating how Kreimer's toy model actually arises out of {\qft}.
We assume basic knowledge of {\qft} and refer to the lecture notes \cite{Tong}, covering all the material we are going to need.

For simplicity\footnote{In case of ordinary Yukawa theory in four dimensions, the fermion $\psi$ is a spinor and we have to take care of form factors. These are technicalities not influencing the basic structure of the argument to come.} consider massless scalar Yukawa theory, a renormalizable {\qft} in six spacetime dimensions given by the Lagrangian density
\begin{equation*}
	\mathcal{L} = (\partial^{\mu} \conj{\psi})(\partial_{\mu} \psi) + \frac{1}{2} ( \partial_{\mu} \phi) (\partial^{\mu} \phi) - g \phi \conj{\psi}\psi.
\end{equation*}
It describes a real scalar boson $\phi$ and a complex scalar boson $\psi$, interacting through a cubic vertex of coupling $g$. In the Feynman graphs of perturbation theory, we will denote the free propagators of $\phi$ and $\psi$ particles by dashed and solid lines, respectively. The fermion analogue $\psi$ also carries a charge flow arrow along its edges. For an example of these simple Feynman rules $\Phi$, consider (working in Euclidean space, after \emph{Wick rotation})
\begin{equation}
	\Phi \left( \Graph{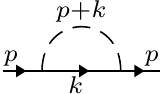} \right) 
	= \int_{\R^6} \frac{\dd[6]k}{(2\pi)^6} \frac{1}{k^2} \frac{1}{(k+p)^2}.
	\label{eq:oneloop}
\end{equation}
We do not include the powers of $g$ into $\Phi$, as those will be restored in a later step using a Dyson-Schwinger equation. Note that in Yukawa theory, all symmetry factors are one.\footnote{All incidences at a vertex are of distinct type: $\phi$ edge, incoming $\psi$ and outgoing $\psi$ edge -- thus there are no non-trivial automorphisms of the Feynman graphs.}

We focus on the correlation function $G^{\psi}$ for the $\psi$ propagator, given as asymptotic series in $g$ summing all \emph{amputated}\footnote{This means the Feynman rules do not include a propagator term for the external edges.} 1PI Feynman graphs with two external $\psi$ legs.

\subsection{Analytic regularization and the one-loop master function}
\label{sec:toymodel-dimreg}
The naive Feynman rules deliver divergent integrals like \eqref{eq:oneloop} -- hence as in section \ref{sec:regularization}, we have to regularize these to obtain well defined functions to work with.

We again employ \emph{analytic regularization}, which in this setting is defined by furnishing each loop integration $\int \dd[6]k$ with an additional factor $(k^2)^{-\reg}$. To evaluate these Feynman rules, generalize \eqref{eq:oneloop} by raising the propagators to arbitrary powers $n,m$:
%dimensional regularization, a popular regulator depending on a complex spacetime-dimension $D=6-2z \in \C$ in the vicinity of the physical dimension $D=6$ (corresponding to $z=0$). It defines an integration-like functional $f \mapsto \int_{\R^D} f(x)\ \dd[D]x$ on integrands $f$ that depend meromorphically on a finite number of Lorentz scalars (built from the external parameters). For a precise definition and examples we refer the reader to chapter 4 of \cite{Collins}.
%As a generalization of \eqref{eq:oneloop} consider arbitrary exponents $n,m$ in
\begin{equation}
	\int \frac{\dd[D]k}{(2\pi)^D}\frac{1}{{\left[ k^2 \right]}^n} \frac{1}{ {\left[ (k+p)^2 \right]}^m} 
		\safed { \left( p^2 \right)}^{D/2-n-m} \master{n}{m}.
	\label{eq:master-def}
\end{equation}
This defines the \emph{one-loop master function} $\master{n}{m}$ which evaluates to\footnote{For a derivation of \eqref{eq:master} we refer to section 2.2 in \cite{Grozin}. See also chapter 5 therein for the two-loop analogue.}
\begin{equation}
	\master{n}{m} 
\hide{
	\int \frac{d^D k}{(2\pi)^D}\frac{1}{{\left[ k^2 \right]}^n} \frac{1}{ {\left[ (k+p)^2 \right]}^m} \nonumber\\
	=&\ %{\left(p^2\right)}^{D/2 - n - m} 
		\int_0^1 dx \frac{\Gamma(n+m)}{\Gamma(n)\Gamma(m)} x^{m-1} {(1-x)}^{n-1}
		\int \frac{d^D k}{{(2\pi)}^D} {\left[ (k+xp)^2 + x(1-x)p^2 \right]}^{-n-m} \nonumber\\
		=&\ \frac{\Gamma(n+m)}{\Gamma(n)\Gamma(m)} {\left( p^2 \right)}^{D/2-n-m} \int_0^1 dx\ x^{m-1}{(1-x)}^{n-1} \nonumber\\
		& \quad \times \quad 2\frac{ {\pi}^{D/2}}{\Gamma(D/2)} \frac{1}{ {(2\pi)}^D} \int_0^{\infty} r^{D-1}\ dr {\left[ r^2 + x(1-x) \right]}^{-n-m} \nonumber\\
		=&\ \frac{\Gamma(n+m)}{\Gamma(n)\Gamma(m)}\frac{\pi^{-D/2}}{2^D \Gamma(D/2)} {\left(p^2\right)}^{D/2-n-m}
		\underbrace{\int_0^1 x^{D/2-1-n} {(1-x)}^{D/2-1-m}\ dx}_{B(D/2-n,D/2-m)} \nonumber\\
		&\quad \times \quad \underbrace{\int_0^{\infty} y^{D/2-1} { (y+1) }^{-n-m} \ dy}_{B(D/2, n+m-D/2)} \nonumber\\
}
		= { \left( 4\pi \right)}^{-D/2} \frac{\Gamma(n+m-D/2)\Gamma(D/2-n)\Gamma(D/2-m)}{\Gamma(D-n-m)\Gamma(n)\Gamma(m)}.
	\label{eq:master}
\end{equation}\hide{
where we employed the definition of dimensional regularization (see \cite{Collins}) and the \emph{Feynman parameter identity}
\begin{equation*}
	\prod_{i=1}^n A_i^{-\alpha_i}
	= \frac{\Gamma\left( \sum_{i=1}^n \alpha_i \right)}{\prod_{i=1}^n \Gamma(\alpha_i)}
	\int_{0}^{\infty} dx_1\ x_1^{\alpha_1-1} \cdots \int_{0}^{\infty} dx_n\ x_n^{\alpha_n-1} \frac{\delta\left( 1-\sum_{i=1}^n x_i \right)}{ {\left[ \sum_{i=1}^n x_i A_i \right]}^{\sum_{i=1}^n \alpha_i}}.
\end{equation*}}
and is well known also in the context of \emph{dimensional regularization}, where $D$ may take arbitrary values in $\C$ (see chapter 4 of \cite{Collins} for a definition of dimensional regularization). However, in our case we fix $D=6$. For instance, the regulated value for \eqref{eq:oneloop} becomes
%For example, the case $n=1+\reg$ and $m=1$ delivers the regulated value for \eqref{eq:oneloop} in
\begin{align*}
	\int_{\R^6} \frac{\dd[6]k}{(2\pi)^6} \frac{(k^2)^{-\reg}}{k^2} \frac{1}{(k+p)^2}
	&= \master{1+\reg}{1} {\left(p^2\right)}^{1-\reg}
	={\left(p^2\right)}^{1-\reg} { \left( 4\pi \right)}^{-3} \frac{\Gamma(\reg-1)\Gamma(2-\reg)}{\Gamma(4-\reg)\Gamma(1+z)} \\
	%\in - \frac{p^2}{6(4\pi)^3\reg} + \K[[\reg,\reg \ln p^2]],
	&= \frac{(p^2)^{1-\reg}}{( 4\pi)^3 \reg (\reg-1)(\reg-2)(\reg-3)},
\end{align*}
where the pole at $\reg\rightarrow 0$ indicates the divergence. Note that as $D=6$, \eqref{eq:master-def} converges only for $\Re(n+m)>3$. Nevertheless we utilize the analytic continuation \eqref{eq:master} as the actual definition of the regularized integrals (dimensional regularization can be defined analogously)! We also evaluate
\begin{align}
	&\Phi \left( \Graph{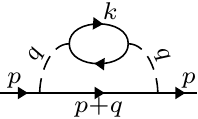} \right)
	= \int \frac{\dd[6]q}{(2\pi)^6} \frac{(q^2)^{-\reg}}{(p+q)^2} \frac{1}{ {\left[ q^2 \right]}^2}  \int \frac{\dd[6]k}{(2\pi)^6} \frac{( k^2)^{-\reg}}{k^2}\frac{1}{(k+q)^2} \nonumber\\
	&= \master{1+\reg}{1} \int \frac{\dd[6]q}{(2\pi)^6} \frac{( q^2)^{-\reg}}{(p+q)^2} \frac{1}{ {\left[ q^2 \right]}^{4+\reg-3 }} 
	= \master{1+\reg}{1}\master{1+2\reg}{1} {\left( p^2 \right)}^{1-2\reg}
	\label{eq:twoloop}
\end{align}
and more generally any graph that arises by iterated insertions of the one-loop propagator graphs into each other: simply replace one-loop subdivergences by the appropriate $\master{n}{m}$ and keep track of the overall exponent of the external momentum.

Note that the procedure given in \eqref{eq:twoloop} is ambiguous, as we might just as well allocate the regularizing factor to $(p+q)$ instead of $q$! This problem does not arise in dimensional regularization, however we can surpass it in the toy model below: we simply define the regulator $(q^2)^{-\reg}$ to use the momentum of the fermion-like line ($\psi$-propagator).
\hide{For this purpose it is useful to label each edge of a graph by a number that denotes the power the corresponding propagator is to be raised to. Of course, in the original theory propagators do not occur with exponents different from one, but with this added freedom we can employ \eqref{eq:master-def} to evaluate the graphs avoiding any integrations! For example, the generalized
\begin{align*}
	&\Phi \left( \Graph{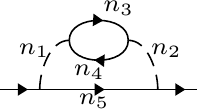} \right)
	= \int \frac{\dd[6] q}{(2\pi)^6}\frac{(q^2)^{-\reg}}{ {\left[(p+q)^2\right]}^{n_5}} \frac{1}{ {\left[ q^2 \right]}^{n_3+n_4}}  \int \frac{\dd[6] k}{(2\pi)^6} \frac{(k^2)^{-\reg}}{ {\left[k^2\right]}^{n_2}}\frac{1}{ {\left[(k+q)^2\right]}^{n_1}} \\
	&= \master{n_1}{n_2+\reg} \int \frac{\dd[6] q}{(2\pi)^6} \frac{(q^2)^{-\reg}}{ {\left[(p+q)^2\right]}^{n_5}} \frac{1}{ {\left[ q^2 \right]}^{n_1+n_2+n_3+n_4+\reg-3 }} 
	= \master{n_1}{n_2+\reg} \Phi \left( \Graph{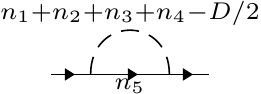}  \right) \\
	&= \master{n_1}{n_2+\reg} \master{n_5}{n_1+n_2+n_3+n_4-D/2} {\left( p^2 \right)}^{D-n_1-n_2-n_3-n_4-n_5}
\end{align*}
yields \eqref{eq:twoloop} upon setting $n_1=n_2=n_3=n_4=n_5=1$. Hence we obtain $\Phi$ by replacing a one-loop propagator subdivergence as \eqref{eq:oneloop} with an edge of corresponding weight upon multiplication by $\master{n}{m}$ for appropriate $n$ and $m$, iterating this process until we reduce the graph to $\1$.}

\subsection{The toy model of iterated insertions}
\label{sec:iterated-insertions}
Now consider the Hopf algebra $H_{\mathcal{L}}$ generated by all superficially divergent 1PI Feynman graphs of the scalar Yukawa theory $\mathcal{L}$. For an account of this kind of Hopf algebras we refer to \cite{CK:RH1} and \cite{Yeats}. The \emph{toy model} is defined as the Hopf subalgebra $\Htoy$ generated by the graphs obtained from iterated insertions of $\Graph{+-}$ into itself.% The operation of insertion itself is specified in
\begin{definition}
	The \emph{insertion operator} $\toycocycle \in \End(H_{\mathcal{L}})$ maps a graph $\gamma$ of $\Htoy$ to
	\begin{equation}
		\toycocycle \left( \gamma \right)
		\defas \frac{1}{\abs{\pi_0(\gamma)}!} \sum_{\sigma \in S_{\abs{\pi_0(\gamma)}}} \raisebox{-2mm}{\includegraphics{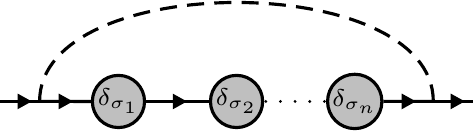}},
		\label{eq:toycocycle}
	\end{equation}
	where $\pi_0(\gamma) = \setexp{\delta_i}{1 \leq i \leq n}$ with $n\defas \abs{\pi_0(\gamma)}$ denotes the multiset of the connected components of $\gamma$ and $\sigma\in S_n$ runs over all permutations of these.
\end{definition}
It is easy to unmask $\toycocycle$ as a cocycle of $\Htoy$ using the coproduct -- however note that this is not the case in the full Hopf algebra $H_{\mathcal{L}}$ (section 3 in \cite{Yeats} gives details on cocycles and insertion operators in general). Thus we can define a morphism
\begin{equation}
	\toymor \defas \unimor{\left[\toycocycle\right]}\!:\ H_R \rightarrow \Htoy,
	\quad
	\toymor \circ B_+ = \toycocycle \circ \toymor
	\label{eq:toymodel-unimor}
\end{equation}
of Hopf algebras using the universal property \eqref{eq:H_R-universal}. As examples observe
\begin{align*}
	\toymor \left( \tree{+-} \right)
	&=\toycocycle \left( \1 \right)
	 = \Graph{+-}
	\qquad\qquad\qquad
	\toymor \left( \tree{++--} \right)
	 = \toycocycle \left( \Graph{+-} \right)
	 = \Graph{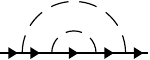} \\
	\toymor \left( \tree{++-++---} \right)
	&= \toycocycle \left( \Graph{+-} \cupdot \Graph{++--} \right)
	 = \frac{1}{2}\ \Graph{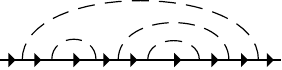} +  \frac{1}{2}\ \Graph{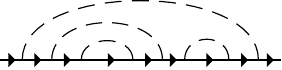}.
\end{align*}
Clearly, the order (determined by the permutation $\sigma$) of the insertions in \eqref{eq:toycocycle} is irrelevant for the \emph{value} assigned to the graph by the Feynman rules $\Phi$, as the same momentum runs through all components $\delta_i$! Hence we find the recursion
\begin{equation}
	\Phi \circ \toycocycle \left(\gamma, p^2 \right)
	= \int \frac{\dd[6]k}{(2\pi)^6} \frac{1}{(p+k)^2} \frac{(k^2)^{-\reg}}{k^2} \prod_{\delta\in\pi_0(\gamma)} \Phi(\delta, k^2) \frac{1}{k^2},
	\label{eq:insertion}
\end{equation}
where the additional factors $\frac{1}{k^2}$ account for the $\psi$-propagators between adjacent components $\delta_{\sigma_i}$ and $\delta_{\sigma_{i+1}}$ in \eqref{eq:toycocycle}.

We obtain the momentum dependence of $\Phi(\gamma)$ through power counting: as the external momentum is the only scale around, we must have $\Phi\left(\gamma,p^2\right) \propto {\left( p^2\right)}^{\sdd(\gamma)/2}$. The \emph{superficial degree of divergence} $\sdd(\gamma)$ is increased by $6-2\reg$ for each loop (yielding an integration $\int {\dd[6]}k\ (k^2)^{-\reg}$ under $\Phi$) and decreased by two for each internal edge (contributing a boson propagator). Define the \emph{loop number} $\loops{\gamma}$ of a graph $\gamma$ as the dimension of its first homology (the cardinality of a \emph{cycle basis} of the graph) and denote the number of internal edges, external edges and nodes by $\internals{\gamma}$, $\externals{\gamma}$ and $\nodes{\gamma}$, then
\begin{equation*}
	\sdd (\gamma)	= (6-2z)\loops{\gamma} - 2\internals{\gamma},
	\quad
	\externals{\gamma} + 2\internals{\gamma} = 3 \nodes{\gamma},
	\quad
	\loops{\gamma} = \internals{\gamma} + \abs{\pi_0(\gamma)} - \nodes{\gamma}
\end{equation*}
together with $\externals{\gamma} = 2\abs{\pi_0(\gamma)}$ (in $\Htoy$ each connected graph $\delta$ has $\externals{\delta} = 2$) leads to
\begin{equation}
	\sdd(\gamma)
%	= (D-6) \loops{\gamma} + 2\abs{\pi_0(\gamma)}
	= 2 \left( \abs{\pi_0(\gamma)} -\reg \loops{\gamma} \right).
	\label{eq:sdd}
\end{equation}
We now separate the trivial powers of $p^2$ in
\begin{equation}
	\widetilde{\Phi}\left( \gamma,p^2 \right)
	\defas {\left( p^2 \right)}^{-\abs{\pi_0(\gamma)}} \Phi\left( \gamma, p^2 \right),
	\label{}
\end{equation}
which still defines a morphism of algebras as $\loops{\cdot}$ is compatible with the multiplication (saying that $\loops{\gamma \delta} = \loops{\gamma \cupdot \delta} = \loops{\gamma} + \loops{\delta}$). Finally we apply \eqref{eq:sdd} to \eqref{eq:insertion} in\hide{
\begin{align*}
	\widetilde{\Phi} \left( \toycocycle (\gamma), p^2 \right)
	&= \frac{1}{p^2} \int \frac{\dd[D]k}{(2\pi)^D} \frac{1}{(p+k)^2} {\left[\frac{1}{k^2}\right]}^{1+\abs{\pi_0(\gamma)}} \Phi\left( \gamma, k^2 \right) \\
	&= \frac{1}{p^2} \int \frac{\dd[D]k}{(2\pi)^D} \frac{1}{(p+k)^2} {\left[\frac{1}{k^2}\right]}^{1+\abs{\pi_0(\gamma)}-\sdd(\gamma)/2} \underbrace{{\left( k^2 \right)}^{-\sdd(\gamma)/2}\Phi\left( \gamma,k^2 \right)}_{\text{independent of $k^2$}} \\
	&= \frac{1}{p^2} {\left( p^2 \right)}^{-\sdd(\gamma)/2} \Phi\left( \gamma,p^2 \right) \int \frac{\dd[D]k}{(2\pi)^D} \frac{1}{(p+k)^2} {\left[ \frac{1}{k^2} \right]}^{1+\reg\loops{\gamma}} \\
	&= {\left( p^2 \right)}^{-1-\abs{\pi_0(\gamma)} + \reg\loops{\gamma}+3-\reg-1-1-\reg\loops{\gamma}} \Phi\left( \gamma,p^2 \right) \master{1}{1+\reg\loops{\gamma}} \\
	&= \widetilde{\Phi}\left(\gamma, p^2 \right) {\left( p^2 \right)}^{-\reg} \master{1}{1+\reg \loops{\gamma}}.
\end{align*}
}\begin{align*}
	\widetilde{\Phi} \left( \toycocycle (\gamma), p^2 \right)
	&= \frac{1}{p^2} \int \frac{\dd[6]k}{(2\pi)^6} \frac{1}{(p+k)^2} {\left[\frac{1}{k^2}\right]}^{1+\reg+\abs{\pi_0(\gamma)}} \Phi\left( \gamma, k^2 \right) \\
	&= \frac{1}{p^2} \int \frac{\dd[6]k}{(2\pi)^6} \frac{1}{(p+k)^2} {\left[\frac{1}{k^2}\right]}^{1+\reg+\abs{\pi_0(\gamma)}-\sdd(\gamma)/2} \underbrace{{\left( k^2 \right)}^{-\sdd(\gamma)/2}\Phi\left( \gamma,k^2 \right)}_{\text{independent of $k^2$}} \\
	&= \frac{1}{p^2} {\left( p^2 \right)}^{-\sdd(\gamma)/2} \Phi\left( \gamma,p^2 \right) \int \frac{\dd[6]k}{(2\pi)^6} \frac{1}{(p+k)^2} {\left[ \frac{1}{k^2} \right]}^{1+\reg+\reg\loops{\gamma}} \\
	&= {\left( p^2 \right)}^{-1-\abs{\pi_0(\gamma)} + \reg\loops{\gamma}+3-1-1-\reg-\reg\loops{\gamma}} \Phi\left( \gamma,p^2 \right) \master{1}{1+\reg\left[1+\loops{\gamma}\right]} \\
	&= \widetilde{\Phi}\left(\gamma, p^2 \right) {\left( p^2 \right)}^{-\reg} \master{1}{1+\reg \left[1+\loops{\gamma}\right]}.
\end{align*}
Identifying the external parameter $s\defas p^2$ and the function $F(\reg)\defas \master{1}{1+\reg}$, this coincides with \eqref{eq:toymodel-mellin} through
\begin{equation}
	\widetilde{\Phi} \circ \toymor \circ B_+ (f)
	= \widetilde{\Phi} \circ \toymor (f) s^{-\reg} F(\abs{B_+(f)}\reg),
	\label{eq:toymodel-graphs}
\end{equation}
where we used that $\toymor$ respects the graduations: a forest $f\in\forests$ of weight $\abs{f}$ is mapped to a linear combination of $\abs{f}$-loop graphs (each node of $f$ corresponds to an application of $\toycocycle$, which adds another loop). So instead of considering $\Htoy$, we can equivalently study $H_R$ with the Feynman rules defined by $\widetilde{\Phi}_{p^2} \circ \toymor \safed \toy_s$.

Hence we realize how for a special choice of $F$, we obtain the toy model of section \ref{sec:toymodel} as the restriction of a {\qft} to graphs obtained by iterated insertions into a single primitive divergence! In particular, therefore all the results derived earlier for the abstract toy model do apply here. For more information on the four dimensional Yukawa toy model and in particular a non-perturbative result, study \cite{Kreimer:ExactDSE}.

Finally we remark that the coefficients $\dsecoeff(t)$ in the solution \eqref{eq:DSE-solution} of the Dyson-Schwinger equation \eqref{eq:DSE-propagator} cancel with the factors $\frac{1}{\abs{\pi_0(\gamma)}}$ in \eqref{eq:toycocycle}, such that $X(\alpha)$ in \eqref{eq:DSE-propagator} indeed sums over all graphs of the toy model with a coefficient of unity. As a simple example, check
\begin{align*}
	\toymor(a_4) 
	&= \toymor\left( \tree{++++----} \right) + \toymor\left(\tree{+++-+---}\right) + 2\toymor\left(\tree{++-++---}\right) + \toymor\left(\tree{++-+-+--}\right) \\
	&= \Graph{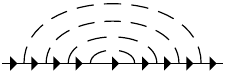} + \Graph{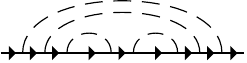} \\
	& \qquad + \Graph{++-++---} + \Graph{+++--+--} + \Graph{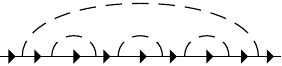}.
\end{align*}
Note how each graph contributing to $\toymor(a_n)$ corresponds uniquely to an ordered rooted tree (whose ordering specifies the permutations $\sigma$ in \eqref{eq:toycocycle} at each node), which is why proposition \ref{satz:DSE-solution} proves that $\toymor(a_n)$ indeed is the sum over all graphs of the toy model with $n$ loops and a coefficient of one!

\chapter{Conclusion}
\label{chap:conclusion}

While investigating the toy model, we saw how the Hopf algebra not only allows for a precise definition of the renormalization process, but is also the key for clean inductive proofs of properties like finiteness. Most importantly, we learned in corollary \eqref{satz:toyphy-hopfmor} that the physical limit of the renormalized Feynman rules results in a morphism \mbox{$\toyphy\!: H_R \rightarrow \K[x]$} of Hopf algebras.

Section \ref{sec:higher-orders} revealed that it is precisely this compatibility with the coproduct that allows for the reduction of $\toyphy$ to its linear terms $\toylog$ in \eqref{eq:toymodel-higher-orders}. This is a tremendous achievement, eliminating the whole dependence on the external parameter ($\toylog$ is just a functional on $H_R$)!
Further, upon application to the correlation functions, these relations take the well-known form of the renormalization group equations of physics, as we proved in section \ref{sec:rge}.

All these results were derived in the momentum scheme (subtraction $\momsch{\rp}$ at a renormalization point $s\mapsto\rp$). We remark that the popular minimal subtraction scheme does behave much worse from an algebraic viewpoint. In particular, by \eqref{toyZ-MS:()} it does not\footnote{Otherwise it would have to map $\ker\counit$ into polynomials without a constant term!} yield a Hopf algebra morphism in the physical limit! Also it is not possible to calculate these values recursively as easily as in the case of the momentum scheme.

For further study it will be interesting to compare the above observations with the similar relations occurring among counterterms in the minimal subtraction scheme as reported in \cite{CK:RH2}, or the differential equation for the physicial limit in this scheme derived in \cite{Brouder}. Considering the correlation between \eqref{eq:toymodel-leading-log} and \eqref{eq:toymodel-leading-pole} it seems worth investigating the general relation connecting the counterterms and the renormalized results.

Finally, the equation given at the end of section \ref{sec:non-perturbative} remains to be analyzed in general. Its understanding would give insight into how the apparently most relevant Mellin transform \eqref{eq:mellin-trafo} determines the correlation functions, non-perturbatively.

\vspace{0.8cm}
I wish to thank my family for their support, Dirk Kreimer for his endless pool of exciting ideas, my colleagues Marko Berghoff, Markus Hihn and Lutz Klaczynski for illuminating discussions and great eats around and everyone else from our group for making this place so friendly and welcoming!

Furthermore I am particularly indebted to Dzmitry Doryn, Henry Ki{\ss}ler and Oliver Schnetz for their careful reading of this work, spotting and kindly pointing out to me numerous errors all over the place.

\newpage\thispagestyle{empty}
\vspace*{\fill}
\begin{center}\large\emph{Cutkosky rulez!}\end{center}
\vspace*{\fill}

\bibliographystyle{plainurl}
\bibliography{qft}

\end{document}